\documentclass[11pt,a4paper]{article}

\usepackage[latin1]{inputenc}
\usepackage[all]{xy}
\usepackage[dvips]{graphics}
\usepackage[dvips]{graphicx}
\usepackage{amsfonts}
\usepackage{amssymb}
\usepackage{amsmath}
\usepackage{amstext}
\usepackage{amsbsy}
\usepackage{amsopn}
\usepackage{amsthm}
\usepackage{epsfig,rotating}
\usepackage{amscd}
\usepackage{amsxtra}
\usepackage{mathrsfs}
\usepackage{upref}
\usepackage[title]{appendix}
\usepackage[english]{babel}
\usepackage{caption}
\usepackage{cleveref}

\let\small=\footnotesize
\def\longpage#1{\newdimen\addstr  \addstr=#1\baselineskip
  \advance\topmargin-2\baselineskip \advance\textheight4\baselineskip
  \advance\topmargin-0.5\addstr     \advance\textheight\addstr
  \advance\oddsidemargin-0.5\addstr \advance\textwidth\addstr}
\longpage{0.44}

\def\R{\mathbb R}              \def\ra{\rightarrow} 
              \def\F{\Phi}      
\def\b{\beta}                         
          \def\:{\colon\,}        
\def\D{\partial}               
                 \def\g{\gamma} 
\def\e{\varepsilon}             
                 
\def\sph{\mathbb S}

\def\a{\alpha}
                 
\def\fin{\hfill{$\square$}}    
\def\d{\delta}                  
               \def\E{{\mathcal E}}
               \def\A{{\mathcal A}}
\def\V{{\mathcal V}}  
          
              \def\RP{\mathbb R\mathrm P}
\def\tilde{\widetilde}         \def\hat{\widehat} 

\author{Ricardo\,Uribe-Vargas \\ \footnotesize{Institut de Mathématiques de Bourgogne, UMR 5584, CNRS \& Univ. de Bourgogne and} \\
\footnotesize{Lab. Solomon Lefschetz UMI 2001, CNRS \& Universidad Nacional Autonoma de M\'exico} \\ 
\footnotesize{r.uribe-vargas@u-bourgogne.fr}}

\date\empty                     
\title{\vspace{-2.5cm} {\footnotesize  To appear in \textsc{Journal of Dynamical and Control Systems}
2024.  Its old title was} 

{\footnotesize \textit{Surface Evolution Implicit Differential Equations and Pairs of Legendrian Fibrations}} \\ \ \\
 Evolving Surfaces and Evolving Implicit Differential 
Equations Via Contact Geometry and Singularities}

\begin{document}

\captionsetup[figure]{labelfont={bf},labelformat={default},labelsep=period,name={\small Fig.}}

\numberwithin{equation}{section}                
\theoremstyle{plain}
\newtheorem{theorem}{\bf Theorem}
\numberwithin{theorem}{section}
\newtheorem*{Theorem}{\bf Theorem}
\newtheorem{teorema}{\bf Theorem}
\def\theteorema{\Roman{teorema}}
\newtheorem*{Euler}{\bf Euler's Criterion}
\newtheorem{lemma}{\bf Lemma}
\numberwithin{lemma}{section}
\newtheorem*{mlemma}{\bf Main Lemma}
\newtheorem*{lemma0}{\bf Lemma 0}
\newtheorem*{lemmaa}{\bf Lemma \small A}
\newtheorem*{lemmab}{\bf Lemma \small B}
\newtheorem*{lemmac}{\bf Lemma \small C}
\newtheorem*{lemmad}{\bf Lemma \small D}
\newtheorem*{lemmae}{\bf Lemma \small E}
\newtheorem*{lemmaf}{\bf Lemma \small F}
\newtheorem*{lemmag}{\bf Lemma \small G}
\newtheorem*{fitheorem}{\bf Inflection-Contour Theorem}
\newtheorem*{nc-theorem}{\bf No Cusp Theorem}

\newtheorem{proposition}{\bf Proposition}
\numberwithin{proposition}{section}
\newtheorem*{proposition*}{\bf Proposition}
\newtheorem{corollary}{\bf Corollary}
\numberwithin{corollary}{section}
\newtheorem*{corollary*}{\bf Corollary}
\newtheorem*{corollary3}{\bf Corollary {\rm (of Theorem \ref{General})}}
\theoremstyle{definition}
\newtheorem{definition}{\bf Definition}[section]
\newtheorem*{definition*}{\bf Definition}
\newtheorem*{conjecture}{\bf Conjecture}

\newtheorem{example}{\bf Example}
\numberwithin{example}{section}
\newtheorem*{example*}{\bf Example}
\newtheorem{Example}{\bf Example}[section]
\newtheorem*{exampleE}{\bf Example (The Euler Group)}
\theoremstyle{remark}
\newtheorem{note}{\sc Note}
\newtheorem*{remark*}{\sc Remark}
\newtheorem{remark}{\sc Remark}
\newtheorem*{Aremark}{\bf Affine Remark}
\newtheorem*{iremark}{\bf Inflection-Remark}
\newtheorem*{cremark}{\bf Natural Orientation of the Parabolic Curve}
\newtheorem*{aremark}{\bf Remark on Asymptotic Curves}
\newtheorem*{inf-leg-remark}{\bf Remark on Inflections and Legendrian Fibrations}
\newtheorem*{wremark}{\bf Whitney Pleat Remark}

\maketitle

\noindent
{\footnotesize \bf Abstract. \rm 
We present the list of unavoidable local phenomena (transitions) occurring on 
the configuration of the flecnodal and parabolic curves of evolving smooth surfaces 
in $\R^3$ (or $\RP^3$). We also present the list of transitions occurring on the curve of 
inflections of the solutions of evolving implicit differential equations (IDE). 
Our results are based on the properties of the contours of surfaces (in a contact $3$-space)
for projections all whose fibres are Legendrian. 
\smallskip

\noindent
{\em Keywords}: Surface, flecnodal curve, 
contact geometry, implicit differential equations. 
\smallskip

\noindent
MSC2010: 14B05, 32S25, 58K35, 58K60, 53A20, 53A15, 53A05, 53D99, 70G45.
}

\section*{Introduction and Main Results}
We investigate two subjects: the geometry of solutions of implicit differential equations (IDE) in one variable 
and the local transitions of some robust properties of surfaces in $3$-space. 
The main result on IDEs is a 
classification of generic $1$-parameter local transitions of the discriminant (see \S\ref{subsect:bifurcations of BIDE}) 
and the \textit{curve of inflections} (formed by the inflection points of the solution curves). 
The main result on smooth surfaces is a classification of the local transitions of the flecnodal 
and parabolic curves (defined in \S\ref{subsect:bifurcations of EvSurf}) occurring in generic $1$-parameter 
families of smooth surfaces in $\R^3$ (or $\RP^3$). 

Our study of surfaces in $3$-space is done in terms of IDEs, and in both subjects, the investigation is done in the setting of 
contact geometry and the Legendre singularity theory initiated by V.\,I.\,Arnold. 

We assume that all manifolds and maps are \textit{smooth}, which means ``continuously differentiable the 
necessary number of times'', for example $C^\infty$.

\subsection{Transitions in Evolving Surfaces}\label{subsect:bifurcations of EvSurf}
Along the paper, a \textit{smooth evolving surface} $S_\e$ means a $1$-parameter family of smooth surfaces 
$S\times\R\longrightarrow\RP^3$ (or $\R^3$). 

A generic smooth surface in $\R^3$ can have an open {\em hyperbolic domain} $H$ at which 
the Gaussian curvature $K$ is negative,
an open {\em elliptic domain} $E$ where 
$K$ is positive
and a {\em parabolic curve} $P$ where  
$K=0$. 
The {\em flecnodal curve} $F$ is formed by the \textit{inflections} of the asymptotic curves
(points at which the first two derivatives of the curve are collinear). 
If the surface depends on one parameter (say, the time), the configuration
formed by the flecnodal and parabolic curves may change. 

In the middle of the 1980's, while working in computer vision problems, 
D.\,Mumford communicated to V.\,Arnold the following
\smallskip

\noindent
\textit{\textbf{Mumford's problem}}:
{\em Find all local transitions of the flecnodal and parabolic curves occurring in generic $1$-parameter 
families of surfaces.}  
\smallskip

It turns out that there are four types of points of generic surfaces 
which are involved in such transitions. We have to mention them:

We distinguish two branches of the flecnodal curve, called \textit{left} and \textit{right}, 
according to the orientation of the framing (see $\S\S$\,\ref{subsect:relevant-properties} for the definition).

A \textit{hyperbonode} is a point of transverse intersection of the left and right branches of the 
flecnodal curve, and an \textit{ellipnode} is a real intersection point of complex conjugate flecnodal curves. 
A {\em biflecnode}, noted $b$ in Fig.\,\ref{points}, is an isolated point of the flecnodal curve at which an asymptotic curve
has a \textit{bi-inflection} (the first $3$ derivatives are collinear - see \S\,\ref{Tangential-Classification}, 
p.\,\pageref{inflections_bi-inflections}).
\smallskip


A {\em godron} is a parabolic point at which the (unique) asymptotic direction is tangent
to the parabolic curve. 
The following theorems have already been well known:
\medskip

\noindent
{\bf Theorem} \cite{Salmon, Kortewegpp, Platonova, Landis}. {\em At a godron of a generic 
smooth surface the flecnodal curve is 
tangent to the parabolic curve.} 
\medskip

\noindent
{\bf Theorem 0} \cite{Uribegodron}. 
{\em Godrons locally separate the left and right branches of the flecnodal curve}. 
\medskip

In Fig.\,\ref{14perestroikas}, we represent the elliptic domain in white, the hyperbolic domain in grey 
(red for online version), the right branch $F_r$ of the flecnodal curve in black, the left 
branch $F_\ell$ in white and godrons as white points on $P$.
\smallskip

\begin{figure}[ht]
\centerline{\psfig{figure=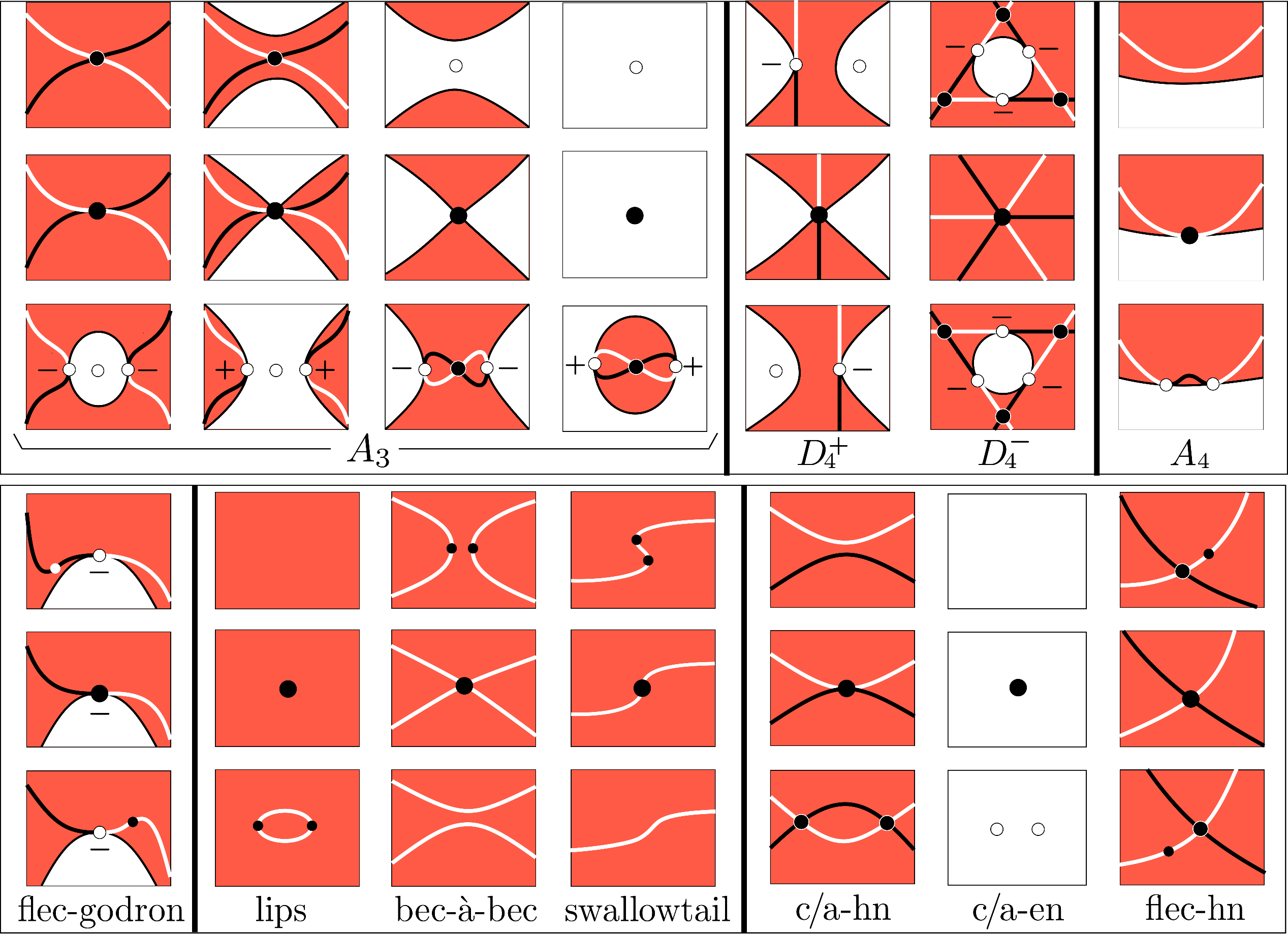,height=8.3cm}}
\caption{\small Transitions of the tangential singularities in generic evolving smooth surfaces.
The sign $+$ or $-$ is the {\em index} (defined in p.\,\pageref{index-page}) of the godrons taking part in the bifurcation.}
\label{14perestroikas}
\end{figure}

\noindent
{\bf List 1}. Our first main result solves Mumford's problem. 
It is the list, depicted in Fig.\,\ref{14perestroikas}, of local transitions of the flecnodal and parabolic 
curves (and the above four special points) that can occur on a generic evolving smooth surface.  


For example, the 4th $A_3$ transition in Fig.\,\ref{14perestroikas} means that (see Fig.\,\ref{fig:eight-figure}): 
\medskip 

\noindent
{\em The birth of a hyperbolic disc generates the birth of a flecnodal curve inscribed
in the bounding parabolic curve and having the shape of the figure eight. That transition 
occurs at an ellipnode which becomes hyperbonode after the birth.} 
\begin{figure}[ht]
\centerline{\psfig{figure=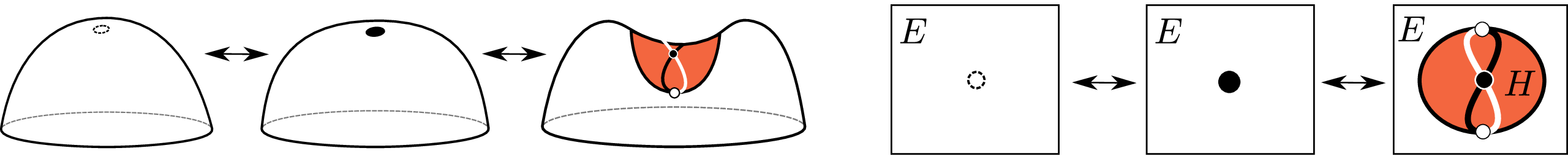,height=1.28cm}}
\caption{\small The unavoidability of the ``eight'' of $F$ at the birth of a hyperbolic disc.}
\label{fig:eight-figure}
\end{figure}

In \S\,\ref{section:bifurcation-theorems}, 
Theorem\,\ref{th:bigodron&flecgodron} corresponds to $A_4$ and flec-godron transitions; 
Theorem\,\ref{A_3-bifurcations} to $A_3$ transitions; Theorem\,\ref{theorem:lips_bec-a-bec} to lips, bec-\`a-bec and swallowtail transitions, 
Theorem\,\ref{th:ellipnodes-hyperbonodes} to creation/annihilation of hyperbonodes (or ellipnodes) and flec-hyperbonode transitions; 
Theorem\,\ref{theorem:flat-umbilics} to $D_4$ transitions. 
\medskip

\noindent
\textbf{Fact}. 
Fig.\,\ref{14perestroikas} implies that \textit{to have an $A_3$ transition of the flecnodal curve, it is necessary to have either an ellipnode 
or a hyperbonode at which the transition takes place\,: that ellipnode is replaced 
by a hyperbonode $($or the opposite$)$.}  \\
{\rm (In Fig.\,\ref{fig:eight-figure}, there was an ellipnode before the ``birth of the 8''.)} 
\smallskip

In some transitions of Fig.\,\ref{14perestroikas}, only the left flecnodal curve $F_\ell$ (in white) is 
represented; however, the right flecnodal curve $F_r$ (in black) may have the same transitions. 
The branches $F_\ell$ and $F_r$ play symmetric roles and are interchangeable in all transitions of Fig.\,\ref{14perestroikas}.

\subsubsection{Flecnodal curve on propagating wave fronts}\label{wavefronts}

%
The {\em tangential map} of a smooth surface $S$, $\tau_S:S\ra (\RP^3)^\vee$,
associates to each point of $S$ its tangent plane at that point.
The {\em dual surface $S^\vee$ of $S$} is the image of $\tau_S$. 
{\em Under $\tau_S$ the parabolic curve of $S$ corresponds to the cuspidal edge of $S^\vee$ 
and a godron to a swallowtail point} (c.f. \cite{Salmon}). 

The natural approach to the singularities 
of the tangential map is via Arnold's theory of Legendre 
singularities \cite{avg}. The image of a Legendre map is called 
its {\em front} (see \S\,\ref{geometry-IDE}). The tangential map of $S$ is a Legendre 
map; so if $S$ is in general position \textit{the only local singularities 
of its dual $S^\vee$ are those of generic fronts}\,:
\textit{cuspidal edges and swallowtails}.
Thus the transitions of the parabolic curve may be obtained from the list of 
transitions of wave fronts in $3$-space (appeared first in 
\cite{Arnoldcancon}, later in \cite{Arnoldwfeeml,avg}).

Since the tangent planes to $S$ along the flecnodal curve form the flecnodal curve of its 
dual $S^\vee$, the projective dual to a hyperbonode (ellipnode) 
is a hyperbonode (resp. ellipnode) \cite{Uribetesis}. So from the $A_3$-transitions of 
Fig.\,\ref{14perestroikas} (Th.\,\ref{A_3-bifurcations}) we get the transitions of the flecnodal 
curve on fronts (see Fig.\,\ref{A3-perestroikas}):
\smallskip

\noindent
\begin{proposition}
The $A_3$ transitions occurring in a generic propagating wave front $S_\e$ 
take place at an ellipnode or at a hyperbonode which is replaced, respectively, by a hyperbonode or 
an ellipnode {\rm (Fig.\,\ref{A3-perestroikas})}.  
In front-time $3$-space $\{S_\e\times\{\e\}\}$, the flecnodal curves of the fronts 
form a folded Whitney umbrella whose cuspidal edge is swiped by two swallowtail points, which collapse 
and disappear at the folded pinch-point {\rm (Fig.\,\ref{folded-umbrellas})}. 
\end{proposition} 

\begin{figure}[ht]
\centerline{\psfig{figure=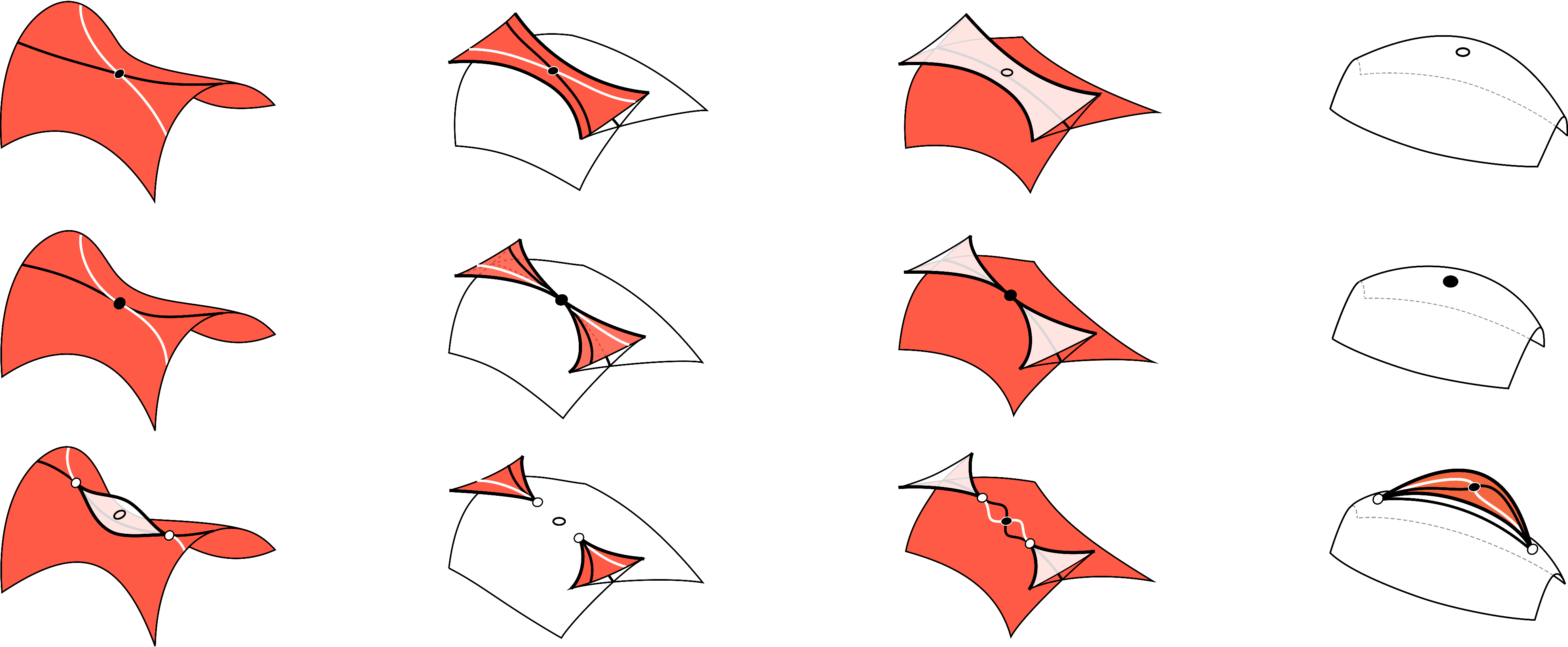,height=5cm}}
\caption{\small The four $A_3$-perestroikas of generic propagating wave fronts.}
\label{A3-perestroikas}
\end{figure}

\smallskip


\begin{figure}[ht]
\centerline{\psfig{figure=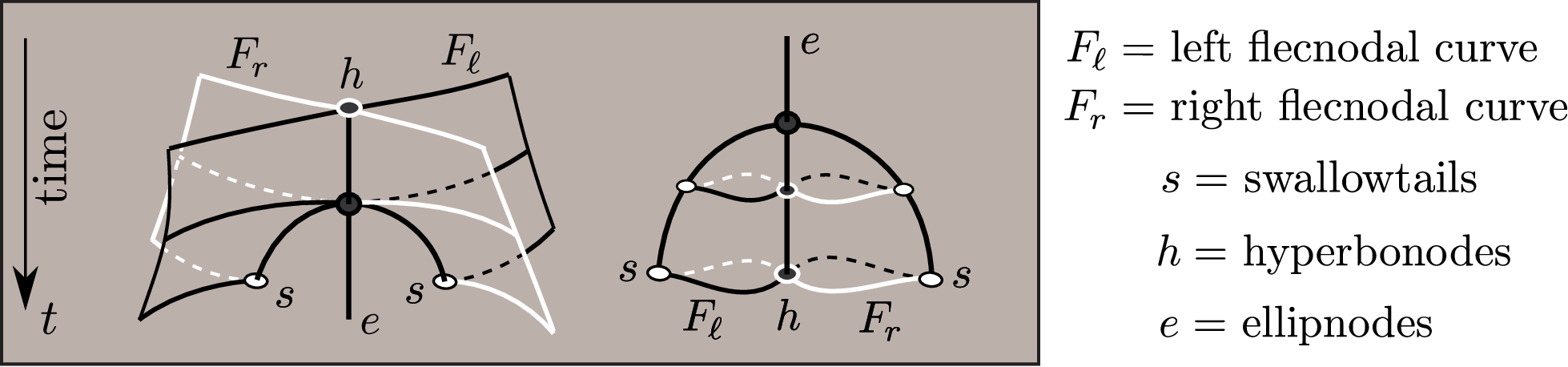,height=2.5cm}}
\caption{\small Folded umbrellas formed by the flecnodal curve of propagating wave fronts.}
\label{folded-umbrellas}
\end{figure}

\subsection{Transitions in Evolving Implicit Differential Equations}\label{subsect:bifurcations of BIDE}
Given a smooth function $F:J^1(\R,\R)\ra\R$, consider the implicit differential equation (IDE) 
\begin{equation}\label{eq:general-IDE}
  F(x,y,p)=0\,, \quad \mbox{ where } p =dx/dy\,. 
\end{equation}

In the space of $1$-jets, IDE\,\eqref{eq:general-IDE} determines a surface $\V^F$. 
The standard contact structure of $J^1(\R,\R)$ endows $\V^F$ with a foliation by Legendre curves, 
called characteristic curves of $\V^F$ (see \S\,\ref{characteristics} for the definitions).  

%

The projection of $\V^F$ on $\R^2$ in the direction of the $p$-axis
\begin{equation}\label{eq:pi:V-->R^2}
  \pi_{|_{\V^F}}:\V^F\to\R^2\,, \qquad \pi(x, y, p) = (x, y)\,,
\end{equation} 
sends the foliation of $\V^F$ to the family of \textit{solution curves} of IDE\,\eqref{eq:general-IDE}. 
The \textit{discriminant curve} $D$ of IDE\,\eqref{eq:general-IDE} is the set of critical values of $\pi_{|_{\V^F}}$. 
\smallskip

%
 

\noindent
\textit{\textbf{Inflection}}. An {\em inflection} ({\em bi-inflection}) of a 
plane curve is a point at which the curve has at least $3$-point (resp. $4$-point)
contact with its tangent line. 
\smallskip

The \textit{curve of inflections} $I$ of IDE\,\eqref{eq:general-IDE}, which consists of the 
inflections of the solution curves, 
has two branches $I_\ell$ and $I_r$, called \textit{left} and \textit{right}: 
they arise from different parts of $\V^F$ (see \S\,\ref{characteristics},
\S\ref{sect:inf-Leg_fibrations-contour} for the definitions).
A \textit{hyperbolic node} is a point of transverse intersection of the left and right
curves of inflections. 
A \textit{bi-inflection} of IDE\,\eqref{eq:general-IDE} is just a bi-inflection of a solution 
curve (left or right). 

If $\V^F$ is in general position, its characteristic foliation may have isolated 
singularities (saddle/node/focus), called \textit{characteristic} points. 
The projection of such a point in the $(x,y)$-plane is called \textit{folded singularity} 
of the IDE (\cite{Davidov}). 

We restrict ourselves to \textit{binary IDE} (BIDE): such an IDE defines two (real or complex) 
directions, counting multiplicities,  at each point in the plane 
\begin{equation}\label{eq:binary-IDE}
  F(x,y,p)=a(x,y)+2b(x,y)p+c(x,y)p^2=0. 
\end{equation}

\begin{remark*}
Equation\,\eqref{eq:binary-IDE} forbids the ``vertical'' direction ($p=\infty$). 
But we can choose any given direction as the $x$-axis ($p=0$). In the 
multilocal case, we can choose the $y$-axis to be different from the second direction. 
\end{remark*}

Now we can formulate a BIDE version of Mumford's problem: 
\medskip

\noindent 
\textit{\textbf{Mumford's problem 2}}\,:
{\em find all local transitions of the curve of inflections, the discriminant 
and the special points (hyperbolic nodes, bi-inflections, folded singularities) 
occurring in generic $1$-parameter families of BIDE.}  
\medskip

\noindent
{\bf List 2}.  Our second main result solves Mumford's problem\,2. It is the list, 
depicted in Fig.\,\ref{inflections}, of local transitions of the curve of inflections, 
the discriminant  (and the above special points) that can occur on a generic evolving
BIDE. In Fig.\,\ref{inflections}, the (real) domain of definition of the BIDE is represented 
in grey (red for online version), the right branch $I_r$ in black, the left branch $I_\ell$ 
in white, and folded singularities as white points on the discriminant curve. 

\begin{figure}[ht]
\centerline{\psfig{figure=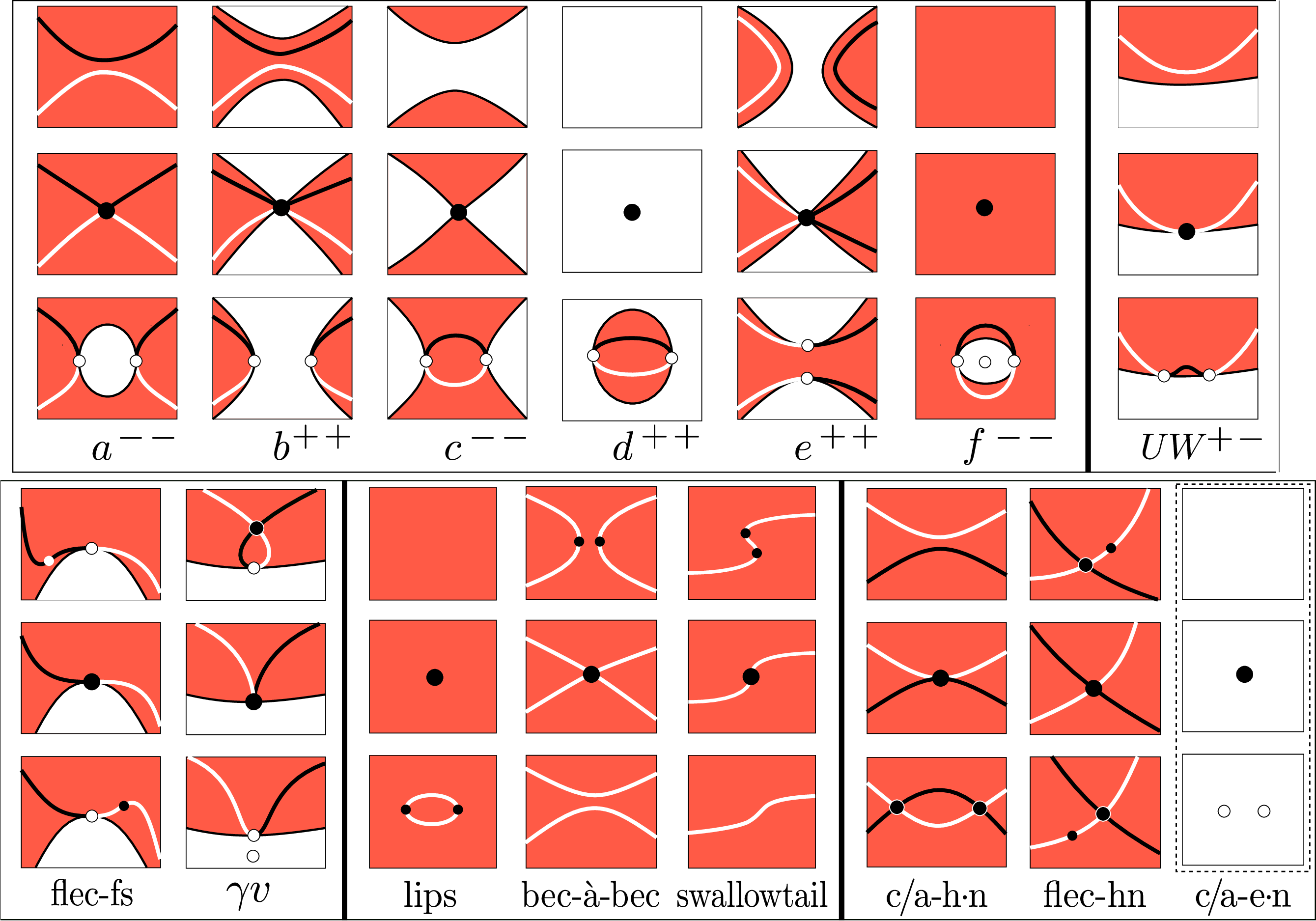,height=8.6cm}}
\caption{\small Transitions of the curve of inflections, discriminant and special
points in generic evolving BIDE. 
The signs $+$ or $-$ are the indices of the folded singularities (defined in p.\,\pageref{indices-page}).}
\label{inflections}
\end{figure}

\noindent
\textbf{\textit{\small Elliptic Nodes}}. 
In the domain where a BIDE determines a field of pairs of complex conjugate directions, 
the inflections of their solution curves form a pair of complex conjugate curves 
which intersect at real points called {\em elliptic nodes}. 
To underline that a creation/annihilation of two elliptic nodes could have no meaning 
for some evolving IDEs, we include it in a dotted box.

In \S\,\ref{section:bifurcationIDEs}, 
Theorem\,\ref{th:bifurcation-withPsmooth} corresponds to $UW$, flec-godron and $\g v$ transitions; 
Theorem\,\ref{th:6bifurcations-BIDE} to $a$, $b$, $c$, $d$, $e$, $f$ transitions; 
Theorem\,\ref{th:IDE-swallowtail-lips-bec-a-bec} to lips, bec-\`a-bec and swallowtail transitions, 
Theorem\,\ref{th:multisingularities} to creation/annihilation of hyperbolic (or elliptic) nodes and 
flec-hyperbolic node transitions.

\subsection{Equivalence of transitions and models of transitions}
A smooth $1$-parameter family of plane curves $C_\e$ determines a surface $\cup_{\e\in\R}C_\e$ in plane-time 
$3$-space, naturally foliated by its isochronal sections $C_\e$. 
\smallskip

\noindent
\textit{\textbf{\small Equivalence of Transitions}}. 
Two $1$-parameter families of curves $C_\e$, $C_\mu$ having a local transition 
at $q\in\R^2$ for $\e=\mu=0$ are said to be \textit{equivalent} if there is a 
local diffeomorphism (homeomorphism) from $\R^2\times\R$ to itself, near $(q,0)$, sending diffeomorphically 
isochronal sections to isochronal sections. 
\smallskip

In this paper, most transitions  
are equivalent to the ``isochronal'' sections $\e=\mathrm{const.}$ of the surfaces, 
or pairs of surfaces (given by equations) of Table\,\ref{tb:models-of-transisions}.

\begin{table}[h!]
\centering
\begin{tabular}{|c|c|}\hline
equation(s) & name of the transition model\\ \hline \hline
 $x^2\pm y^2-\e=0$ &  $+$ {\small Morse centre-type ; $-$ Morse saddle-type}  \\ \hline 
 $x^2\pm y^2=\,\e^2$ & {\small $+$ elliptic cone section ; $-$ hyperbolic cone section}  \\ \hline 
 $y^2=\e x^2+\a x^4$ & {\small $\a>0$ elliptic $A_3$; \, $\a<0$ hyperbolic $A_3$}  \\ \hline 
 $(y-x^2+\e)(y+x^2-\e)=0$ & {\small creation/anihilation of two crossings} \\ \hline 
 $(y+\e)(x^2\pm y^2)=0$ &  {\small $D_4^\pm$ transitions} \\ \hline
 $y=0\,;\quad x^2-y(y-\e)^2=0$ & {\small $\g v$-transition} \\ \hline 
 $y=0\,;\quad y-(x^2-\e)^2=0$  & {\small $UW$-transition} \\ \hline 
  \end{tabular}.
\caption{\small Seven surfaces whose sections $\e=\mathrm{const.}$ provide most transitions of this paper.}
\label{tb:models-of-transisions}
\end{table}

A centre or saddle transition is respectivelly called \textit{lips} or \textit{bec-\`a-bec} 
if it involves only the curve of inflections (or only the flecnodal curve). 
\smallskip

We have moreover two transition types in which the involved curves (which are tangent or have a crossing) 
undergo no structural change, but one special point (biflecnode), which moves along one of them, pass 
through their common point at the transition moment. 
Any two such transitions with tangency (or with crossing) are equivalent.

\subsection{Flecnodal Curve of Surfaces via Asymptotic IDEs}
The connection between IDEs and the flecnodal curve of a surface is as follows. 
Given a Monge form $z=f(x,y)$ of a surface $S$ in $\R^3$, the \textit{asymptotic} IDE 
\begin{equation}\label{eq:asymptote-ide}
f_{xx}(x,y)+2f_{xy}(x,y)p+f_{yy}(x,y)p^2=0\,,
\end{equation}
provides the flecnodal and parabolic curves. Namely, the projection of $S$ to the $(x,y)$-plane, 
in the $z$-direction, sends the asymptotic curves of $S$ to the solution curves of \eqref{eq:asymptote-ide}, 
sends the flecnodal curve to the curve of inflections (Prop.\,\ref{flecnodal-inflections}) and 
sends the parabolic curve to the discriminant.

The class of asymptotic IDEs is very thin in the space of all BIDEs. For this reason the local 
transitions for these two classes are discussed separately. 
They correspond respectively to List 1 and List 2 above.

\subsection{Curve of Inflections and Legendrian Fibrations}\label{flecnodal-Legendrian}

A key element, necessary to investigate the curve of inflections of IDE\,\eqref{eq:general-IDE} and its transitions, 
is the following fact, explained and proved in \S\,\ref{sect:inf-Leg_fibrations-contour} (Fig.\,\ref{fig:fold-inflection}). 
Consider the apparent contour $\hat{I}$ of the surface $\V^F\subset J^1(\R,\R)$ by the Legendrian (dual) 
projection $\pi^\vee:(x,y,p)\mapsto (X,Y)=(p, px-y)$. (I.e. $\hat{I}$ is the critical set of the map 
$\pi^\vee_{|\V^F}:\V^F\ra(\R^2)^\vee$.)  
 
\noindent
\textbf{Fact}. \textit{Projection \eqref{eq:pi:V-->R^2} sends $\hat{I}$ to the curve of inflections 
of IDE\,\eqref{eq:general-IDE}}- Fig.\ref{fig:fold-inflection}.

{\footnotesize 
\subsection{\small Some relevant references and the origin of this paper}\label{sect:history-paper}
At the end of the 19th Century, the flecnodal and parabolic curves of complex algebraic surfaces were investigated
by Cayley, Zeuthen and Salmon (\cite{Salmon}). In the same years, Korteweg investigated the transitions 
of the parabolic and conodal curves of evolving real surfaces - in thermodynamics, he studied the evolution of 
the graph of the energy (as a function of the volume and the entropy) \cite{Kortewegpp,Korteweggtp}; so Korteweg was a founder of 
catastrophe theory. These Korteweg works were forgotten (or unknown) by the scientific community. 
(In 2004 E.\,Ghys and D.\,Serre let me know the book \cite{Levelt},
which describes Korteweg's studies.) 

In the 1980's, R.\,Thom (\cite{k-Thom}) and V.\,Arnold's school (cf. \cite{Landis, Platonova}) revived the subject 
from the view point of singularity theory; while others (cf.\,D Mumford) applied it to computer vision.  
Some later contributions are, for example, \cite{Dima, Uribegodron, Ovsienko-Tabachnikov, Uribeinvariant}. 
The classification of jets of functions of \cite{Toru} (and \cite{Toru2}) is complementary to our work; 
but to apply it to geometry of surfaces, one requires additional considerations (on the higher order terms and the moduli). 

Concerning the local normal forms for IDEs, we have to mention \cite{Cibrario, Dara, Davidov-Rosales, BT95, BT97, BFT2000, Davidov-Ishi-Izu}, and 
specially \cite{Davidov} related to godrons and \cite{Arnold-Suf-Hyperbolic-eqs} related to conic singularities of the surface $\V^F$. 

Among some pictures of computer experiments by A.\,Ortiz, the ``8'' in the transition of Fig.\,\ref{fig:eight-figure}
drew V.\,Arnold's attention, who suggested to prove (or disprove) the
{\em unavoidability of the ``8''}. Arnold considered it as a fundamental fact of geometry of surfaces (Paris, 1999).
This \textit{unavoidability problem} (solved by Theorem\,\ref{A_3-bifurcations}), the preprint of Panov's paper \cite{Dima} and Mumford's problem 
were the motivations for this research.
Many results of this paper come from Ch.\,7 of my Ph.\,D. Thesis (\cite{Uribetesis}) and were exposed in the Singularity 
Theory Semester at Newton Institute (2000) 
and in a mini-course (Int. Conf. of Real and Complex Singularities Sao Carlos, 2002), where 
the preprint \textit{Surface Evolution, Implicit Differential Equations and Pairs of 
Legendrian Fibrations} was distributed. This is a completely revamped version: shorter, clearer and includes
several new results. 

}

\medskip

\noindent
{\bf Organisation}.
\textbf{Part I} (On IDE): 
In \S\,\ref{section:bifurcationIDEs},  we present our results on evolving IDE (Ths.\,\ref{section:bifurcationIDEs}.0 
to \ref{th:multisingularities}). 
In \S\,\ref{geometry-IDE}, we explain the ideas of elementary contact geometry and singularity theory 
on which most theorems of the paper are based; there, we study IDE as surfaces in a contact $3$-manifold and 
describe the curve of inflections (of the solutions of an IDE) in terms of pairs of Legendre fibrations. 
In \S\,\ref{proofs}, we state and prove the theorems of \S\,\ref{section:bifurcationIDEs} in the setting of 
contact geometry and singularity theory (Ths.\,\ref{proofs}.0 to \ref{multi-sing}).  
\smallskip

\noindent
\textbf{Part II} (On surfaces in $3$-space): 
In \S\,\ref{classification}, we recall the classification of points of a surface (by 
the contact with its tangent lines) and some robust properties.  
In \S\,\ref{section:bifurcation-theorems}, we state Theorems\,\ref{th:bigodron&flecgodron} 
to \ref{theorem:flat-umbilics} on transitions of the flecnodal curve and give further results. 
In \S\,\ref{preparatory-results}, we explain the connection between surfaces and our contact geometry study of IDE. 
In \S\,\ref{sect:proofs-of-theorems-on-surfaces}, we prove the theorems of \S\,\ref{section:bifurcation-theorems}. 

The reader interested only in the transitions of the flecnodal and parabolic curves on evolving surfaces in $3$-space 
can directly go to \S\,\ref{classification} and \S\,\ref{section:bifurcation-theorems} 
(in Part\,II) and then read \S\,\ref{geometry-IDE}, \S\,\ref{preparatory-results} and \S\,\ref{sect:proofs-of-theorems-on-surfaces}. 
\medskip

\noindent
{\footnotesize 
{\bf Acknowledgements.} I thank D.\,Panov, M.\,Kazarian, P.\,Pushkar, F.\,Aicardi and T.\,Ohmoto
for useful comments, to D.\,Meyer for valuable remarks to the first 
version (2001), to UMI2001 CNRS Solomon Lefschetz 
UNAM M\'exico where I wrote the last version and to the two referees whose numerous 
remarks help a lot to improve the presentation.} 
\medskip

\noindent 
\emph{To the memory of V.I.\,Arnold who, in life, pressured me to publish this work.} 



\part{Transitions of Evolving BIDE}

\section{Transitions of the curve of inflections of BIDE}\label{section:bifurcationIDEs}

Given a BIDE $F(x,y,p)=0$, the critical set of the map $\pi_{|_{\V^F}}:\V^F\to\R^2$ 
(given by $F=F_p=0$) is called the \textit{criminant} of that BIDE, and noted $\widehat{D}$. 
\medskip

\noindent
{\bf Theorem\,$\mathbf{\ref{section:bifurcationIDEs}.0}$}.\label{sdpoint}
$(i)$ {\em If a folded singular point $(x,y)$ of a BIDE $F=0$ verifies the following 
genericity conditions at its characteristic point $(x,y,p)$$:$
\smallskip

\noindent
$a)$ the criminant $\widehat{D}$ of $\V^F$ is not tangent to the
$\pi^\vee$-contour $\widehat{I}$ of $\V^F$; 

\noindent
$b)$ the $\pi^\vee$-fibre is not tangent to the $\pi^\vee$-contour $\widehat{I}$ of $\V^F$;

\noindent
$c)$ the $\pi$-fibre is not tangent to the $\pi^\vee$-contour $\widehat{I}$ of $\V^F$;  
\smallskip

\noindent
then the curve of inflections $I$ is quadratically tangent to the discriminant $D$. 

\noindent
$(ii)$ That point locally separates $I$ into its left and right branches} (Fig.\,\ref{fig:left-right_branches}). 
\begin{figure}[ht]
\centerline{\psfig{figure=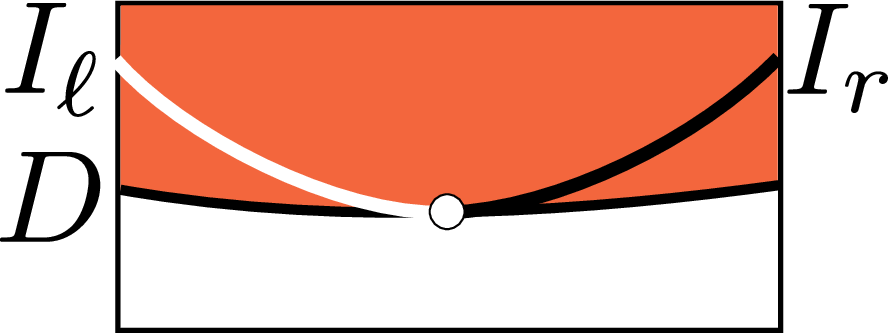,height=0.92cm}}
\caption{\small A folded singular point separates the left and right branches of $I$.}
\label{fig:left-right_branches}
\end{figure}

\noindent
{\sc Note}. In terms of $F$, the genericity conditions $a, b, c$ are given by
\medskip

\noindent 
$a) \ F_{y} \cdot \left|
\begin{array}{cc}
F_{xp}  & F_{pp} \\
F_{xx}+pF_{xy}  & F_{xp}+pF_{yp} +F_y 
\end{array}
\right|\neq 0$;  \ \ \ \ \ $b)  \ (\D_x+p\,\D_y)^2F\neq 0$; 
\bigskip

\noindent
$c) \ F_{xp}+pF_{yp}+F_y\neq 0\,. $
\medskip

Item $(i)$ of Theorem\,$\ref{section:bifurcationIDEs}.0$ appeared first in \cite{BT}; 
however, the condition $c$, $F_{xp}+pF_{yp}+F_y\neq 0$, 
which prevents the curve of inflections from having a cusp, was absent, and 
the left-right statement of $(ii)$ was not considered.

\begin{theorem}\label{th:bifurcation-withPsmooth}
Let $\V^{F_\e} \subset J^1(\R,\R)$ be a generic $1$-parameter family of smooth surfaces (of IDEs). 
If the surface $\V=\V^{F_0}$ has a characteristic point $Q$ that breaks one of the conditions a, b or c of 
Theorem\,$\ref{section:bifurcationIDEs}.0$, then the curve of inflections of $\V^{F_\e}$ 
undergoes the following respective transitions$:$

\noindent 
$a)$ UW$:$ $(\e=0)$ at the folded singularity $q=\pi(Q)$ the discriminant $D=\pi(\hat{D})$ has $4$-point 
contact with the curve of inflections $I=\pi(\,\hat{I}\,)$. It disappears $(\e < 0)$ or splits into two folded 
singular points of opposite indices $(\e > 0)$.

\noindent 
$b)$ flec-folded singularity$:$ a negative folded singular point overlaps with a biflecnode which passes from one branch 
of the curve of inflections to the other. At $\e = 0$ the curve of inflections 
$I=\pi(\,\hat{I}\,)$ has an inflection at $q=\pi(Q)$.

\noindent
$c)$ $\g v:$ the union of the curves of inflections of $\V^{F_\e}$, with $\e\in (-1, 1)$, 
is a surface of the ``plane-time'' $\R^2\times (-1, 1)$ locally diffeomorphic to the Whitney umbrella. 
For $\e=0$ the curve $\pi(\,\hat{I}\,)$ has a semi-cubic cusp at the folded singular point $q=\pi(Q)$. 
For any sufficiently small $|\e|\neq 0$ the folded singular point is generic {\rm (as in Th.\,\ref{section:bifurcationIDEs}.0)}. 
\end{theorem}

At a critical point $Q$ of a function $F:J^1(\R,\R)\to\R$ the second degree terms of the Taylor 
series of $F$ determine a quadratic cone $K\subset T_QJ^1(\R,\R)$. 
\begin{theorem}\label{th:6bifurcations-BIDE}
Let $\{F_\e(x,y,p)=0,\ \e \in\R\}$ be a generic one parameter family of IDEs. 
If $F=F_0$ has a Morse critical point in general position with respect to $\pi$ and to 
$\pi^\vee$ {\rm (see \S.\,\ref{section:cone-fibrations})}, and the $\pi$- and $\pi^\vee$-fibres 
are not conjugate diameters of the cone $K$ defined by $F$ {\rm (see \S.\,\ref{section:cone-fibrations})}, 
then both curves $D$ and $I$ undergo a Morse transition, giving one of the transitions $a$ to $f$ in Fig.\,\ref{inflections}, 
according to the relative position of $K$ with the $\pi$- and $\pi^\vee$-fibres {\rm (Fig.\,\ref{conos})}. 
Moreover, two folded singularities of equal indices are born or die. 
\end{theorem}

\begin{remark*} 
Supposing our critical point is the origin, the conditions that guarantee 
it is in general position are (see example \ref{quadratic-cone}):
\[(F_{xp}^2-F_{xx}F_{pp})(\bar{0})\neq 0,\quad F_{pp}(\bar{0})\neq 0\quad
\mathrm{and}\quad F_{xx}(\bar{0})\neq 0\,.\]
The $\pi$- and $\pi^\vee$-fibres are not conjugate diameters of $K$ iff $F_{xp}(\bar{0})\neq 0$ 
(see Lemma\,\ref{Lemma C} and Prop.\,\ref{prop:vertical-flecsurface}). 

The Morse transition types of $I_\e$ and of $D_\e$ are those   
of their respective degree two approximations (formed by conics). We get these types from the signs 
of the discriminants of their respective quadratic forms:
%
\[
-F_{xx}F_{xp}^2
\left|
\begin{array}{rrr}
F_{xx}  &  F_{xy}  & F_{xp} \\
F_{xy}  &  F_{yy}  & F_{yp} \\ 
F_{xp}  &  F_{yp}  & F_{pp}
\end{array}
\right|_{|(\bar{0})} \ \ \mathrm{and} \ \ 
-F_{pp}
\left|
\begin{array}{rrr}
F_{xx}  &  F_{xy}  & F_{xp} \\
F_{xy}  &  F_{yy}  & F_{yp} \\ 
F_{xp}  &  F_{yp}  & F_{pp}
\end{array}
\right|_{|(\bar{0})}\,. 
\]

Thus the transitions at $\e=0$ of $I_\e$ and $D_\e$ are both of center-type or both of saddle-type 
iff $F_{xx}F_{pp}>0$, and are of different types iff $F_{xx}F_{pp}<0$. 
\end{remark*}

According to \cite{Whitney3}, the fold and the Whitney pleat are the only stable singularities 
of maps from a surface to the plane - types 1 and 2 in Fig\,\ref{sing-gener-proj}.
On movement of the surface, in a generic $1$-parameter family, there are only three other 
singularities which can appear at isolated moments - they reduce locally to the projection, 
along the $x$-axis, of the surfaces $z=f(x,y)$ of types 3, 4 or 5 (\cite{Arnoldspr}), 
whose normal forms and contours are shown in Fig.\,\ref{sing-gener-proj}.
\begin{figure}[ht]
\centerline{\psfig{figure=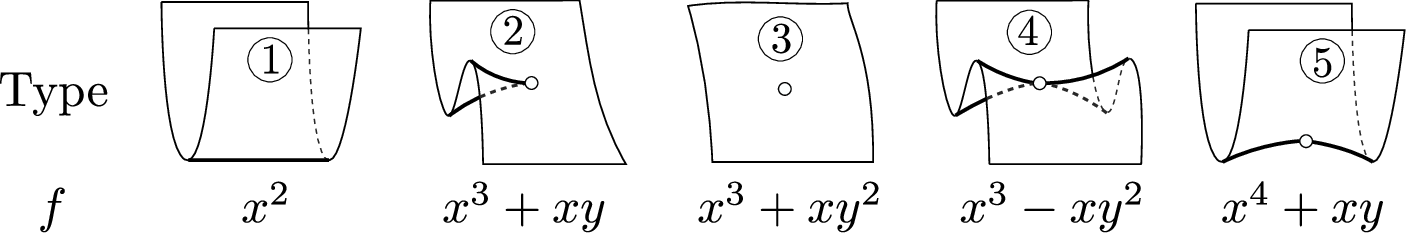,height=2cm}}
\caption{\small Generic singularities of projections in one-parameter
families of surfaces.}
\label{sing-gener-proj}
\end{figure}

\noindent
(the surfaces are projected into the $(y,z)$-plane, and the 
reduction is realised by a change of coordinates of the form
$X(x,y,z)$, $Y(y,z)$, $Z(y,z)$).
\smallskip


Let $\V^{F_\e}\subset J^1(\R,\R)$ be the surface of the BIDE $F_\e=0$. Write $\V^F:=\V^{F_0}$.  
\begin{theorem}\label{th:IDE-swallowtail-lips-bec-a-bec}
If the map $\pi^\vee_{|_{\V^F}}:\V^F \to (\R^2)^\vee$, has a singularity of type $3$, $4$ or $5$ 
of Fig.\,\ref{sing-gener-proj}, then when $\e$ passes through $0$, the curve of 
inflections of the solutions of the BIDE $F_\e=0$ undergoes, respectively, the lips-, bec-\`a-bec- 
or swallowtail-transition in Fig.\,\ref{inflections} 
{\rm (Compare with Fig.\,\ref{apparent} and Figs.\,\ref{fig:fold-inflection} and \ref{bi-inf-cusp})}
\end{theorem}

\begin{figure}[ht]
\centerline{\psfig{figure=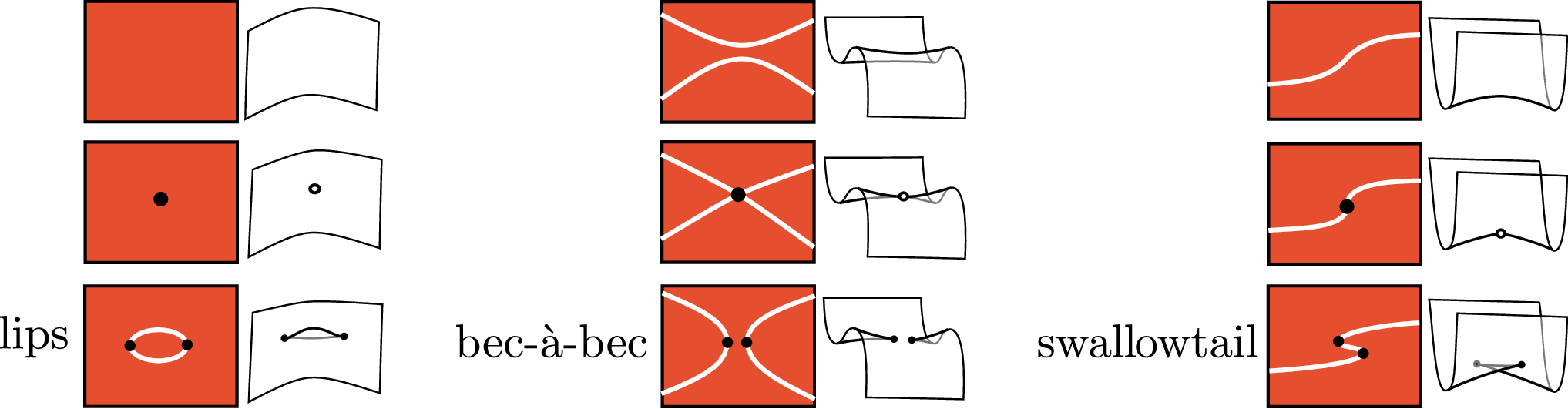,height=3cm}}
\caption{\small The transitions of the curve of inflections correspond to  
the transitions of projections of surfaces in generic one-parameter families.}
\label{apparent}
\end{figure}

\begin{theorem}\label{th:multisingularities}
Consider a generic 1-parameter family of BIDE $F_\e=0$.

\noindent
$(i)$ If for $F_0$, $I_\ell$ and $I_r$ are quadratically tangent at $q$, then at $\e=0$
we get a creation/annihilation transition of two hyperbolic nodes 
{\rm (c/a-h$\cdot$n in Fig.\,\ref{inflections})}.

\noindent 
$(ii)$ If $q$ is a hyperbolic node for $F=F_0$ and $q$ is also a left bi-inflection, then as $\e$ passes 
through $0$ a left bi-inflection that moves along $I_\ell$ crosses $I_r$ at the hyperbolic node $q$ 
{\rm (flec-h$\cdot$n in Fig.\,\ref{inflections})}. {\rm (Left and right may be interchanged.)} 

\noindent 
$(iii)$ If for $F_0$ two complex conjugate branches of $I$ are tangent at $q$, 
then at $\e=0$ we get a creation/annihilation transition of two elliptic nodes.  
\end{theorem}

Theorems\,\ref{section:bifurcationIDEs}.0 to \ref{th:multisingularities} are particular 
cases of the contact geometry Theorems \ref{proofs}.0 to \ref{multi-sing}, 
which are stated and proved in Section\,\ref{proofs}. 

\subsection{Transitions of IDE in a Riemannian plane}
Suppose a smooth Riemannian metric is given in the $xy$-plane $\R^2$. 

A {\em geodesic inflection} (resp. {\em geodesic bi-inflection}) of a smoothly 
immersed curve in the plane is a point at which the curve has $3$-point 
contact (resp. $4$-point contact) with its tangent geodesic at that point. 

The theory developed in the next section implies the 
\smallskip

\noindent
{\bf Claim}.\label{riemannian}
\textit{Theorem\,$\ref{section:bifurcationIDEs}.0$ holds for the curve of geodesic inflections.} 
\textit{The transitions of the curve of geodesic inflections and of the 
discriminant in generic $1$-parameter families of IDEs are those 
described in Th.\,\ref{th:bifurcation-withPsmooth} to Th.\,\ref{th:multisingularities}} (Fig.\ref{inflections}).

\section{IDEs as Surfaces in a Contact $3$-Space}\label{geometry-IDE}
{\footnotesize 
We give all necessary tools to understand the geometry of the curve of inflections of an IDE, 
showing the 
Legendrian nature of inflections and providing an ``abstract'' (but geometric) definition
of inflection in terms of contact geometry.} 
\medskip

\noindent
\textit{\textbf{Contact Structure}}. A {\em contact structure} on a smooth manifold is a
maximally non-integrable field of hyperplanes - called {\em contact hyperplanes}.
\medskip

{\em Contact structures may exist only on odd-dimensional manifolds} $M^{2n+1}$. 
\smallskip

A field of hyperplanes is locally defined as the field of
kernels of a differential $1$-form, say $\a$. The {\em maximal nonintegrability condition}
is $\a\wedge(d\a)^n\neq 0$.

\begin{example}
The $3$-space $J^1(\R,\R)$ has a contact structure whose 
contact planes are the kernels of the {\em contact form} $\a=pdx-dy$\,:
at $Q\in J^1(\R,\R)$, the contact plane is $\,\Pi_Q=\ker(pdx-dy)_{|Q}\subset T_QJ^1(\R,\R)$.
\end{example}

\subsection{Characteristics\,: Geometric Solutions of IDE} \label{characteristics}

{\em The geometry of IDE is part of the geometry of surfaces in contact $3$-spaces.}
\medskip

Its description needs the most important objects of contact geometry\,: 
\medskip

\noindent
\textit{\textbf{Legendre submanifolds}}. 
A {\em Legendre submanifold} of a contact manifold is an integral submanifold $L$ of the field of contact
hyperplanes of the highest possible dimension (dimension $n$ in a $2n+1$-dimensional contact manifold). 
If the contact structure is locally given by the zeros of a contact form $\a$, the restriction of $\a$
to $L$ vanishes\,: $\a_{|_L}=0$. 

\begin{example}\label{1-graph}
The {\em $1$-graph} of a smooth function $f:\R\ra \R$, which 
is the immersion in $J^1(\R,\R)$ given by $j^1f:x\mapsto (x,f(x),f'(x))$, is a Legendre
curve of $J^1(\R,\R)$. Indeed,  $\a_{(x,f(x),f'(x))}(1,f'(x),f''(x))=f'(x)\cdot 1-f'(x)=0$. 
\end{example}

\noindent 
\textit{\textbf{Legendrian Fibration}}. 
A fibration $M^{2n+1}\ra B^{n+1}$ of a contact manifold $M^{2n+1}$ is said to be {\em Legendrian} 
if all its fibres are Legendre submanifolds.
    
\begin{example}\label{standard}
The `forgetting derivatives map' $\pi:J^1(\R,\R)\rightarrow J^0(\R,\R)$ given by $(x,y,p)\mapsto (x,y)$
is called {\em standard Legendrian fibration}. Indeed, it is Legendrian because the restriction of the
contact form $\a=pdx-dy$ to the vertical lines vanishes\,:
$\,\a_{(x_0,y_0,p)}(0,0,\dot p)=p\cdot 0-0=0.$   
\end{example}

\noindent 
\textit{\textbf{Legendre map}}. A {\em Legendre map} is the projection of a Legendre submanifold of the space
of a Legendrian fibration to its base along its Legendre fibres. 
\medskip

\noindent 
\textit{\textbf{Front}}. The image of a Legendre map $L\subset M\ra B$ is called the {\em front} of $L$. 
\smallskip

So given two Legendrian fibrations of $M$, $\rho_1:M\ra B_1$ and  $\rho_2:M\ra B_2$, every
Legendre submanifold $L\subset M$ has two fronts: the front $\rho_1(L)$ in $B_1$ and the front $\rho_2(L)$ in $B_2$.

\begin{example}\label{graph-front}
Under the standard fibration $\pi:J^1(\R,\R)\rightarrow J^0(\R,\R)$, the front of the $1$-graph $j^1f$ 
of a function $f$ is the graph of $f$ (Examples\,\ref{1-graph}-\ref{standard}).
\end{example}  

{\em On any smooth surface $\V$ of a contact $3$-manifold $M$ the contact structure defines
an intrinsic field of tangent lines}\,: 
\smallskip

\noindent 
\textit{\textbf{Characteristics of a surface}}. 
At almost every point $Q$ of a smooth surface $\V\subset M$ in general position, 
the tangent plane of $\V$ is different from the contact plane. 
These planes intersect on a line tangent to $\V$ at $Q$. 
This defines the field of {\em characteristic tangent lines}
on $\V$. Its integral curves, called the {\em characteristics} of $\V$, are (by construction)
Legendre curves of $M^3$.
\smallskip

\noindent 
\textit{\textbf{Characteristic Fronts}}.
Given a Legendrian fibration $\rho:M^3\ra B$, a curve of $B$ is called
{\em characteristic front of the surface $\V\subset M$} if it is the front  
of a characteristic curve of $\V$ by $\rho:M^3\ra B$.

\begin{example*}
{\em The (geometric) solutions of a given IDE $F(x,y,p)=0$ are the characteristic fronts
of its surface $\V^F\subset J^1(\R,\R)$ by the standard Legendrian fibration $\pi:(x,y,p)\mapsto (x,y)$.}  \\
That is, the solutions are the fronts of the characteristic curves of $\V^F$. 
\end{example*}

\noindent
\textit{\textbf{Characteristic Point}}.
A {\em characteristic point} of a surface $\V\subset M^3$ is a point at which the tangent plane of $\V$
coincides with the contact plane.
\smallskip

\noindent
\textit{\textbf{Index of a Characteristic Point}}. 
For a surface $\V\subset M^3$ in general position the set of characteristic points is discrete.
At these points the field of characteristic lines is singular, having index $+1$ or $-1$ (Fig.\,\ref{folded}).
\medskip

\noindent
\textit{\textbf{Folded Singular Point}}. A {\em folded singular point} of a surface $\V\subset M$ is the image 
by a Legendrian fibration $\rho:M\to B$ of a characteristic point of $\V$. 
\smallskip

For example, a {\em folded singular point} of the IDE $\,F(x,y,p)=0$ is the
image by $\pi:(x,y,p)\mapsto (x,y)$ of a characteristic point of the surface $\V^F$ (see \S\,\ref{subsect:bifurcations of BIDE}). 

Near a folded singular point with smooth discriminant, the IDE is reduced to the normal form
$y=(p+kx)^2$ by a diffeomorphism on the $xy$-plane \cite{Davidov}. 
Depending on the `modulus' $k\in \R$, the characteristic point can be a saddle, a focus or a node. 
Their respective folded singular points, depicted 
in Fig.\,\ref{folded}, are called {\em folded focus}, 
{\em folded node} and {\em folded saddle}.
\medskip

\begin{figure}[ht]
\centerline{\psfig{figure=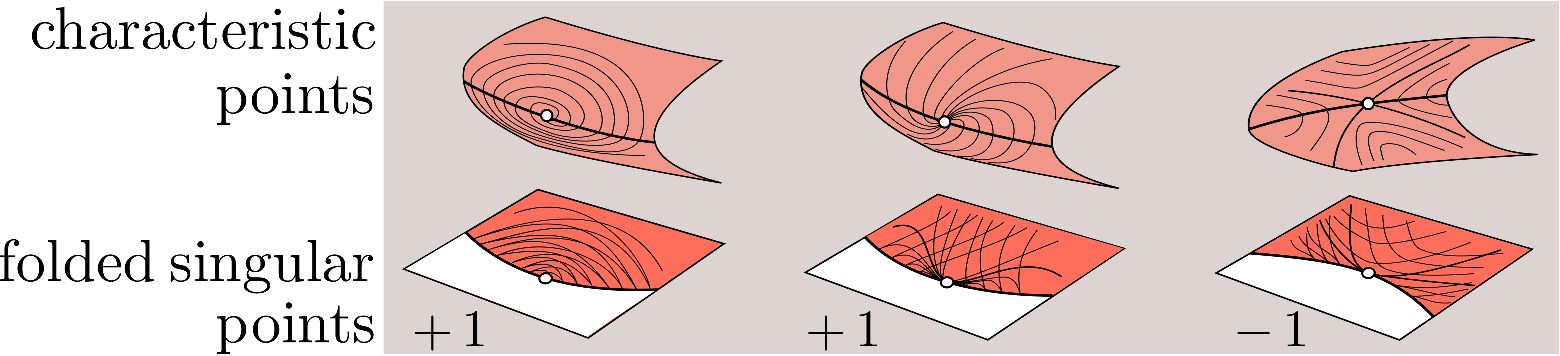,height=2.5cm}}
\caption{\small Folded focus, node and saddle of IDEs with their respective indices.}
\label{folded}
\end{figure}

\noindent
\textit{\textbf{Index of a Folded Singularity}}.  The {\em index of a folded singular point} is 
the index of its characteristic point for the characteristic field on $\V$ (Fig.\,\ref{folded}). \label{indices-page}
\medskip

\noindent
\textit{\textbf{Left and Right Characteristic Fronts}}. Let $\V\subset M^3$ be an oriented surface 
and $\rho:M^3\to B$ a Legendre fibration with $B$ oriented. 
The set of fold singularities of the projection $\rho_{|_\V}:\V\to B$ separates $\V$ 
into two parts $\V_\ell$ and $\V_r$, called \textit{left} and \textit{right parts} of $\V$, 
on which $\rho_{|_{\V}}:\V\to B$ respectively reverses or preserves the orientation. 
A characteristic of the left part $\V_\ell$ (resp. $\V_r$) 
is said to be \textit{left} (resp. \textit{right}), and its projection by $\rho$ is called 
a \textit{left}  (resp. \textit{right}) \textit{characteristic front}. 

\medskip 

\noindent
\textit{\textbf{Left and Right Solutions of IDE}}. 
The same way, orienting the surface $\V^F$ and the $(x,y)$-plane,  the critical set of the map 
$\pi_{|_{\V^F}}:(x,y,p)\mapsto(x,y)$, called \textit{criminant} of 
the IDE (and given by $F=F_p=0$), separates $\V^F$ into its {\em left} and {\em right parts}. 
A solution curve of the IDE $F = 0$ is {\em left} ({\em right}) if it is the projection of a characteristic 
of the left part $\V^F_{\ell}$ (resp. $\V^F_r$).

\subsection{Duality as Geometry of Pairs of Legendrian Fibrations} 

\begin{example}\label{legendre-dual}
The {\em standard dual fibration} $\pi^\vee :J^1(\R,\R)\rightarrow (\R^2)^\vee$  (or {\em Legendre transform})
given by $(x,y,p)\mapsto (X,Y)=(p,px-y)$ is Legendrian. Indeed, the fibre over the point $(X,Y)=(a,b)$ is
the line of equation $ax-y=b$ (on the horizontal plane $p=a$) whose parametrisation
$x\mapsto (x, ax-b, a)$ satisfies $\,\a_{(x,ax-b,a)}(\dot x,a\dot x,0)=a\dot x-a\dot x=0$.
Observe that the front of this
$\pi^\vee$-fibre by the map $\pi$ of Example\,\ref{standard} is the line $y=ax-b$ of the $xy$-plane.

Conversely, the front by $\pi^\vee$ of the $\pi$-fibre $(x_0, y_0, p)$ over $(x_0, y_0)$ is 
the line $Y=x_0X-y_0$ on the $XY$-plane $(\R^2)^\vee$, parametrised by $p\mapsto (p, px_0-y_0)$.
\end{example}

{\em Then the correspondence \,point$\,\leftrightarrow\,$line\, in projective duality
is provided by the Legendrian fibrations
$\,\pi:J^1(\R,\R)\ra\R^2\,$ and $\,\pi^\vee:J^1(\R,\R)\ra(\R^2)^\vee$}. Namely, 
{\em the respective fronts by $\pi$ and $\pi^\vee$ of a $\pi$-fibre are a point of $\R^2$
and a line of $(\R^2)^\vee$, while those
of a $\pi^\vee$-fibre are a line of $\R^2$ and a point of $(\R^2)^\vee$.} 

\begin{remark*}
\textit{The front of the $1$-graph $j^1f$ of a function $f$ by the dual fibration 
$\pi^\vee:J^1(\R,\R)\ra(\R^2)^\vee$ is 
the graph of the so-called Legendre transform of $f$.} 
\end{remark*}

\noindent
{\bf Legendre Duality and Pairs of Legendrian Fibrations}. 
Projective duality is generalised (at least locally) by replacing the contact space $J^1(\R,\R)$ and
its standard Legendrian fibrations $\pi$, $\pi^\vee$ with any contact $3$-manifold $M$ and a pair of
{\em transverse Legendrian fibrations} $\rho:M\ra B$, $\rho^\vee:M\ra B^\vee$ (i.e., such that
their fibres are nowhere tangent). This was discussed in \cite{arnoldcs}. 
\smallskip

\noindent 
In the sequel $M$ denotes a contact $3$-manifold, $\rho:M\ra B$ and $\rho^\vee:M\ra B^\vee$  
a pair of transverse Legendrian fibrations and $\V\subset M$ is a smooth surface. 
\smallskip

\noindent
For the standard fibrations of $J^1(\R,\R)$ we shall keep the notations $\pi$, $\pi^\vee$. 
\smallskip

\noindent 
\textit{\textbf{Generalised Lines}}. We say that the front by $\rho:M\ra B$ of a $\rho^\vee$-fibre is a 
{\em line of} $B$ and the front by $\rho^\vee:M\ra B^\vee$ of a $\rho$-fibre is a {\em line of} $B^\vee$. 
\smallskip

\noindent 
\textit{\textbf{Generalised Inflections}}.
An {\em inflection} ({\em bi-inflection}) of a smoothly immersed curve in $B$ is a point at which 
the curve has $3$-point contact 
(resp. $4$-point contact) with a line of $B$. 
\smallskip

In this generalised duality, which is governed by the pair of Legendrian fibrations, 
the usual correspondence ``\,inflection$\,\leftrightarrow\,$cusp\,'' holds.

\subsection{Inflections, Legendrian Fibrations and Contours} \label{sect:inf-Leg_fibrations-contour}

\begin{lemma}\label{1-graph-model}
The $1$-graphs of two smooth functions, $f,g:\R\rightarrow \R$, have 
$2$-point contact (resp. $3$-point contact) at a point 
$(x_0,y_0,p_0)\in J^1(\R,\R)$ iff their graphs have $3$-point contact 
(resp. $4$-point contact) at $(x_0,y_0)$. 
\end{lemma}

\begin{proof}
The graphs of $f$ and $g$ have $3$-point contact at $(x_0,y_0)$ if and only 
if $f(x_0)=y_0=g(x_0)$, $f'(x_0)=g'(x_0)$ and $f''(x_0)=g''(x_0)$. 
These three equalities imply and are implied by the following two equalities 
defining the tangency of the $1$-graphs of $f$ and $g$\,: 
$(x_0,f(x_0),f'(x_0))=(x_0,g(x_0),g'(x_0))$ and 
$(1,f'(x_0),f''(x_0))=(1,g'(x_0),g''(x_0))$. The proof for the case 
of $4$-point contact of the graphs of $f$ and $g$ is similar. 
\end{proof}

\begin{lemma}\label{tangent->3}
Consider a contact $3$-manifold $M$ and a Legendrian fibration $\rho:M\to B$. 
Two Legendre curves in $M$ which are not tangent to the 
$\rho$-fibre at a point $Q$ have $2$-point contact ($3$-point contact) at $Q$ 
iff their fronts by $\rho$ in $B$ have $3$-point contact (resp. $4$-point contact) at $\rho(Q)$. 
\end{lemma}

\begin{proof} 
Since {\em all Legendrian fibrations of fixed dimension are locally} 
(in the bundle space) {\em contact diffeomorphic} (see \cite{avg}), 
our Legendrian curves in $M$ are locally contact-equivalent to 
the $1$-graphs in $J^1(\R,\R)$ of two smooth functions, and their fronts by $\pi$ 
in $B$ are diffeomorphic equivalent to the pair of graphs of that functions in 
$\R^2$. Now, Lemma\,\ref{tangent->3} follows from Lemma\,\ref{1-graph-model}.  
\end{proof}

\begin{example*}[Inflection$\,\leftrightarrow\,$cusp]
Let $f:\R\ra\R$ be a smooth function. By Lemma\,\ref{tangent->3} its graph has a usual inflection
at $x_0$ iff its $1$-graph is tangent to the $\pi^\vee$-fibre of $J^1(\R,\R)$ at
$(x_0,f(x_0), f'(x_0))$, i.e. $f''(x_0)=0$. This implies that the dual curve
on the $XY$-plane $\g^\vee:x\mapsto (f'(x), f'(x)x-f(x))$ has a cusp at $x_0$.
Namely, $(\g^\vee)'(x_0)=(f''(x_0),x_0f''(x_0))=(0,0)$. 
\end{example*}

\begin{example*}
Check that the inflection of the graph of $f(x)=x^3$ corresponds to the cusp of its dual curve
$\g^\vee:x\mapsto (X(x), Y(x))=(3x^2, 2x^3)$. 
\end{example*}


\noindent
\textit{\textbf{Left and Right Curve of Inflections}}. 
The \textit{left} (\textit{right}) \textit{curve of inflections} $I_\ell$ (resp. $I_r$) 
of an IDE consists of the inflections of its left (resp. right) solution curves.
\medskip

\noindent
\textit{\textbf{Contour and Discriminant}}. 
The set of critical points of the restriction of $\rho$ to $\V$, $\,\rho_{|_\V}:\V\ra B$,\, noted $\hat{D}$,
is called {\em $\rho$-contour} of $\V$. The {\em $\rho$-discriminant} of $\V$ is the 
set $D=\rho(\hat{D})$ of critical values of $\rho_{|_\V}$. 

\begin{example*}
Given an IDE $F(x,y,p)=0$ and the standard Legendrian fibration $\pi:J^1(\R,\R)\ra \R^2$, 
the $\pi$-contour of $\V^F$ (or {\em criminant} of the IDE $F=0$) is the subset $\hat{D}$ 
of $\V^F$ at which $F_p=0$. 
If the surface $\V^F$ is in general position its $\pi$-contour $\hat{D}$ is a smooth curve (possibly empty).

We only consider IDEs whose contour by $\pi$ satisfies at every point $F_{pp}\neq 0$
(i.e. the $\pi$-fibre has exactly $2$-point contact with $\V^F$).

The $\pi$-discriminant of $\V^F$ is the usual {\em discriminant} $D$ of the IDE $F=0$. 
\end{example*}



\begin{remark*}
\textit{Given a smooth surface $\V$ of a contact $3$-manifold $M$,  
for every Legendrian fibration $\rho:M \ra B$, all characteristic points 
of $\V$ belong to its $\rho$-contour $\hat{D}$.} 
\end{remark*} 




\begin{proposition}\label{tangent-fold}
A Legendre curve $L$ is tangent to a characteristic of $\V$ at a (non characteristic) point 
of $\V$ iff $L$ is tangent to $\V$ at that point.
\end{proposition}

\begin{proof}
A Legendre curve $L$ is tangent to $\V$ at a point iff at that point the
tangent line to $L$ lies in the tangent plane to $\V$; but,
since $L$ is Legendrian, its tangent line lies also in the contact 
plane. So this tangent line is a characteristic line of $\V$.
\end{proof}

\begin{fitheorem}
The curve of inflections of the characteristic fronts of $\V$ is the image under
$\rho$ of the $\rho^\vee$-contour $\hat{I}$ of $\V$ {\rm (Fig.\,\ref{fig:fold-inflection})}.  
\end{fitheorem}



\begin{proof} 
The front $\rho(\g)$ of a characteristic $\g$ of $\V$ has
an inflection at a point $q=\rho(Q)$ iff it has at least $3$-point contact
at $q$ with a line $\rho(\ell)$ of $B$ ($\ell$ being a $\rho^\vee$-fibre). 
This holds, by Lemma\,\ref{tangent->3}, iff the characteristic $\g$
has at least $2$-point contact with the $\rho^\vee$-fibre $\ell$ at $Q$.  
This means that $Q$ belongs to the $\rho^\vee$-contour $\hat{I}$ of $\V$ (by Proposition\,\ref{tangent-fold},
applied to $L=\ell$).
\end{proof}

Therefore, the map $\rho$ send the fold singularities of the map $\rho^\vee_{|_\V}:\V\ra B^\vee$ 
to the inflections of the characteristic fronts of $\V$ -- Fig.\,\ref{fig:fold-inflection}.
\begin{figure}[ht]
\centerline{\psfig{figure=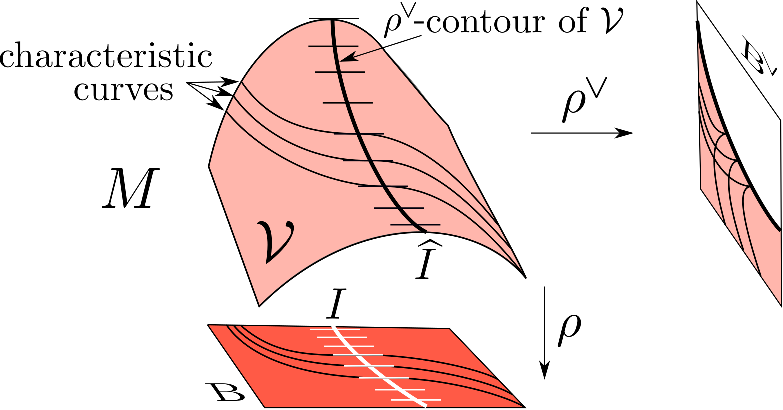,height=3.6cm}}
\caption{\small The map $\rho$ sends the critical set of $\rho^\vee:\V\ra B^\vee$ to the curve of inflections.}
\label{fig:fold-inflection}
\end{figure}

\subsection{Bi-inflections as Whitney Pleat Singularities}

It is well known that {\em a generic map between surfaces has fold singularities along 
a (possibly empty) curve on the source, whose closure is smooth} \cite{avg} (type $1$ in Fig.\,\ref{sing-gener-proj}). 
Besides fold singularities, such a map may only have the so-called {\em Whitney pleat singularity} at isolated points 
(type $2$ in Fig.\,\ref{sing-gener-proj})\,:  
\medskip

\noindent
\textit{\textbf{Whitney pleat}}. The {\em Whitney pleat}, also called the {\em Whitney cusp}, 
is a stable singularity and it is the singularity, at zero, of the map $(x, y) \mapsto (z, w)$: 
\[z=x^3+xy, \quad w=y\,.\]

The fold and the Whitney pleat reduce locally to the projection 
of the surfaces $z=x^2$ and $z=x^3+xy$ (in the $(x,y,z)$-space), respectively, 
to the $(y,z)$-plane along the $x$-axis. 

\begin{remark*}
Let $Q$ be a point of a characteristic $\g$ of $\V$. 
{\em If the front $\rho(\g)$ has a bi-inflection at $q=\rho(Q)$, 
then the surface $\V$ has at least $3$-point contact with the $\rho^\vee$-fibre at $Q$} 
(Lemma\,\ref{tangent->3}). Conversely, 
{\em if a $\rho^\vee$-fibre has $3$-point contact with $\V$ at $Q$ then that 
$\rho^\vee$-fibre has $3$-point contact with the characteristic $\g$}.  
That is, {\em the front $\rho(\g)$ has a bi-inflection at $q=\rho(Q)$}. 
\end{remark*}


So for $\V\subset M$ in general position, the Whitney pleats of $\rho^\vee_{|_\V}:\V\ra B^\vee$ 
correspond to the bi-inflections of the characteristic fronts of $\V$ by $\rho$ (this holds 
for any Legendrian fibration $\rho:M\rightarrow B$ transverse to $\rho^\vee$).
Therefore, we obtain the 
\medskip

\noindent
{\bf Bi-inflection-Pleat Theorem}.
{\em Let $\g$ be a characteristic of a smooth surface $\V\subset M$ in general position. 
The map $\rho^\vee_{|_\V}: \V\ra B^\vee$ has a Whitney pleat singularity at $Q\in\g$
iff the front $\rho(\g)$ has a bi-inflection at $b=\rho(Q)$.}
\begin{figure}[ht]
\centerline{\psfig{figure=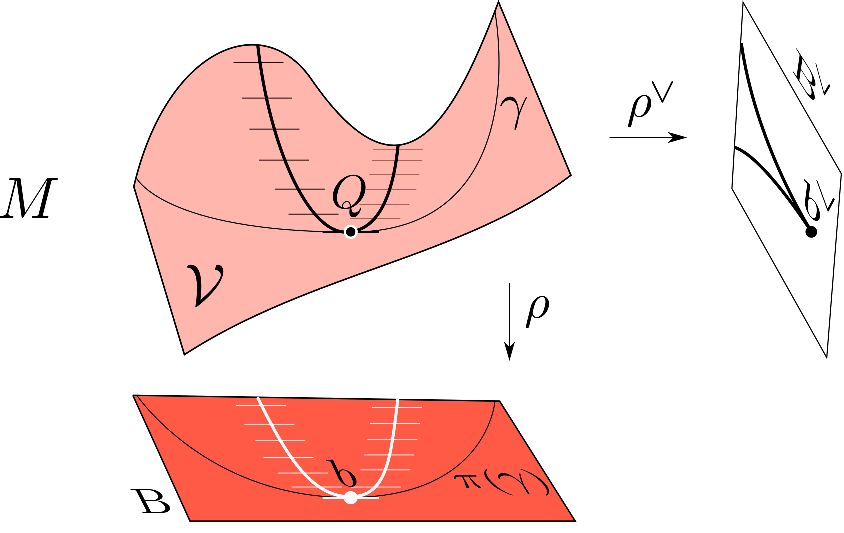, height=4.3cm}}
\caption{\small A Whitney pleat of the map $\rho^\vee:\V\ra B^\vee$ and a bi-inflection.}
\label{bi-inf-cusp}
\end{figure}

\begin{example*}
If $\V^F\subset J^1(\R,\R)$ is the surface of a generic IDE $F(x,y,p)=0$, 
the dual map $\pi^\vee_{|_{\V^F}}:\V^F\ra (\R^2)^\vee$ has a Whitney pleat at a
point $Q\in\V^F$ iff a solution through $b=\pi(Q)$ of the IDE $F=0$ has a
usual bi-inflection. 
\end{example*}

\begin{wremark}
For a generic map between two $2$-dimensional manifolds Whitney Theorem implies:  \ \
1. {\em The critical set is a smooth curve;} \\ 
2. {\em At each point of this curve the kernel of the derivative is $1$-dimensional;}  \\  
3. {\em This kernel is transverse to the critical set at its generic points, being tangent
to this curve at some isolated points.} 

An explicit calculation (using the above normal form) shows that the 
kernel of the derivative is tangent to the critical set only 
at the Whitney pleats.  
\end{wremark}

This implies the following proposition.

\begin{proposition}\label{kernel-tangent}
Let $\V\subset M$ be a generic smooth surface.
A $\rho^\vee$-fibre is tangent to the $\rho^\vee$-contour of $\V$ at a point 
$Q$ iff the front by $\rho$ of the characteristic curve of $\V$ through 
$Q$ has a bi-inflection at $q=\rho(Q)$. 
\end{proposition}

\begin{corollary}[bi-inflection characterisation]\label{cor:bi-inflecrion}
A point $b\in B$ is a bi-inflection of a given characteristic front of $\V$ iff the
curve of inflections (of the neighbouring characteristic fronts of $\V$) is tangent, at $b$, 
to that given front.  
\end{corollary}

\begin{proof}
Since the kernel of the map $\rho^\vee_{|_\V}:\V \ra B^\vee$ at a point (of the 
$\rho^\vee$-contour $\hat{I}$) is the tangent space of the $\rho^\vee$-fibre through 
that point, the statement follows from 
Prop.\,\ref{tangent-fold}, Inflection-contour Theorem and 
Prop.\,\ref{kernel-tangent}.
\end{proof}

\subsection{Quadratic Cones and Legendrian Fibrations}\label{section:cone-fibrations}
{\footnotesize The next definitions and results are needed to prove Theorems\,\ref{th:6bifurcations-BIDE}, 
\ref{6-bifurcations} and \ref{A_3-bifurcations}.
             
Let us list the generic positions of a quadratic cone and two lines through its
vertex.}
\medskip

\noindent
\textit{\textbf{Cone in General Position}}. 
A non degenerate quadratic cone in the vector space $\R^3$ is said to be 
{\em in general position with respect to two given vectorial lines} 
if the cone is not tangent to the plane determined by these lines and none of 
these two lines is a cone generatrix. 
\medskip

\noindent
{\bf Six Generic Positions}. There are six generic relative positions of a quadratic cone
(possibly imaginary) with respect to two given vectorial lines $\ell_1$, $\ell_2$ and to 
the plane plane $\Pi_{\ell_1\ell_2}$ determined by those lines (Fig.\ref{conos}): 
\medskip

\noindent
($a$) the interior of the cone contains $\ell_1$ but not $\ell_2$;

\noindent
($b$) the plane $\Pi_{\ell_1\ell_2}$ intersects the cone 
only at the origin; 

\noindent
($c$) the interior of the cone contains $\ell_2$ but not $\ell_1$;

\noindent
($d$) the cone is imaginary, but its vertex being real; 

\noindent
($e$) the interior intersects the plane $\Pi_{\ell_1\ell_2}$ 
but does not contain $\ell_1$ nor $\ell_2$;

\noindent
($f$) both $\ell_1$ and $\ell_2$ lie inside the cone.
\medskip

\begin{figure}[ht]
\centerline{\psfig{figure=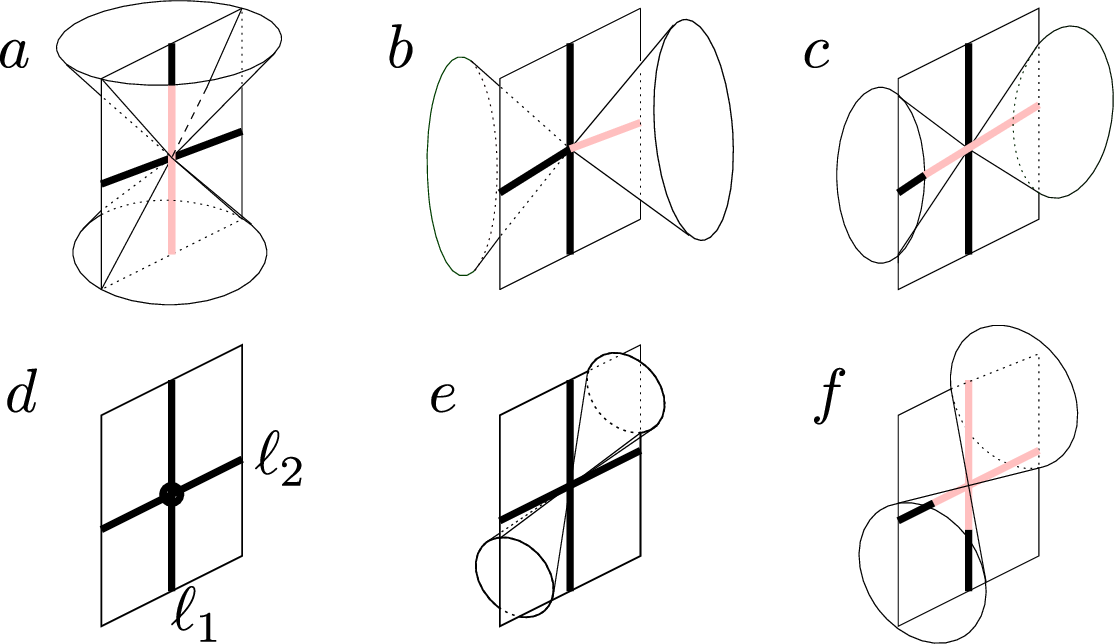,height=5cm}}
\caption{\small Generic relative positions of a quadratic cone with respect to two vectorial lines.}
\label{conos}
\end{figure}

\noindent
\textit{\textbf{Singular Point in General Position}}. 
A Morse conic singular point of a surface $\V\subset M^3$ is {\em in general position
with respect to two transverse Legendrian fibrations} if the quadratic cone determined by that
surface is in general position with respect to the tangent lines to the two fibres at that point. 

\begin{remark}
Let $\pi$ be the standard Legendre fibration of $J^1(\R,\R)$ and 
$\pi^\vee$ its standard dual fibration.  
If $\ell_1$ and $\ell_2$ are the tangents to the $\pi$- and $\pi^\vee$-fibres of $J^1(\R,\R)$ at 
a Morse conic point $Q$ of $\V^F$, then the above six relative positions correspond to the 
transitions $a$ to $f$ in Fig.\,\ref{inflections} (Ths.\ref{th:6bifurcations-BIDE} and \ref{6-bifurcations}).
\end{remark}

\begin{example}\label{quadratic-cone}
If $F(x,y,p)=0$ is an IDE with a Morse singularity at the point
$Q=(0,0,0)$, noted $\bar{0}$, then the equation of the quadratic cone $K$ defined by $F$ at $\bar{0}$ is 
$$F_{xx}(\bar{0})x^2+F_{yy}(\bar{0})y^2+F_{pp}(\bar{0})p^2+
2F_{xy}(\bar{0})xy+2F_{xp}(\bar{0})xp+2F_{yp}(\bar{0})yp=0.$$
At $\bar{0}$ the $\pi$-fibre is the $p$-axis and the $\pi^\vee$-fibre 
is the $x$-axis. Therefore, $\bar{0}$ \textit{is a critical point of $F$ in general position
with respect to $\pi$ and $\pi^\vee$, iff  }
$$(F_{xp}^2-F_{xx}F_{pp})(\bar{0})\neq 0,\quad F_{pp}(\bar{0})\neq 0\quad
\mathrm{and}\quad F_{xx}(\bar{0})\neq 0.$$
\end{example}

In classical geometry one proves that, given a straight line $\ell$ through the vertex of a non-degenerate
(possibly imaginary) quadratic cone $K\subset\R^3$ of equation $P_2(x,y,z)=0$,
the quadrics of equation $P_2(x,y,z)=\e$, for any $\e\in\R$ (including $K$), 
have the same \textit{diametral plane conjugate to $\ell$}: 
\label{diametral-construction}
\medskip

\noindent
\textit{\textbf{Conjugate Diametral Plane}}. 
The \textit{diametral plane of $K$ conjugate to $\ell$} is defined as follows: 
a plane $\Pi$ containing $\ell$ intersects $K$ along two lines $\ell_2$, $\ell_3$ 
(possibly imaginary). Take the line $\ell_4=\ell_4(\Pi)$ that 
together with $\ell$ separate harmonically the pair $\ell_2$, $\ell_3$  
(i.e., the cross ratio $(\ell_2, \ell_3; \ell, \ell_4)$ equals $-1$). 
\textit{The union of the fourth harmonic lines $\ell_4(\Pi)$, taken over 
all planes containing $\ell$ is a plane}. It is called \textit{diametral plane} of $K$ 
conjugate to $\ell$. 
\medskip


\noindent
\textit{\textbf{Conjugate Diameters}}. 
Two lines through the vertex of a non-degenerate quadratic cone $K$ 
are said to be {\em conjugate diameters} of $K$, if one of them is 
contained in the conjugate diametral plane of the other.
\smallskip


\noindent
\textbf{\textit{Harmonicity}}. \label{example:conjugate-diameters}
We conclude that \textit{two lines $\ell$, $\ell'$ are conjugate diameters of $K$ iff their cross-ratio with the lines 
$\ell_1$, $\ell_2$ (real or complex) on which $K$ intersects the plane $\Pi_{\ell\ell'}$, determined by  $\ell$, $\ell'$, 
equals $-1$}, that is $(\ell,\ell',\ell_1,\ell_2)=-1$. 
\smallskip

Let $\mu:\R^3\rightarrow \R^2$ be a fibration of parallel lines 
and write $\ell$ for the $\mu$-fibre through the origin.
The basic properties of conjugated diameters imply: 

\begin{proposition}\label{conjugate-plane}
Let $K\subset\R^3$ be a non-degenerate quadratic cone of equation $P_2(x,y,z)=0$. 
If the $\mu$-fibre $\ell$ does not lie on $K$, then the union of the $\mu$-contours of the quadrics
$P_2(x,y,z)=\e$, with $\e\in\R$, is the diametral plane conjugate to $\ell$ of any of these 
quadrics (including $K$).
\end{proposition}

In fact, {\em the $\mu$-contour of a quadric $P_2(x,y,z)=\e$ is the 
intersection of that quadric with the diametral plane of $K$ conjugate to $\ell$}.

\begin{proposition}\label{tangent-plane}
Let  $F:\R^3\ra \R$ be a function whose level surface 
$F=0$ has a conic Morse singularity at $\bar 0$. 
If the $\mu$-fibre $\ell$ at $\bar 0$ does not lie on the quadratic cone $K$ defined 
by $F$, then the surface formed by the $\mu$-contours 
of the level surfaces $F=\e$ is smooth at $\bar 0$, its tangent plane at $\bar 0$ is the 
diametral plane of $K$ conjugate to $\ell$ and this plane does not contain $\ell$.
\end{proposition}

\begin{proof}
By Proposition\,\ref{conjugate-plane}, near $\bar 0$, the 
$\mu$-contours of the surfaces $F=\e$ form a surface which is a deformation (by 
higher order terms) of the diametral plane of $K$ conjugate to $\ell$.  
The tangent plane of this surface at $\bar 0$ is the diametral plane of $K$ conjugate to $\ell$  
beacause it depends only on the quadratic terms of $F$.  
The plane does not contain $\ell$ because a diametral plane of $K$ conjugate to $\ell$ contains $\ell$ 
iff that plane is tangent to $K$ along $\ell$. 
\end{proof}

\section{Local Transitions of the Curve of Inflections}\label{proofs}

We separate the local transitions of the curve of inflections and discriminant of BIDEs in Fig.\,\ref{inflections} 
(and those of the flecnodal and parabolic curves of surfaces in Fig.\,\ref{14perestroikas}) into classes 
of distinct natures\,: \\ 
{\bf I}. Transitions involving characteristic points: $UW$, flec-folded singularity and $\g v$ for BIDE 
($A_4$ and flec-godron for surfaces in $3$-space); \\
{\bf II}. Transitions involving Morse conic points of $\V^F$: $a$ to $f$ for BIDE ($A_3$ for surfaces
in $3$-space); \\
{\bf III}. Transitions involving Whitney pleats (lips, bec-à-bec, swallowtail);\\
{\bf IV}. Transitions involving multi-singularities: creation/annihilation of hyperbolic (elliptic) nodes,  
flec hyperbolic node for BIDE (creation/annihilation of hyperbonodes (or ellipnodes) and flec-hyperbonode for surfaces in $3$-space); 

\noindent
For surfaces in $3$-space there is an additional type: $D_4$ transitions. 

\smallskip
Along this section we consider a contact $3$-manifold $M$ and we take two fixed transverse 
Legendre fibrations $\rho:M\ra B$ and $\rho^\vee:M\ra B^\vee$. 

\subsection{Transitions Involving Characteristic Points}\label{UW-gv}

\noindent
{\bf Theorem \ref{proofs}.0}. {\em Let $Q$ be a characteristic point of a smooth surface $\V\subset M$ such 
that the $\rho$-fibre at $Q$ has exactly $2$-point contact with $\V$ and 
\smallskip  

\noindent
a) the $\rho$-contour $\widehat{D}$ of $\V$ is not tangent to the
$\rho^\vee$-contour $\widehat{I}$ of $\V$; 

\noindent
b) the $\rho^\vee$-fibre is not tangent to the $\rho^\vee$-contour $\widehat{I}$ of $\V$;

\noindent
c) the $\rho$-fibre is not tangent to the $\rho^\vee$-contour $\widehat{I}$ of $\V$. 


\noindent
Then the discriminant $\,D=\rho(\hat{D})\,$ and the curve of inflections $I=\rho(\,\hat{I}\,)\,$
are quadratically tangent at $q=\rho(Q)$. This point $q=\rho(Q)$ locally separates
$I$ into its left and right branches $\,I_\ell$, $\,I_r\,$} (Fig.\ref{contour}).

\begin{proof}
The derivative of $\rho_{|_\V}:\V\ra B$ at $Q$ sends tangent lines of $\hat{D}$ and
$\hat{I}$ to the image line of the tangent plane of $\V$ at $Q$. So $I$ and $D$ are tangent. 
  
The condition of $2$-point contact of the $\rho$-fibre at $Q$ with $\V$
implies that there are local smooth coordinates $(x,y,z)$ in $M$ (near $Q$) and
$(x,y)$ in $B$ (near $q=\rho(Q)$) such that $\V$ is given by the equation $y=z^2$;
the $\rho$-contour $\hat{D}$ is the $x$-axis and the map $\rho$ is given by $(x,y,z)\mapsto (x,y)$. 

In the sequel, the notation ``$\mathrm{h.o.t.}(x)$'' means ``higher order terms in $x$''. 

By condition $(b)$, the contour $\hat{I}$ may be parametrised by the $x$-coordinate  
\begin{equation}\label{pre-inflections}
x\mapsto (x,(\a x+\b x^2 + \mathrm{h.o.t.}(x))^2 ,\a x+\b x^2 +\mathrm{h.o.t.}(x)).
\end{equation}
Its image $I=\rho(\hat{I})$ is the curve $(x,\a^2x^2+2\a \b x^3+\b^2x^4+ \mathrm{h.o.t.}(x))$.
Condition $(a)$ implies $\a\neq 0$, proving the quadratic tangency with $D$ (the $x$-axis).
 
\begin{figure}[ht]
\centerline{{\psfig{figure=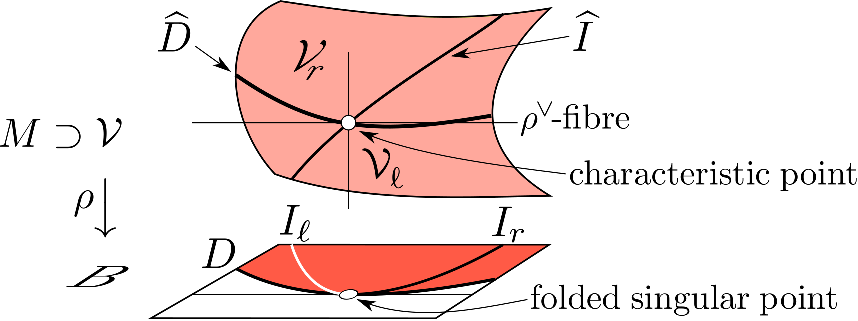,height=3.3cm}}}
\caption{\small Projection of two transverse contours at a characteristic point.}
\label{contour}
\end{figure}

Since $\hat{D}$ and $\hat{I}$ intersect each other transversely at $Q$, the contour 
$\hat{D}$ locally separates $\hat{I}$ at $Q$ (Fig.\,\ref{contour}).
That is, $\hat{I}$ has one branch on the left component of $\V$ and the other branch on the right
one. This implies that $q=\rho(Q)$ locally separates the left and right branches of the curve of
inflections.
\end{proof}

In generic $1$-parameter families of IDE, one of the conditions $a)$, $b)$ or $c)$ can fail at isolated parameter 
values at which occur the following respective transitions:

\begin{theorem}\label{special-tangent}
Let $\V_\e\subset M$ be a generic $1$-parameter family of smooth surfaces. 
If the surface $\V=\V_0$ has a characteristic point in which one of the conditions a, b or c of 
Th.\,{\rm \ref{proofs}.0} is broken, then the curve of inflections of $\V_\e$ undergoes the 
following respective transitions at $\e=0$ {\rm (Fig.\,\ref{inflections}):} 
\smallskip

\noindent 
{\bf a)} $UW$$:$ $(\e=0)$ the point $q=\rho(Q)$ is a folded singular point at which the discriminant $D=\rho(\hat{D})$ 
has $4$-point contact with $I=\rho(\,\hat{I}\,)$. It disappears $(\e<0)$ or splits into
two folded singular points of opposite indices $(\e>0)$.
\smallskip

\noindent 
{\bf b)} flec-folded singularity$:$ a negative folded singularity overlaps $($at $\e=0$$)$ with a biflecnode
which passes from one branch of the curve of inflections to the other (as $\e$ changes sign).
At $\e=0$ the curve of inflections $\rho(\,\hat{I}\,)$ has itself an
inflection at $q=\rho(Q)$.
\smallskip

\noindent 
{\bf c)} $\g v$$:$ for $\e\in(-1,1)$ the surface of the $3$-space $B\times (-1,1)$ formed by
the curves of inflections of $\V_\e$ is locally diffeomorphic to
the Whitney umbrella. For $\e=0$ 
the curve $\rho(\,\hat{I}\,)$ has a semi-cubic cusp at the folded 
singular point $q=\rho(Q)$. For any sufficiently small $|\e|$ 
the folded singular point is generic. 
\end{theorem}

\begin{proof}
\noindent
{\bf Case a}. The curve $\hat{I}$ can be parametrised by \eqref{pre-inflections}, with 
$\a=0$ and $\b\neq 0$. So its image $I=\rho(\,\hat{I}\,)$ is the curve 
$(x, \b^2 x^4 + \mathrm{h.o.t.}(x))$, that is, $I$ has $4$-point contact with 
$D=\rho(\hat{D})$ (the $x$-axis). 

A small generic deformation of the function $\b x^2 +\mathrm{h.o.t.}(x)$ is given (up to a 
`translation in the $x$-coordinate') by $\e+\bar\b x^2 +\mathrm{h.o.t.}(x)$, with $\bar\b\neq 0$. 
The corresponding deformation of the contour $\hat{I}$, $x\mapsto (x, (\b x^2 + \mathrm{h.o.t.}(x))^2 ,\b x^2 +\mathrm{h.o.t.}(x))$, 
is given by $x\mapsto (x, (\e+\bar\b x^2 + \mathrm{h.o.t.}(x))^2 ,\e+\bar\b x^2 +\mathrm{h.o.t.}(x))$. So the
family of curves of inflections is equivalent to the standard $UW$ family of curves  
$x\mapsto (x, (\e+\bar\b x^2)^2)$.  

Suppose $\bar\b<0$, i.e. $\b<0$ (the case $\bar\b>0$ is equivalent).  
For small $\e<0$, the function $\e+\bar\b x^2 + \mathrm{h.o.t.}(x)$ has no zeros near the origin.  
Thus the curves $\hat{D}_\e$ and $\hat{I}_\e$ 
have no point of intersection (see Fig.\,\ref{duble-par} left). 

\begin{figure}[ht]
\centerline{\psfig{figure=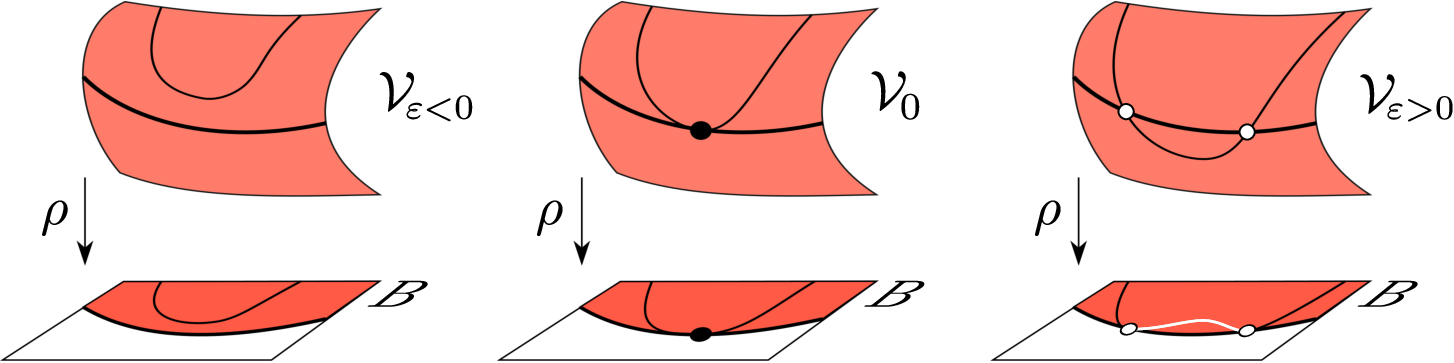,height=3cm}}
\caption{\small The $UW$-transition.}
\label{duble-par}
\end{figure}

For small $\e>0$ the function $\e+\bar\b x^2 + \mathrm{h.o.t.}(x)$ has two zeros 
near the origin. Thus the curves  $\hat{D}_\e$ and $\hat{I}_\e$ 
intersect at two neighbouring folded singular points (see Fig.\,\ref{duble-par} right). 
\medskip

\noindent
{\bf Case b}. \ 
By Whitney Pleat Remark and Proposition\,\ref{kernel-tangent} (p.\pageref{kernel-tangent}), 
condition $(b)$ implies that the map $\rho^\vee_{|_\V}:\V\ra B^\vee$
has a Whitney pleat at $Q$. 
The $\rho$-contour $\hat{D}_\e$ of $\V_\e$ locally separates $\V_\e$ into its left and right
components and, hence, it also separates $\hat{I}_\e$ into its left and right branches. 
Since the characteristic points and the Whitney pleats are both stable, also 
the surfaces $\V_{\e \neq 0}$ have a characteristic point and have a Whitney pleat. 
For a generic family $\V_\e$ the Whitney pleat crosses transversely the $\rho$-contour, 
passing from one branch of $\hat{I}_\e$ 
to the other at $\e=0$.  
Theorem\,\ref{special-tangent}.$b$ follows from this fact. 
\medskip

\noindent
{\bf Case c}. \ 
By condition $(c)$, the $\rho^\vee$-contour $\hat{I}$ can be parametrised 
by the $z$-coordinate as $(\b z^2+\g_0 z^3+ \mathrm{h.o.t.}(z), z^2, z)$. Its image under $\rho$ 
is the curve 
$$(\b z^2 + \g_0 z^3 + \mathrm{h.o.t.}(z), z^2),$$
which is sent by the affine transformation $(x,y)\mapsto (x-\b y, y)$ to the 
semi-cubic cusp $(\g_0 z^3+\mathrm{h.o.t.}(z), z^2)$, $\g_0\neq 0$ (the condition  
$\g_0 = 0$ is of codimension $2$ and would give rise to a ramphoid cusp 
$(\d z^4+\mathrm{h.o.t.}(z), z^2)$, $\d\neq 0$). 

\begin{figure}[ht]
\centerline{\psfig{figure=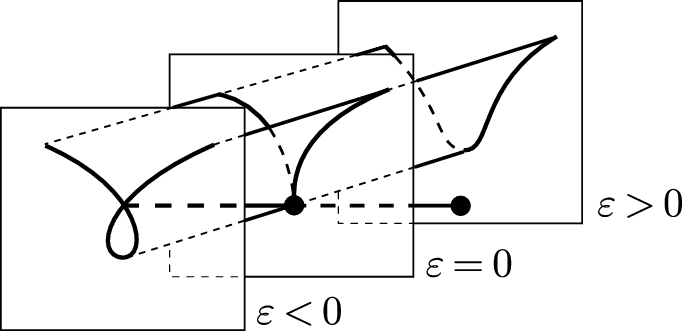,height=2.2cm} \ \ \psfig{figure=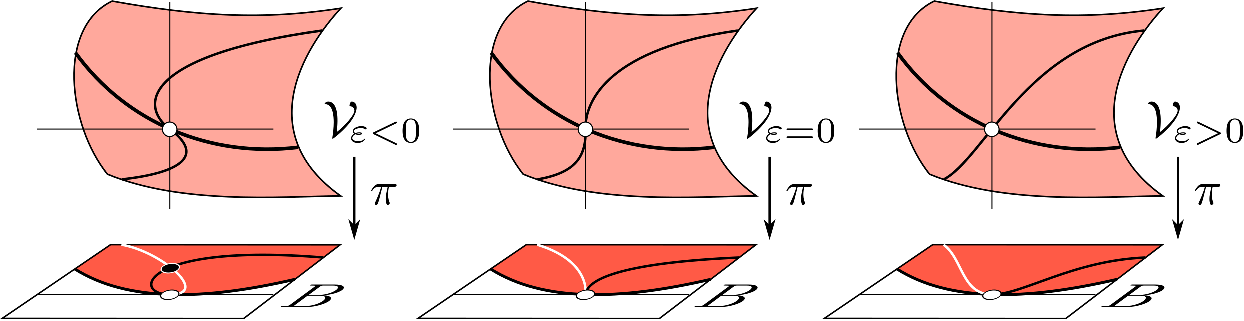,height=2.2cm}}
\caption{\small The $\g v$-transition in ``surface-time'' $B\times(-1,1)$: a Whitney umbrella.}
\label{umbrellasections}
\end{figure}

A deformation of $\V=\V_0$ inside a generic $1$-parameter family $\V_\e\subset M$ 
induces a deformation of the curve of inflections $z \mapsto (\g_0 z^3+\mathrm{h.o.t.}(z), z^2)$,
which is given (up to a `translation in the 
first coordinate' and an affine transformation $(x,y)\mapsto (x-\b y, y)$) 
by the parametrisation 
$z\mapsto (\e z+\g z^3+\mathrm{h.o.t.}(z), z^2)$, where $\g=\g(\e)$ and $\g(0)=\g_0$. 

As the parameter $\e$ changes, the tangent line 
to the $\rho^\vee$-contour of $\V_\e$ at the characteristic point passes 
through the `vertical' position (at $\e=0$) with non zero angular velocity. 


The surface in plane-time $(x,y,\e)$ formed by the union of the curves 
of this family is the image of the map $(z,\e)\mapsto (\e z+\g z^3+\mathrm{h.o.t.}(\e,z), z^2, \e)$.
This map is a perturbation with higher order terms of the (Whitney) map 
$(z,\e)\mapsto (\e z+\g_0 z^3, z^2, \e)$, which is stable (Appendix\,\ref{pinch-point}). 
For fixed $\e<0$, the parametrised curve $z\mapsto (\e z+\g_0 z^3, z^2)$ has a double
point which is the image of the values $z=\pm \sqrt{-\e/\g_0}$, like 
in Fig.\,\ref{umbrellasections}. 
\end{proof}

\subsection{Transitions Involving Quadratic Conic Points}\label{sect:morse}


\begin{theorem}\label{6-bifurcations}
Let $\{\V^{F_\e}\subset M\}_{\e \in\R}$ be a generic $1$-parameter family of 
surfaces given by functions $F_\e:M\ra \R$. 
If $F=F_0$ has a Morse conic point $Q$ in general position with respect to $\rho$ and $\rho^\vee$,
and the tangents to the two fibres at $Q$ are not conjugate diameters of the cone $K$ 
defined by $\V$, then both curves $I$ and $D$ undergo a Morse transition -- 
we get one of the transitions $a$ to $f$ in Fig.\,\ref{inflections}, in accordance with 
the position of the cone $K$, $a$ to $f$ in Fig.\,\ref{conos}. 
Moreover, two folded singularities of equal indices are born or die {\rm (Fig.\,\ref{inflections})}. 
\end{theorem}

In Fig.\,\ref{rica}, 
we show a deformation of the cone (in the cases $a$ and $d$) 
and the projections to the $xy$-plane of its contours by $\pi$ and by $\pi^\vee$.
See also Fig.\,\ref{inflections}.

\begin{figure}[ht]
\centerline{\psfig{figure=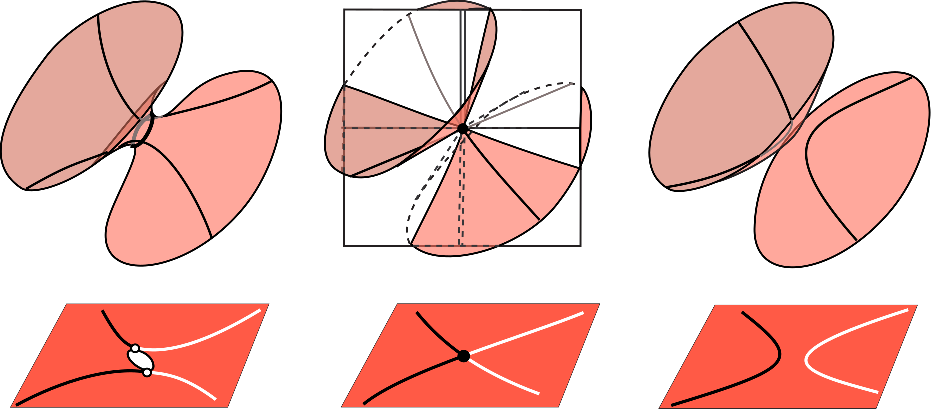,height=3cm} \ \psfig{figure=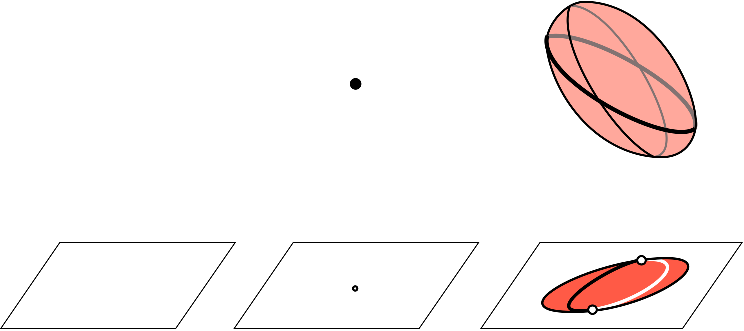,height=3cm}}
\caption{\small Deformation of a cone in $J^1(\R,\R)$ and the projections by $\pi$ 
of the contours with respect to $\pi$ and $\pi^\vee$. Cases $a$ and $d$. 
(In case $d$ the cone is imaginary.)}
\label{rica}
\end{figure}



\subsection*{Proof of Theorem\,\ref{6-bifurcations}}
Consider a Legendrian fibration $\mu:M\ra B$ of our contact $3$-manifold $M$.  

\begin{lemma}\label{tangent-plane-legendrian}
Let  $F:M\rightarrow \R$ be a function whose level surface 
$F=0$ has a Morse conic singular point $Q$. 
If the tangent line $\ell$ of the $\mu$-fibre at $Q$ does not lie 
on the quadratic cone $K\subset T_QM$ defined by $F$, 
then 

\noindent
$(i)$ at $Q$, the surface swiped by the $\mu$-contours of the 
level surfaces $F=\e$ {\rm (that we note $\Sigma^\mu_F$)} is smooth,  
its tangent plane is the diametral plane of $K$ conjugate to $\ell$ and 
this plane does not contain $\ell$; 

\noindent
$(ii)$ near $Q$, the foliation of the surface $\Sigma^\mu_F$, which is naturally defined by the $\mu$-contours 
of the surfaces $F=\e$, is diffeomorphic to the foliation of the plane by the level curves of a 
function near a Morse critical point. 
\end{lemma}

\begin{proof}
$(i)$ It follows from Proposition\,\ref{tangent-plane}: use a local diffeomorphism sending $Q$ to $\bar{0}\in\R^3$ 
and rectifying the fibres of $\mu$ to parallel lines - its differential preserves the conjugacy relations. 
\medskip

\noindent
$(ii)$ By item $(i)$, we can take local coordinates $x,y,z$ near $Q$ 
such that the surface  $\Sigma^\mu_F$ is locally given by the equation $z=0$ and 
the surface $\V^{F_\e}$ has equation $\e=ax^2+by^2+cz^2+dxy+exz+gyz+\psi_{\geq3}(x,y,z)$, 
where the quadratic cone $K$, of equation $ax^2+by^2+cz^2+dxy+exz+gyz=0$, is non degenerate and not 
tangent to the surface $z=0$. 

So, in these coordinates, the foliation of the surface $\Sigma^\mu_F$ is given by the level curves 
$\e=ax^2+by^2+dxy+\psi_{\geq3}(x,y,0)$. Lemma\,\ref{tangent-plane-legendrian} is proved. 
\end{proof}

Let us come back to our fixed Legendrian fibrations $\rho:M\ra B$ and $\rho^\vee:M\ra B^\vee$.
The following lemma is the simplest case of Theorem~\ref{6-bifurcations}. 

\begin{lemma}
Let  $F:M\ra \R$ be a function whose level surface $F=0$ has a Morse conic point $Q$ 
in general position with respect to $\rho$ and $\rho^\vee$. If the tangents to the $\rho$ and 
$\rho^\vee$-fibres at $Q$ are not conjugate diameters of the cone $K$, 
then both curves $D_\e$, $I_\e$ undergo a Morse transition at $\e=0$$:$ 
We get one of the transitions $a$ to $f$ in Fig.\,\ref{inflections} 
according to the position of $K$ in Fig.\,\ref{conos}. 
\end{lemma}

\begin{proof}
We apply Lemma\,\ref{tangent-plane-legendrian} for our fixed fibrations $\mu=\rho$ and $\mu=\rho^\vee$. 
\end{proof}

To prove Theorem\,\ref{6-bifurcations}, we need the following notion\,:
\medskip

\noindent
\textit{\textbf{Flec-Surface}}. 
The \textit{flec-surface of a function} $F:M\ra \R$ is the union of the $\rho^\vee$-contours $\hat{I}_\e$ 
of the level surfaces $F=\e$. 
More generally, the {\em flec-surface $\mathcal{F}\ell(\V_\e)$ of a $1$-parameter family of surfaces} $\V_\e\subset M$ is
the surface swiped by the $\rho^\vee$-contours $\hat{I}_\e$ of the surfaces $\V_\e$. 

\begin{remark*}
Thus the $\rho^\vee$-contour $\hat{I}_\e$ of the surface $\V_\e$ is the intersection of $\V_\e$ with the 
flec-surface of the function $F_\e$. The $\rho^\vee$-contours $\hat{I}_\e$, $\e \in \R$, define a (singular) 
foliation on the flec-surface $\mathcal{F}\ell(\V_\e)$ (by its definition).  
\end{remark*}

\noindent
{\bf Proof of Theorem~\ref{6-bifurcations}}. \ 
Write $K$ for the quadratic cone of $T_QM$ defined by the surface $\V_0$, $\mathcal{F}\ell(\V_\e)$ for the 
flec-surface of the family $\V_\e$ and $\ell$ for the tangent line of the 
$\rho^\vee$-fibre at $Q$. The flec-surface $\mathcal{F}\ell(\V_\e)$ is  
smooth at $Q$ and very close to the diametral plane conjugate to $\ell$ of $K$ 
(being a smooth deformation of this plane, see 
Lemma\,\ref{tangent-plane-legendrian}). 

First, we have to prove that the tangent plane of $\mathcal{F}\ell(\V_\e)$ at $Q$ is still the 
diametral plane of $K$ conjugate to $\ell$. 
The flec-surface $\mathcal{F}\ell(\V_\e)$, being the union of the $\rho^\vee$-contours 
$\hat{I}_\e$ of the surfaces $\V_\e$ contains the curve $\hat{I}_0$ (singular at $Q$) which is the 
intersection of the surface $\V_0$ (having a conic Morse singularity at $Q$) 
with the (smooth) flec-surface of the function $F:=F_0$ (see Lemma~\ref{tangent-plane-legendrian}). 
So {\em the curve $\hat{I}_0$ has a Morse double 
point at $Q$}. If the branches of $\hat{I}_0$ are real, then their tangent 
lines determine the tangent plane of the surface $\mathcal{F}\ell(\V_\e)$, which is the tangent 
plane of the flec-surface of $F_0$ at $Q$ and is therefore the 
diametral plane of $K$ conjugate to $\ell$. If the branches of $\hat{I}_0$ are 
complex conjugate, then  (locally) $\hat{I}_0$ is just the point $Q$ but 
the tangent plane of the flec-surface $\mathcal{F}\ell(\V_\e)$ at $Q$ is still the tangent 
plane of the flec-surface of the function $F_0$, that is, the 
diametral plane of $K$ conjugate to $\ell$.

The fact that the tangent lines of the fibres of $\rho$ and $\rho^\vee$ at $Q$ are not conjugate 
diameters of $K$ iff the tangent plane to $\mathcal{F}\ell(\V_\e)$ at $Q$ \textit{is not vertical} 
(i.e., it does not contain the $\rho$-direction) is stated below in Proposition\,\ref{prop:vertical-flecsurface} 
for $M=J^1(\R,\R)$, $\rho=\pi$ and $\rho^\vee=\pi^\vee$. This fact holds for arbitrary $M$ 
because all Legendre fibrations are locally contactomorphic. 
Thus, the projection $\rho_{|_{\mathcal{F}\ell(\V_\e)}}: \mathcal{F}\ell(\V_\e)\ra B$  of the flec-surface
$\mathcal{F}\ell(\V_\e)$ at $Q$ is a local diffeomorphism. The local transition of the curve of 
inflections is thus described by the singular foliation of $\mathcal{F}\ell(\V_\e)$ in the neighbourhood 
of $Q$. Since the curve $\hat{I}_0$ has a Morse singularity at $Q$, the curve of 
inflections undergoes a Morse transition. 

The change (by $2$ or $-2$) of the Euler characteristic of the surface $\V_\e$ at the transition,
implies that the number of folded singular points changes locally by 2 and also provides the
indices of those folded singular points.  \fin


\subsection{Transitions Involving Whitney Pleats}
Let $\V_\e$ be a generic $1$-parameter family of smooth surfaces 
of $M$, with $\V=\V_0$. 

\begin{theorem}\label{3-4-5}
If the map $\rho^\vee_{|_\V}: \V \rightarrow B^\vee$, has a singularity of type $3$, $4$ or $5$ 
of Fig.\,\ref{sing-gener-proj}, then when the parameter $\e$ passes through $0$, the curve of inflections of 
the $\rho$-solutions of $\V_\e$ undergoes a lips-, bec-\`a-bec-  or 
swallowtail-transition, respectively, described in Figs.~\ref{inflections} 
and \ref{14perestroikas}. 
\end{theorem}


\begin{proof}
Theorem~\ref{3-4-5} follows from the known fact (c.f. \cite{avg}) that under a generic $1$-parameter 
deformation of the singularities of types 3, 4 and 5 (in Fig.\,\ref{sing-gener-proj}), their respective 
contours undergo a transition in which two Whitney pleats are born or die (see Fig.\,\ref{apparent}). 

A generic $1$-parameter deformation of the surfaces 
$z=x^3+ xy^2$ and $z=x^3-xy^2$ whose projections along the $x$-axis have the 
singularities of the types 3 and 4, is given by the families
$z=x^3+ xy^2+\e x$ and $z=x^3-xy^2+\e x$, respectively. For each $\e$, 
the intersection of the plane $y=y_0=\mathrm{const}.$ with the 
curve of singular points of the projection of the surface along 
the $x$-axis consists of the points of the graph of the function 
$f|_{y=y_0}(x)= x^3\pm xy_0^2+\e x$ at which the tangent line 
is horizontal. Thus, the projection of the curve of 
singular points to the $(x,y)$-plane consists of the points at which 
the partial derivative with respect to $x$ vanishes, $f_x(x,y;\e)=0$. 
One gets the family of conics $3x^2\pm y^2+\e=0$ which, at $\e=0$, has a Morse transition 
(lips or bec-\`a-bec in Fig.~\ref{apparent}). The born or dying bi-inflections in these transitions 
correspond to the born or dying Whitney pleats of the $\rho^\vee$-contour of $\V$. 

Similarly, for the type $5$ singularity, we have the surfaces $z=x^4+xy+\e x^2$ and we 
obtain the curve $4x^3+y+2\e x=0$. For $\e>0$ this curve has two points whose tangent 
is parallel to the $x$-axis (corresponding to two Whitney pleats), at $\e=0$ 
these points collapse and for $\e<0$ they disappear. 
\end{proof}

\subsection{Transitions Involving Multisingularities} 

The bec-à-bec transition of the curve of inflections of the fronts of $\V$
comes from a Morse transition at an unstable double point of the $\rho^\vee$-contour of $\V$. 

The stable self-intersection points of the curve of inflections are multisingularities: 
if a generic surface $\V$ has two parts $\V_\ell$ and $\V_r$, their $\rho^\vee$-contours $\hat{I}_\ell$ and $\hat{I}_r$
are smooth curves whose images $I_\ell=\rho(\hat{I}_\ell)$ and $I_r=\rho(\hat{I}_r)$ are smooth curves of $B$ 
which may have points of transverse intersection (hyperbolic nodes), which are stable multisingularities. 
For a BIDE they are also hyperbolic nodes and 
for a smooth surface they are hyperbonodes (Fig.\,\ref{points}). 

\begin{theorem}\label{multi-sing} 
Consider a generic $1$-parameter family of surfaces $\V_\e\subset M$. 
\medskip

\noindent
$(i)$ If, for $\V_0$, the curves $I_\ell$, $I_r$ are quadratically tangent at $q\in B$, then 
we get a creation/annihilation transition of two hyperbolic nodes 
{\rm (c/a-h$\cdot$n in Fig.\,\ref{inflections})}.
\medskip

\noindent 
$(ii)$ If $q\in B$ is a hyperbolic node for $\V_0$ and $q$ is also a left bi-inflection, then as $\e$ passes 
through $0$ a left bi-inflection moves along $I_\ell$, crossing $I_r$ at the hyperbolic node $q$ 
{\rm (flec-hn in Fig.\,\ref{inflections})}. {\rm (Left and right may be interchanged.)}
\end{theorem}

\noindent
{\em Proof}. \ 
We have two parts $\V_\ell$, $\V_r$ (of $\V$) whose $\rho^\vee$-contours $\,\hat{I}_\ell,$\, $\hat{I}_r\,$ 
project under $\rho$ to $I_\ell$ and $I_r$. 
Item (i) follows from the following fact:
a generic deformation of $\V_\ell$ and $\V_r$ will not keep the tangency of their corresponding 
curves of inflections $I_\ell$ and $I_r$, giving rise to a creation/annihilation transition.

For the proof of Item (ii), first, we can fix $\V_1$ in such a way that the 
projection $\rho^\vee|_\V:\V \rightarrow C$ has a Whitney pleat at a point of $\V_1$, 
that is, there is a bi-inflection point lying in $I_1=\rho(\hat{I}_1)$. 
Next, one obtains the transition 14 of Fig.~\ref{inflections} 
by moving $\V_2$ in a generic way: the curve $I_2=\rho(\hat{I}_2)$ 
passes trough the bi-inflection point in $I_1$. 
The stability of the transition (ii) follows from the stability 
of the Whitney pleat singularity. 
\fin

\part{Transitions on Evolving Surfaces}

\section{Basic Properties of Smooth Surfaces in $3$-Space}\label{classification}

\subsection{Classification of Points by the Contact with Lines}\label{Tangential-Classification}

The points of a generic smooth surface in a real 3-space (projective, affine or Euclidean)
are classified by the contact of the surface with its tangent lines .
\begin{figure}[ht]
\centerline{\psfig{figure=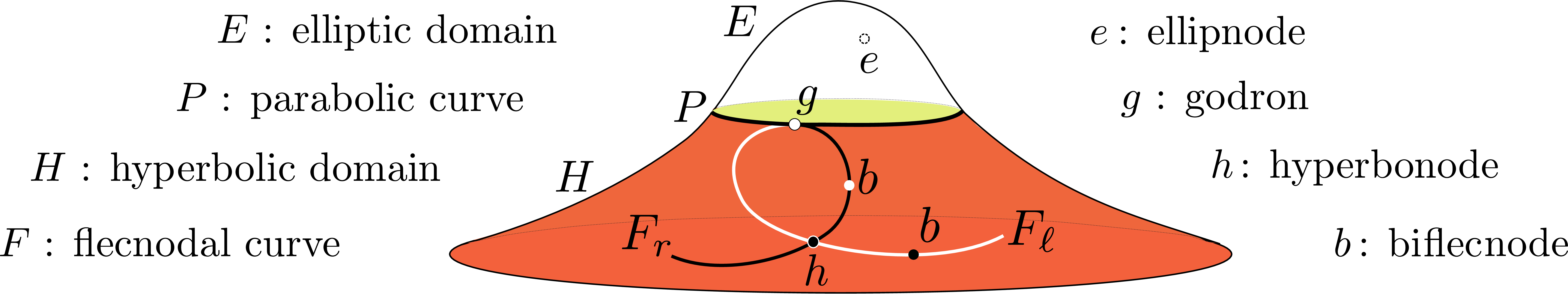,height=2.4cm}}
\caption{\small The 8 tangential singularities of a generic smooth surface.}
\label{points}
\end{figure}

A generic smooth surface $S$ has three (possibly empty) parts\,:  \\
({\bf E}) an open {\em domain of elliptic points}\,: no real tangent line exceeds
2-point contact with $S$; \ 
({\bf H}) an open {\em domain of hyperbolic points}\,: there are 
two such lines, called {\em asymptotic lines} (they determine two {\em asymptotic directions}); and  
({\bf P}) a smooth {\em curve of parabolic points}\,: a unique (double) asymptotic line. 

The {\em parabolic curve} separates the elliptic and hyperbolic domains of $S$.

In the closure of the hyperbolic domain there is 
({\bf F}) a smooth immersed {\em flecnodal curve} formed by the points at which an asymptotic 
tangent line exceeds 3-point contact with $S$.  
 
There are also four types of isolated points\,: 
({\bf g}) a {\em godron} is a parabolic point at which the (unique) asymptotic line
is tangent to the parabolic curve; \ 
({\bf h}) a {\em hyperbonode} is a point of transverse self-intersection of the flecnodal curve; \
({\bf b}) a {\em biflecnode} is a point of the flecnodal curve at which one asymptotic tangent has 
$5$-point contact with $S$;
({\bf e}) an {\em ellipnode} is a real point in the elliptic domain of the simplest self-intersection
of the complex conjugate flecnodal curves associated to the complex conjugate asymptotic lines.
\smallskip

\label{inflections_bi-inflections}
\noindent
\textit{\textbf{Inflections of Plane and Space Curves}}. 
For a smooth curve in $3$-space or in the plane an {\em inflection} (resp. \textit{bi-inflection})  
is a point at which the first two (resp. first three) derivatives are colinear. 
Equivalently, the curve has $3$-point (resp. $4$-point) contact with its tangent line at that point. 
If a Euclidean structure is admitted, the curvature satisfies $k=0$ (resp. $k=k'=0$). 
\smallskip

{\em A generic smooth curve of $\RP^3$ {\rm (or $\R^3$)} has no inflection}.
However, the asymptotic curves of a generic smooth surface may have inflections;
being `stable' under small perturbations of the surface. 
Moreover, if an asymptotic curve has an inflection, then 
the neighbouring asymptotic curves have also a neighbouring inflection.
In \S\ref{subsect:flecnodal-inflections} (Proposition\,\ref{prop:flacnodal-space-inflections}), we prove that 
\smallskip

\noindent
{\em The curve formed by the inflections of the asymptotic curves of a smooth surface 
coincides with the flecnodal curve of that surface}.
\medskip

\noindent
{\bf Hyperbonodes and ellipnodes}. 
If the surface is locally presented in Monge form 
$$z=f_2(x,y)+f_3(x,y)+\ldots,$$
where $f_j$ is a homogeneous polynomial of degree $j$, \textit{the point is a hyperbonode or an ellipnode
iff the quadratic form $f_2$ divides the cubic form $f_3$, that is, the cubic form is zero modulo
the quadratic form} \cite{Kazarian-Uribe}. 

\begin{example*}
A smooth surface in $\RP^3$ 
is a one-sheet hyperboloid (ellipsoid) iff all its points are hyperbonodes (resp. ellipnodes). 
\end{example*}

{\footnotesize
\begin{remark*}
Seven of the above tangential singularities were well known in the 19th Century (\cite{Salmon}), but 
their normal forms up to the 5-jet, under projective transformations,
were given in \cite{Landis,Platonova} and ellipnodes were found by Panov (\cite{Dima}). 
The numbers of ellipnodes, hyperbonodes and godrons on a surface are 
related to the Euler characteristic of that surface by formulas which 
determine the possible coexistences of these points on the surface \cite{Kazarian-Uribe}.
A projective invariant and an index for hyperbonodes were introduced in \cite{Uribeinvariant}. 
\end{remark*}}

\subsection{Some Relevant Properties of Smooth Surfaces}\label{subsect:relevant-properties}

An {\em asymptotic curve} is an integral curve of a field of asymptotic directions.
\medskip

\noindent
\textit{\textbf{Left-Right}}.
Fix an orientation in $\RP^3$ (or $\R^3$). A regular smooth curve is said to be 
\textit{left} (\textit{right}) at the points where its first three derivatives form a negative (resp. positive)
frame. {\em At a hyperbolic point of a generic smooth surface one asymptotic curve is left and the other
is right} (cf. \cite{Uribegodron})\,: one of them twists like a left screw and the other like a right screw. 
Their respective asymptotic tangent lines are called {\em left} and {\em right asymptotic tangents}. 

\begin{remark*} 
At an inflection of an asymptotic curve the above frame is not defined. But, since 
the osculating plane to an asymptotic curve of a generic surface $S$ at a non inflection $q$ is 
the tangent plane to $S$ at $q$; 
for an arbitrary surface $S$ we say that \textit{an asymptotic curve $\g$ is left} (\textit{right}) 
at $q\in\g$ if the tangent plane to $S$ along $\g$ twists negatively (resp. positively) at $q$: 
if $T:=\g'$ and $N\in T_pS$ is a vector orthogonal to $T$, the frame 
$(T,N,N')$ is negative (resp. positive) at $q$. This enables, for example, to distinguish the two families 
of straight lines that foliate a one-sheet hyperboloid (its asymptotic curves). 
\end{remark*}

\noindent
\textit{\textbf{Left and Right Flecnodal Curve}}. 
The {\em left} ({\em right}) {\em flecnodal curve} $F_\ell$ (resp. $F_r$) of a surface 
$S$ consists of the points of the flecnodal curve of $S$ at which the over-osculating tangent 
line is a left (resp. right) asymptotic line.  

\begin{remark*}
An observer ignoring this distinction will be unable to understand
why the points of transverse self intersection of $F$ are sometimes stable and sometimes unstable, and 
why the transitions of $F$ in Fig.\,\ref{fig:two stable transitions} are both stable. 
\begin{figure}[ht]
\centerline{\psfig{figure=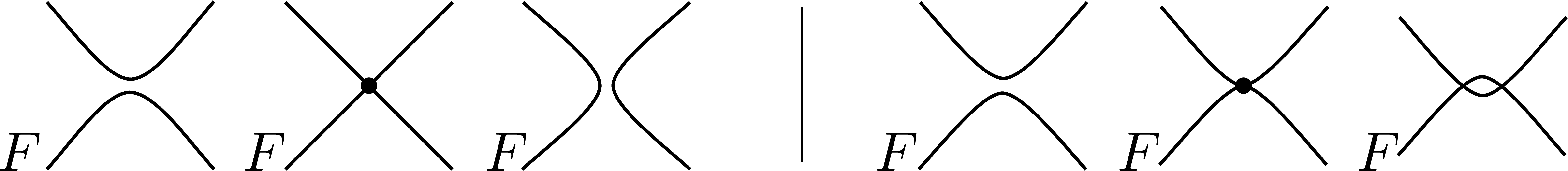,height=1.1cm}}
\caption{\small Two stable transitions of the flecnodal curve.}
\label{fig:two stable transitions}
\end{figure}
\end{remark*}

Presenting $S$ in Monge form and using the description of the asymptotic curves in terms of IDEs 
(\S\ref{flecnodal-IDE}), Corollary\,\ref{cor:bi-inflecrion} provides the 
\medskip 

\noindent
{\bf Biflecnode's Characterisation}. 
{\em A point of a smooth surface $S$ in $\RP^3$ {\rm (or $\R^3$)} is a left biflecnode of $S$
iff at that point the left flecnodal curve 
is tangent to the left asymptotic direction} (the same holds for the right biflecnodes). 
\medskip 

\noindent
{\bf Godron-Swallowtail}.
A godron is also defined in terms of singularities of the contact with the tangent plane\, 
-- materialised as singularities of the dual surface. An ordinary godron corresponds to a
swallowtail point of the dual surface (an $A_3$ Legendre singularity - cf.\,\cite{avg}).
A double godron is unstable and corresponds to a transition where 
two swallowtails are born or die (an $A_4$ Legendre singularity).
A godron is {\em simple} if it corresponds to a swallowtail of the dual surface.
{\em All godrons of a surface in general position are simple.} 
\smallskip

\noindent
\textit{\textbf{Index of a Godron}}.  
A simple godron is said to be {\em positive} or {\em of index $+1$} 
(resp.~{\em negative} or {\em of index $-1$}) if, at the neighbouring parabolic 
points, the half-asymptotic lines directed to the hyperbolic domain point 
towards (resp.~away from) the godron (Fig.\,\ref{fig:index of a godron}).
\begin{figure}[ht]
\centerline{\psfig{figure=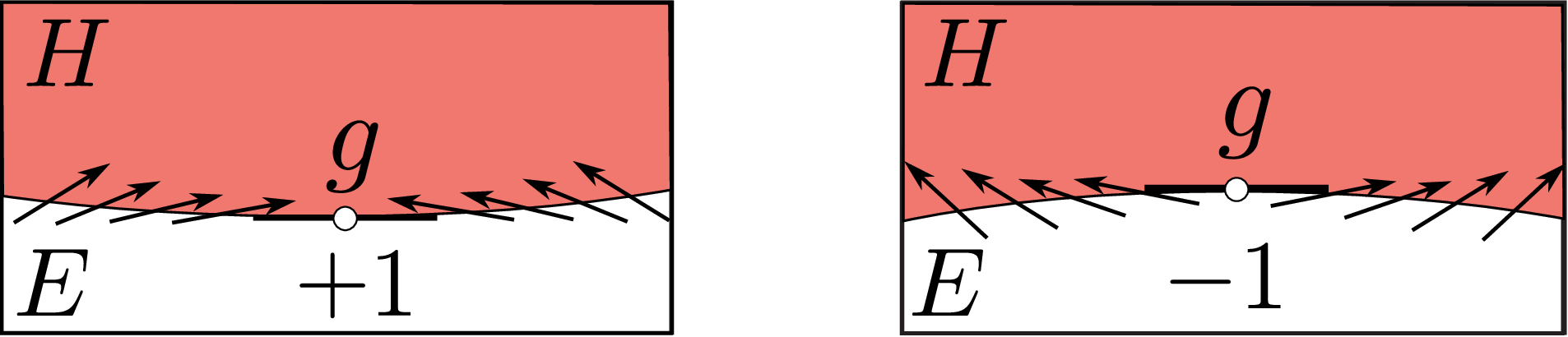,height=1.3cm}}
\caption{\small The index of a godron.}
\label{fig:index of a godron}
\end{figure}

Five characterisations and several geometric properties of positive and negative godrons 
and swallowtails are described in \cite{Uribegodron}.
 \label{index-page}

\begin{cremark}
At each elliptic point, the surface locally lies in one of the two half-spaces 
determined by the tangent plane. This half-space, 
called {\em positive}, determines a {\em natural co-orientation} 
on each connected component of the elliptic domain. By continuity we get a 
positive half-space on the smooth part of the parabolic curve \cite{Uribegodron}. 

Let $q$ be a parabolic point of a generic smooth surface $S$. 
Take an affine coordinate system $x,y,z$ at $q$ such that the $xy$-plane is 
tangent to $S$ and the $x$-axis is tangent to the parabolic curve at $q$.
Direct the positive $z$-axis to the positive half-space at $q$ and the positive $y$-axis
to the hyperbolic domain.
Finally, direct the positive $x$-axis in such way that any basis $(e_x,e_y,e_z)$ of 
$x,y,z$ forms a positive frame for the chosen orientation of $\R^3$ (or of $\RP^3$).
\smallskip

The vector $\,e_x\,$ determines a\, {\em  natural orientation} \,of the parabolic curve. 
\end{cremark}

Of course, the natural orientation fits with the local
behaviour of the parabolic curve at the $A_3$ and $D_4$ transitions (Fig.\,\ref{fig:orient-parabolic-curve})
\begin{figure}[ht]
\centerline{\psfig{figure=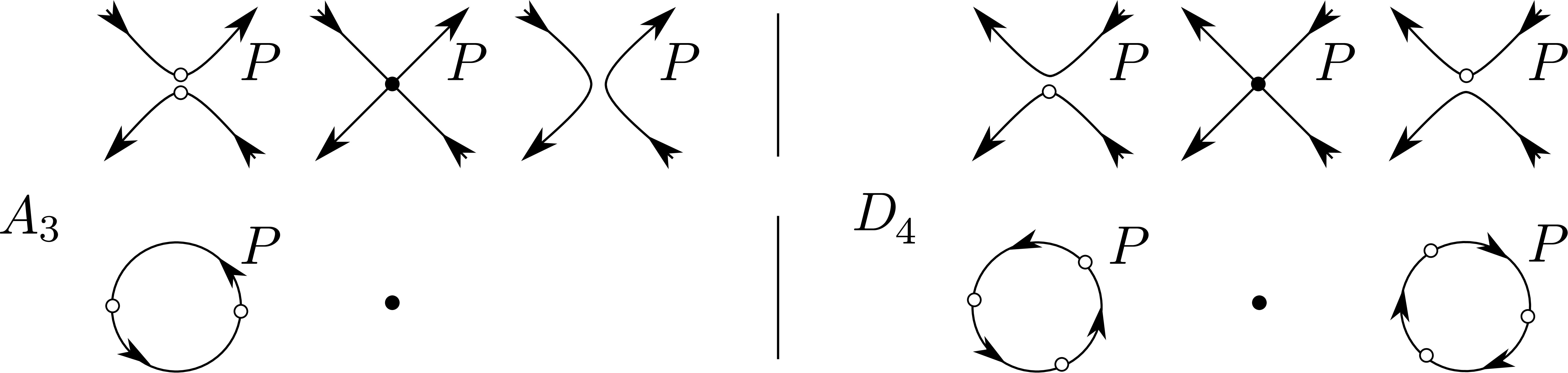,height=2.75cm}}
\caption{\small Natural orientation of the parabolic curve.}
\label{fig:orient-parabolic-curve}
\end{figure}


\section{Transitions of flecnodal and parabolic curves}\label{section:bifurcation-theorems}
\subsection{The Five Main Theorems on Evolving Surfaces}\label{Main-theorems-on-surfaces}

{\footnotesize
Theorems \ref{th:bigodron&flecgodron} to \ref{theorem:flat-umbilics} describe the local transitions 
of the flecnodal and parabolic curves and tangential singularities occurring in generic evolving surfaces.} 
\medskip

\noindent
\textbf{Monge form}. We 
present the germs of surfaces in $\R^3$ at the origin in Monge form \,$z=f(x,y)$\, with\,$f(0,0)=0$\, and \,$df(0,0)=0$, 
expressing the partial derivatives of $f$ with numerical subscripts\,:
\[f_{ij}(x,y):=\frac{\D^{i+j}f}{\D x^i\D y^j}(x,y) \quad \text{and} \quad f_{ij}:=f_{ij}(0,0)\,.\]
Suppose $S_\e$ is a generic $1$-parameter family of smooth surfaces in $\RP^3$ (or $\R^3$) depending 
smoothly on the parameter $\e$. It will be given in Monge form $z=f^\e(x,y)$.
Below, all conditions for $f=f^0$ are given at $(x,y)=(0,0)$. 

\begin{theorem}\label{th:bigodron&flecgodron}
Suppose $S_0$ has a godron $g$ and $P$ is smooth at $g$. 

\noindent
a) If $P$ and $F$ have $4$-point contact at $g$, then we get a bigodron transition ($A_4$, equivalent to $UW$)\,: 
two godrons of opposite index are born or die.
\smallskip

\noindent
b) If $S_0$ has $5$-point contact with the asymptotic tangent at $g$, then $g$ has index $-1$
and we get a flec-godron transition\,: a biflecnode, which passes from one local branch of the flecnodal 
curve to the other, overlaps with $g$ at $\e=0$. 
\end{theorem}

\begin{note}
To get a godron we need  $f_{20}=f_{11}=f_{30}=0$ with  $f_{02}f_{21}f_{40}\neq 0$. \\
$a)$ the equality $3f^2_{21}-f_{02}f_{40}=0$ guarantees a double godron and the condition
$\,9f_{21}f_{31}-4f_{12}f_{40}-f_{02}f_{50}\neq 0$ guarantees that it is not triple. \\
$b)$ conditions $f_{40}=0$ and $f_{50}\neq 0$ guarantee the godron is also a biflecnode.  
\end{note}

\begin{example*}
For an ordinary godron Platonova's normal form of the $4$-jet of the surface is \,$f(x,y)=y^2/2-x^2y+\rho x^4/2$\,
with $\rho\neq 1, 0$.

If $\rho=1$ (i.e. $3f^2_{21}-f_{02}f_{40}=0$) this normal form provides 
an infinitely degenerate godron\,: $f(x,y)=\frac{1}{2}(y-x^2)^2$. To get just 
a bigodron it suffices to add a multiple of $x^3y$, by our second condition. The bigodron and 
flec-godron transitions, in generic $1$-parameter families, were studied in \cite{Uribegodron}.
\end{example*}

The $\g v$-transition of Fig\,\ref{inflections}, which depicts item $c$ of Theorems\,\ref{special-tangent} 
and \ref{th:bifurcation-withPsmooth}, is absent from Fig.\,\ref{14perestroikas}. 
This absence is explained by the following result: 
\medskip

\noindent
\textbf{No-Cusp Theorem}. 
{\em The flecnodal curve of a smooth surface in $\RP^3$} (or in $\R^3$)
{\em has never a cusp at a godron at which the parabolic curve is smooth}.
\medskip

This fact imposes topological restrictions on the possible configurations of the flecnodal
curve of a smooth surface. For example, 
\medskip

\noindent
{\bf Theorem} \cite{Uribegodron}. 
{\em Inside a hyperbolic disc bounded by a closed parabolic curve there is 
an odd number of hyperbonodes} (at least one).

\begin{theorem}\label{A_3-bifurcations}
If $S_0$ has a point $q$ at which $P$ has a Morse singularity 
with a unique asymptotic direction in general position 
{\rm (transverse to the branches of $P$)}, then the flecnodal curve has an $A_3$ transition at $q$. 
We have four cases: 
\smallskip

\noindent
$a)$ $q$ is an isolated parabolic point inside a hyperbolic region. An elliptic island is born or disappears. 

\noindent
$b)$ and $c)$ $q$ is a crossing of two branches of $P$ at which the asymptotic line is pointing, 
respectively, to the hyperbolic sectors and to the elliptic sectors. Two locally disjoint hyperbolic 
(elliptic) regions merge, while an elliptic (hyperbolic) region is separated into two locally disjoint regions. 

\noindent
$d)$ $q$ is an isolated parabolic point inside an elliptic region.  A hyperbolic island is born or disappears. 

In the four cases, two godrons are born or die, and in ``surface-time'' $3$-space, $\{S_\e\times\{\e\}\}$, 
the flecnodal curves of the surfaces $S_\e$ form a Whitney umbrella whose self-intersection line consists of hyperbonodes 
and whose handle consists of ellipnodes. 
The section $\e=0$ is tangent to the umbrella at its pinch-point, giving an $A_3$ curve singularity 
{\rm (see Appendix\,\ref{pinch-point} and Fig.\,\ref{umbrella})}.   
\end{theorem}

\begin{note}\label{note:conic-points}
To take the unique asymptotic tangent as $x$-axis imposes $f_{20}=0$. 
From eq. of $P$, $\Delta(x,y):=(f_{20}f_{02}-f_{11}^2)(x,y)=0$, at $(0,0)$ we get $f_{11}=0$ and $f_{02}\neq 0$.  
The condition that $(0,0)$ is a critical point of $\Delta(x,y)$,  
\smallskip

\centerline{$f_{20}f_{12}+f_{30}f_{02}-2f_{11}f_{21}=0$ \ and \  $f_{20}f_{03}+f_{21}f_{02}-2f_{11}f_{12}=0$,} 
\smallskip

\noindent
implies $f_{30}=f_{21}=0$. It is of Morse type if 
$f_{40}(f_{22}f_{02}-2f_{12}^2)-f_{02}f_{31}^2\neq 0$. 

Finally, the $x$-axis is transverse to both branches of $P$ if $f_{40}\neq 0$.  
\end{note}

\begin{figure}[ht]
\centerline{\psfig{figure=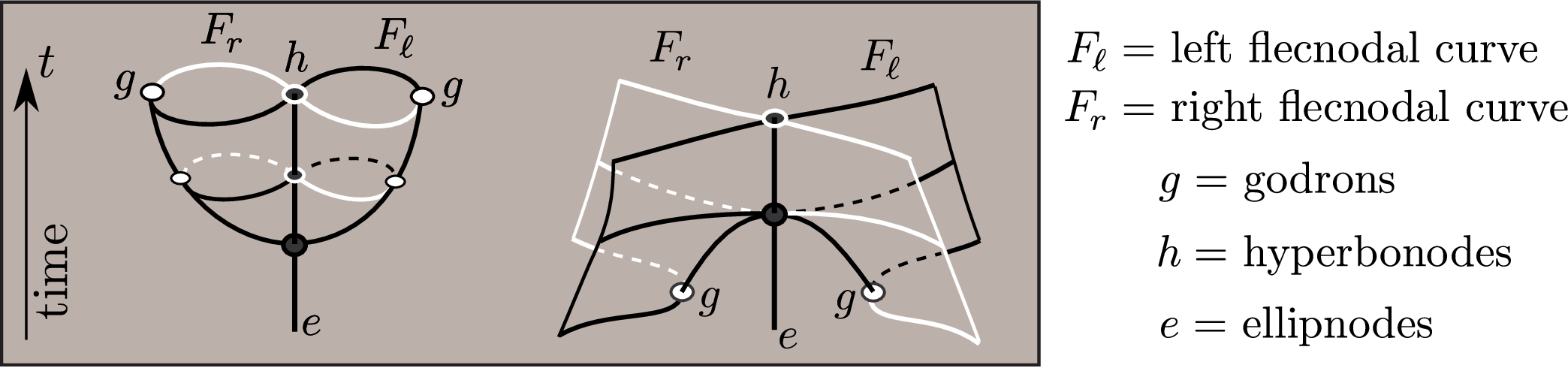,height=2.5cm}}
\caption{\small Whitney umbrellas formed by the flecnodal curve in ``surface-time'' $3$-space.}
\label{umbrella}
\end{figure}

\begin{theorem}\label{theorem:lips_bec-a-bec}
$a)$ If $S_0$ has a point $q$ at which $F_\ell$ has a Morse singularity 
and the left asymptotic line is transverse to both branches of $F_\ell$ {\rm (if they are real)}, 
then $q$ is a degenerate left biflecnode. We get a lips transition if $F_\ell$ is an isolated point, 
and a bec-à-bec transition if $F_\ell$ has a local crossing {\rm (Fig.\,\ref{14perestroikas})}. 
In both cases two biflecnodes are born or die. {\rm (The same holds for $F_r$.)}
\smallskip

\noindent 
$b)$ If $q\in S_0$ is a left triflecnode {\rm ($6$-point contact with the asymptotic tangent)} and 
$F_\ell$ is smooth at $q$, then we get a ``swallowtail'' transition - two left   
biflecnodes are born or die, collapsing like the maximum and minimum values of the functions 
$f_\e(x)=x^3+\e x$, at $\e=0$ - {\rm Fig.\,\ref{14perestroikas}}. 
{\rm (The same holds for $F_r$.)}
\end{theorem}

\begin{note}\label{nota3}
The asymptotic lines are coordinate axes iff $f_{20}=f_{02}=0$, $f_{11}\neq 0$. \\
a) The condition $f_{50}\neq 0$ guarantees a simple biflecnode. 

\noindent
Write $A:=f_{50}f_{11}$, $\,B:=f_{41}f_{11}-2f_{31}f_{21}$, 
$\,C:=3f_{31}f_{12}-\frac{9}{2}f_{22}f_{21}+f_{32}f_{11}$. 
The equality $3f^2_{21}-2f_{31}f_{11}=0$ implies the flecnodal curve of $S=S_0$ has a singularity 
at the origin, which is of Morse type if $AC-B^2\neq 0$.  \\
b) The condition $\,3f^2_{21}-2f_{31}f_{11}\neq 0$ guarantees the flecnodal curve is smooth at the origin; 
and the conditions $f_{30}=f_{40}=f_{50}=0$ with $f_{60}\neq 0$ guarantee the origin is a triflecnode
but not a quadriflecnode. 
\end{note}

\begin{example*}
  By Theorem\,\ref{theorem:lips_bec-a-bec}\,$a$, the flecnodal curve of the surface given by 
\begin{equation}\label{example_Morse-sing}
f(x,y)=xy+x^5+y^3+\a\,x^2y+\a^2\,x^3y+\g\,x^3y^2\,
\end{equation}
 has a Morse singularity at the origin if $\a\neq 0$ and $\g\neq \frac{2}{5}\a^6$ because $AC-B^2=2\cdot 6!\g-24^2\a^6$. 
 But if $\g=\frac{2}{5}\a^6$, the flecnodal curve has a semi cubic cusp. 
\end{example*}


\begin{theorem}\label{th:ellipnodes-hyperbonodes}
$a)$ If $F_\ell$ and $F_r$ are tangent at $q\in S_0$, then at $\e=0$ we get a creation/annihilation transition 
of two hyperbonodes. 
\smallskip

\noindent
$b)$ If $q\in S_0$ is a hyperbolic point at which the asymptotic tangents have $4$- and $5$-point contact 
with $S_0$, then we get a ``flec-hyperbonode'' transition$:$ 
a left biflecnode moves along $F_\ell$ and, at $\e=0$, crosses $F_r$ overlapping with a hyperbonode. 
{\rm ($F_\ell$ and $F_r$ may be interchanged.)} \par
\smallskip

\noindent 
$c)$ If two complex conjugate branches of $F$ are tangent at $q\in S_0$, 
then at $\e=0$ we get a creation/annihilation transition of two ellipnodes.  
%
\end{theorem}

\begin{note}
If we take the asymptotic lines as coordinate axes, the condition
$f_{30}=f_{03}=0$ guarantees the origin is a hyperbonode. Then \\
a) the condition $4f_{40}f_{04}f_{11}-(3f^2_{21}-2f_{31}f_{11})(3f^2_{12}-2f_{13}f_{11})=0$
guarantees the tangency of $F_\ell$ and $F_r$. \\
b) the condition $(3f^2_{21}-2f_{31}f_{11})(3f^2_{12}-2f_{13}f_{11})\neq 0$ guarantees the 
transversality of $F_\ell$ and $F_r$, and $f_{40}=0$, $f_{50}\neq 0$ 
guarantee our point is a biflecnode which is not triflecnode. \\
c) At an ellipnode we can choose affine coordinates such that $f_{20}=f_{02}$ and $f_{30}=f_{03}=f_{21}=f_{12}=0$. 
In such coordinate system, the equality $(f_{40}-3f_{22})(f_{04}-3f_{22})-(f_{31}-3f_{13})(f_{13}-3f_{31})=0$ 
guarantees the tangency of two complex conjugate branches of the flecnodal curve at $q=(0,0)$.
\end{note}

\begin{example*}
The respective normal forms for hyperbonodes of Landis-Platonova 
\cite{Landis, Platonova} and Tabachnikov-Ovsienko \cite{Ovsienko-Tabachnikov} are the following:
\[f(x,y)=xy\pm x^4+\a x^3y+\b xy^3+y^4\,,\]
\[f(x,y)=xy\pm Ix^4+x^3y+xy^3+Jy^4\,.\]
Both satisfy $f_{21}=f_{12}=0$. So, to get the tangency of $F_\ell$ et $F_r$ we need to have 
$f_{40}f_{04}-f_{31}f_{13}=0$. We respectively get $\a\b=\pm 16$ and $\pm 16IJ= 1$. 
\end{example*}

\noindent 
\textit{\textbf{Flat Umbilics}}. A point $q$ at which the quadratic part of $f$ is zero and has 
nondegenerate cubic form is called {\em nondegenerate flat umbilic}.
The projective dual surface at $q$ has a $D_4$ front singularity. All tangent lines at $q$ 
are asymptotic, but either one or three of those lines have $4$-point contact with $S$ at $q$  
because the cubic form vanishes along either one line ($D_4^+$) or three lines ($D_4^-$). 
Thus we have either one or three branches of $F$ through $q$ (Fig.\,\ref{ricD4+}). 

Clearly, a vector $\vec{u}=(u_1,u_2)$ belongs to the kernel of the cubic form of $f$ 
iff $\D^3_{\vec{u}}f=0$, where $\D_{\vec{u}}:=u_1\D_x+u_2\D_y$. Thus in the $D_4^-$ case there are 
three vectors $\vec{u}$, $\vec{v}$, $\vec{w}$ satisfying $\D^3_{\vec{u}}f=\D^3_{\vec{v}}f=\D^3_{\vec{w}}f=0$. 

\begin{theorem}\label{theorem:flat-umbilics}
Suppose $S_0$ has a non degenerate flat umbilic at $q$ and no asymptotic line has 
$5$-point contact with $S_0$. Then 

\noindent   
a) If $S_0$ has three asymptotic tangents with $4$-point contact, then $S_0$ has locally three transverse 
flecnodal curves and we get a $D_4^-$ transition$:$ an elliptic disc with tree negative godrons 
on $P$ and three neighbouring hyperbonodes collapse and reappear $($the reappeared disc 
has opposite natural co-orientation.$)$

\noindent 
b) If $S_0$ has one asymptotic tangent with $4$-point contact, then we get a $D_4^+$ transition$:$ 
{\rm (say $\e<0$)} two elliptic regions of opposite co-orientation approach each other; {\rm ($\e=0$)} one branch of $F$ and 
two branches of $P$ intersect transversely at $q$; {\rm ($\e>0$)} the elliptic regions do not merge, but repulse each other. 
In the process, an ellipnode passes from one elliptic domain to the other and a negative godron passes from one parabolic curve to the other.
\end{theorem}

\begin{note}
a) Let $\vec{u}$, $\vec{v}$, $\vec{w}$ be the vectors satisfying $\D^3_{\vec{u}}f=\D^3_{\vec{v}}f=\D^3_{\vec{w}}f=0$. 
The condition $\D^4_{\vec{u}}f\,\D^4_{\vec{v}}f\,\D^4_{\vec{w}}f\neq 0$ ensures each over-osculating asymptotic
line has exactly $4$-point contact with $S_o$: this bans the local presence of biflecnodes.
The analogue genericity condition in b) is $\D^4_{\vec{u}}\neq 0$. 
\end{note}

\noindent 
{\bf Degenerate flat umbilics}.
A flat umbilic $D_4^+$ or $D_4^-$ satisfying $\D^4_{\vec{u}}f=0$ (for some vector $\vec{u}$) 
is not stable in generic $1$-parameter families of surfaces but is stable in $2$-parameter ones 
if $\D^5_{\vec{u}}f\neq 0$ - for small values of the parameter
the surfaces have a biflecnode on the flecnodal curve corresponding to the $\vec{u}$-direction. 
Similarly, a flat umbilic $D_4^-$ with $\D^4_{\vec{u}}f=\D^4_{\vec{v}}f=0$ is not stable in generic $2$-parameter
families of surfaces but is stable in $3$-parameter ones - two flecnodal curves
of the surfaces have one biflecnode. 

\begin{example*}
The $D_4^-$ flat umbilic point of the surface given in Monge form by
\begin{equation}\label{normal-formP1f}
f(x,y)=x^3-xy^2+\a x^3y+\b y^4+\psi_{\geq 5}(x,y)
\end{equation}
is not stable in generic $1$-parameter families of smooth surfaces if  
$\D_y^4f=0$, $(\D_x+\D_y)^4f= 0$ or $(\D_x-\D_y)^4f=0$, that is, if 
$\b=0$, $\b=-\a$ or $\b=\a$.

So the normal form\,\eqref{normal-formP1f} in Table 2 of \cite{Toru} should be accompanied 
by the restrictions $\b\neq 0$, $\b\neq-\a$ and $\b\neq\a$. 
Similarly, one has to impose $\b\neq 0$ to the $D_4^+$ normal form of \cite{Toru}.
\end{example*}

\subsection{Further Results and Comments}\label{further results}
\textbf{Extension to Surfaces in other Spaces}. 
Since the tangential singularities depend only on the contact of the surface 
with its tangent lines, all our results hold for surfaces in the $3$-sphere 
$\sph^3$; in Lobachevsky $3$-space $\Lambda^3$; in de Sitter world, etc.
Namely, $\sph^3$ is the double covering of $\RP^3$; the lines of the Klein model of $\Lambda^3$ are 
lines of the affine space $\R^3$; etc. 

\subsubsection{First Degenerate ``Conic'' Points of the surface $a^f=0$}\label{sect:degenerate-conic-points}     
Consider a surface $z=f(x,y)$ such that the function $a^f$ has a critical
point at $\bar 0$ with $a^f(\bar 0)=0$. Assume that $f_{40}\neq 0$ and $f_{02}=1$, and write 
{\small\[\Delta=3f_{40}\left[2f_{12}(f_{12}f_{03}f_{40}^2+3f_{12}^2f_{40}f_{31}+f_{31}^3-f_{40}^2f_{13})+
f_{31}(f_{31}f_{41}-f_{40}f_{32})\right]+f_{31}^3f_{50}+f_{40}^3f_{23}.\]}

\noindent
\textbf{Theorem\,\ref{sect:degenerate-conic-points}}.\, 
\textit{The parabolic curve $P$ and the flecnodal curve $F$ of the surface $z=f(x,y)$ have the singularities 
shown in Table\,\ref{tb:degenerate-conic-points} at $(0,0)$}: 

\begin{table}[h!]
\centering
\begin{tabular}{|ccc|l|c|}\hline
$P$ & $F$ & picture & conditions & codim\\ \hline \hline
$A_1$ & $A_3$ & $\raisebox{-.4\height}{\includegraphics[scale=0.28]{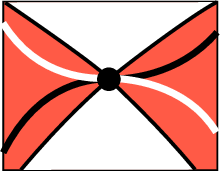}}\,$ & $f_{40}(f_{22}-2f_{12}^2)-f_{31}^2\neq 0$. & $3$ \\ \hline 
$A_2$& $A_6$  & $\raisebox{-.4\height}{\includegraphics[scale=0.28]{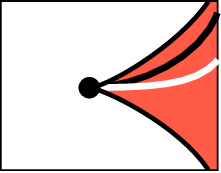}}\,$   & $f_{40}(f_{22}-2f_{12}^2)-f_{31}^2=0$, \, $\Delta\neq 0$. & $4$ \\ \hline 
$A_3$ & $A_9$  & $\raisebox{-.4\height}{\includegraphics[scale=0.28]{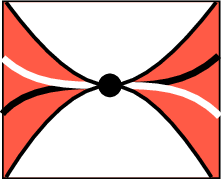}}\,$ & $f_{40}(f_{22}-2f_{12}^2)-f_{31}^2=0$, \, $\Delta=0$. & $5$ \\ \hline 
  \end{tabular}.
\caption{\small A Morse conic singularity and the simplest non Morse conic singularities.}
\label{tb:degenerate-conic-points}
\end{table}

\begin{proof}[On the proof]
The first case is part of Theorem\,\ref{A_3-bifurcations}. 
In the second case, we make a suitable linear change of coordinates so that the 
relevant terms of the Newton diagram of the Hessian polynomial are 
$H_f(x,y)=x^2+\Delta y^3+\ldots$, giving an $A_2$ singularity for the parabolic curve - 
the $A_6$ singularity of the flecnodal curve is obtained by a calculation and using 
the Newton diagram of the polynomial $I^f(x,y)$. 
The condition $\Delta=0$, of third case, clearly provides an $A_3$ 
singularity for the parabolic curve, and the singularity $A_9$ of the flecnodal 
curve is obtained similarly as we did in the second case. 
\end{proof}

\begin{remark*}
In table\,\ref{tb:degenerate-conic-points}, we made one picture for each case, but the $1$st 
and $3$rd cases have four real realisations and the $2$nd case has two real realisations. 
\end{remark*}

\begin{example*}
The surface whose local Monge form is 
\[f(x,y)=y^2/2\pm x^4+\a x^3y+(2\b^2\pm \frac{3}{8}\a^2)x^2y^2+\b xy^2\,,\]
satisfies $f_{40}(f_{22}f_{02}-2f_{12}^2)-f_{02}f_{31}^2=0$ (see Note\,\ref{note:conic-points}). So
if $\a\b(\a^2\pm 8\b^2)\neq 0$ (i.e. $\Delta\neq 0$) the parabolic and flecnodal 
curves of the surface have, respectively, the singularities $A_2$ (a cusp) 
and $A_6$ at $\bar 0$.
If $\b=0$ we get the normal form
$f(x,y)=y^2/2\pm x^4 +\a x^3y\pm \frac{3}{8}\a^2x^2y^2$ of \cite{Toru} 
whose parabolic and flecnodal curves have the respective singularities $A_3$ 
and $A_9$ because $\Delta=0$. 
\end{example*} 

\subsubsection{Comparing the transitions of Figures \ref{14perestroikas} and \ref{inflections}}\label{comparing-the-lists}

On one hand, the curve of inflections of a BIDE $F(x,y,p)=0$ is a particular case of the curve of (generalised) 
inflections of the characteristic fronts of a surface $\V\subset M$ in a contact $3$-manifold $M$. 
On the other hand, the flecnodal curve of a surface $z=f(x,y)$ leads to study the curve of
inflections of the asymptotic IDE $a^f=0$, which is a very special class of BIDE (of infinite codimension). 
So, we have three families of objects summarised in Table\,\ref{tb:three-similar problems}.
{\small 
\begin{table}[h!]
\centering
\begin{tabular}{|c|c|c|}\hline
Surface in $\RP^3$ ($\R^3$) & IDE $\,F(x,y,p)=0$ and $\V^F$  & $\V\subset M$, $\rho$, $\rho^\vee$ \\ \hline \hline 
parabolic curve              & discriminant                     & $\rho$-discriminant \\ \hline 
asymptotic curve             & solution of the IDE $F=0$        & characteristic front  \\ \hline 
flecnodal curve              & curve of inflections             & curve of inflections \\ \hline 
godron                       & folded singular point            & $\rho$-folded singular point \\ \hline 
hyperbonode                  & hyperbolic node                  & hyperbolic node \\ \hline 
ellipnode                    & (elliptic node)                  & (elliptic node) \\ \hline 
biflecnode                   & bi-inflection                     & bi-inflection \\ \hline 
\end{tabular}\ .
\caption{\small Asymptotic curves $|$ Solutions of IDEs $|$ Characteristic fronts.}
\label{tb:three-similar problems}
\end{table}}

\noindent
\textbf{$A_3$ vs ($a$ to $d$) transitions}. The $A_3$ transitions in Fig.\,\ref{14perestroikas} are analogues to the transitions 
$a$ to $d$ in Fig.\,\ref{inflections}: in both, the surface of the BIDE has a Morse 
singularity. But in cases $a$ to $d$,  
the curve of inflections has a Morse transition, while in the $A_3$ cases 
the flecnodal curve undergoes an $A_3$ transition, as a family of sections of a Whitney umbrella one of 
which is tangent to the umbrella at its pinch point (see Appendix B and Fig.\,\ref{umbrella}). 
\medskip

\noindent
\textbf{Infinite Degeneracy of the asymptotic IDE at $D_4$ transitions}. 
If a surface in $\RP^3$ has a flat umbilic point $q$ (a $D_4$-singularity), then  
the surface $\A_0\subset J^1(\R,\R)$ of its asymptotic IDE contains the whole 
fibre of $\pi : J^1(\R,\R)\ra J^0(\R,\R)$ over $q$ (see Fig.\,\ref{ricD4+}). 
Thus, the IDE defined by $\A_0$ is not binary and has infinite codimension in the space of IDEs. 
This explains the absence of $D_4$ transitions in Fig.\,\ref{inflections}

\begin{figure}[ht]
\centerline{\psfig{figure=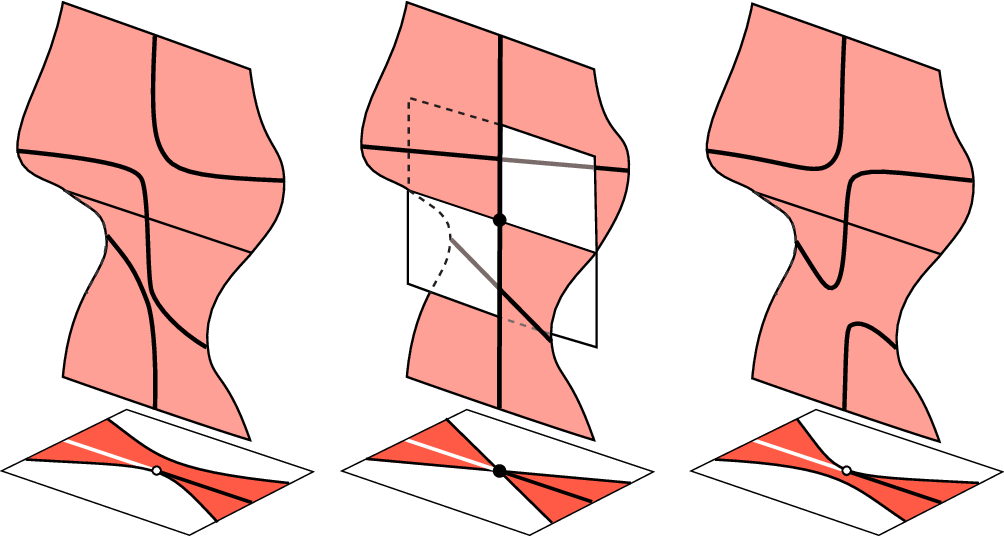,height=4cm}\ \ \ \ \psfig{figure=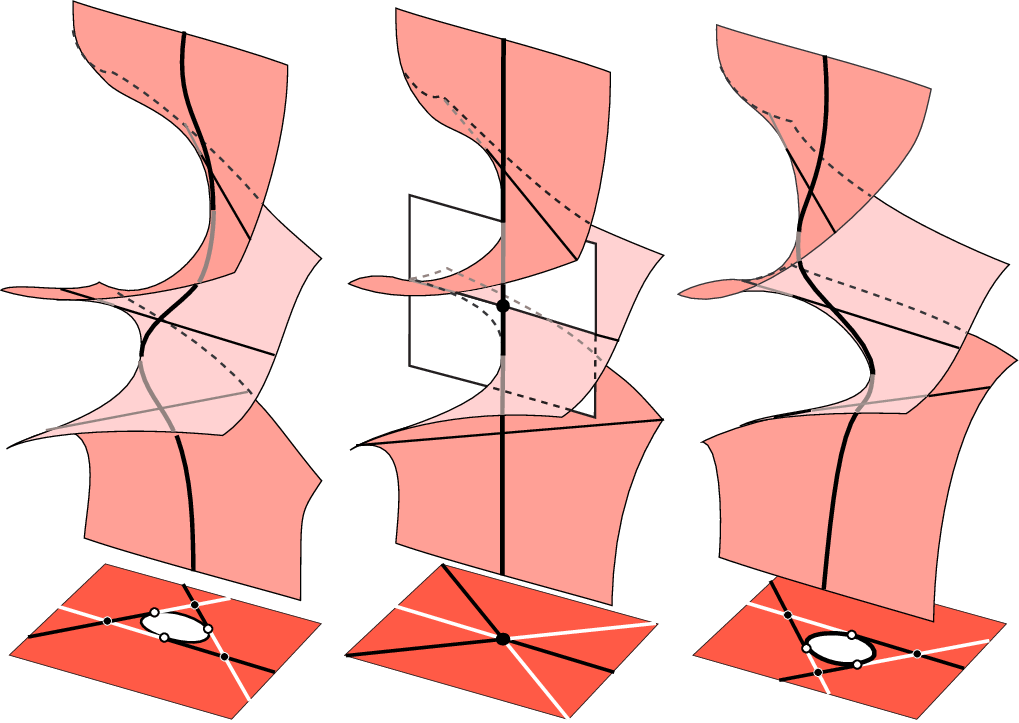,height=4cm}}
\caption{\small Evolution of the projection $\pi:\A^f\rightarrow \R^2$ for the $D_4^{+}$ 
and $D_4^{-}$ perestroikas.}
\label{ricD4+}
\end{figure}

\section{From Surfaces to IDEs: Preparatory Results}\label{preparatory-results}

\subsection{Asymptotic Double and Asymptotic IDE}\label{flecnodal-IDE}

\noindent
\textit{\textbf{Contact Elements}}.
Let $N$ be an $n$-dimensional smooth manifold. 
A {\em contact element} at a point of $N$ is a vectorial hyperplane (of co-dimension $1$) 
of the tangent space of $N$ at that point (called the {\em point of contact}). The set of 
all contact elements of $N$ is a ($2n-1$)-dimensional manifold (noted by $PT^*N$) endowed 
with a natural contact structure.
\smallskip

\noindent
\textit{\textbf{Asymptotic double}}.
Now, consider a generic smooth surface $S$ in $\RP^3$. The {\em asymptotic-double} 
$\A$ of $S$ is the (smooth) surface of $PT^*S$ consisting of the asymptotic directions 
along $S$\,: at each hyperbolic point, the asymptotic lines determine two contact 
elements, while at the parabolic points there is only one asymptotic contact element 
(with multiplicity two). See Fig.\ref{doble}. 
\begin{figure}[ht]
\centerline{\psfig{figure=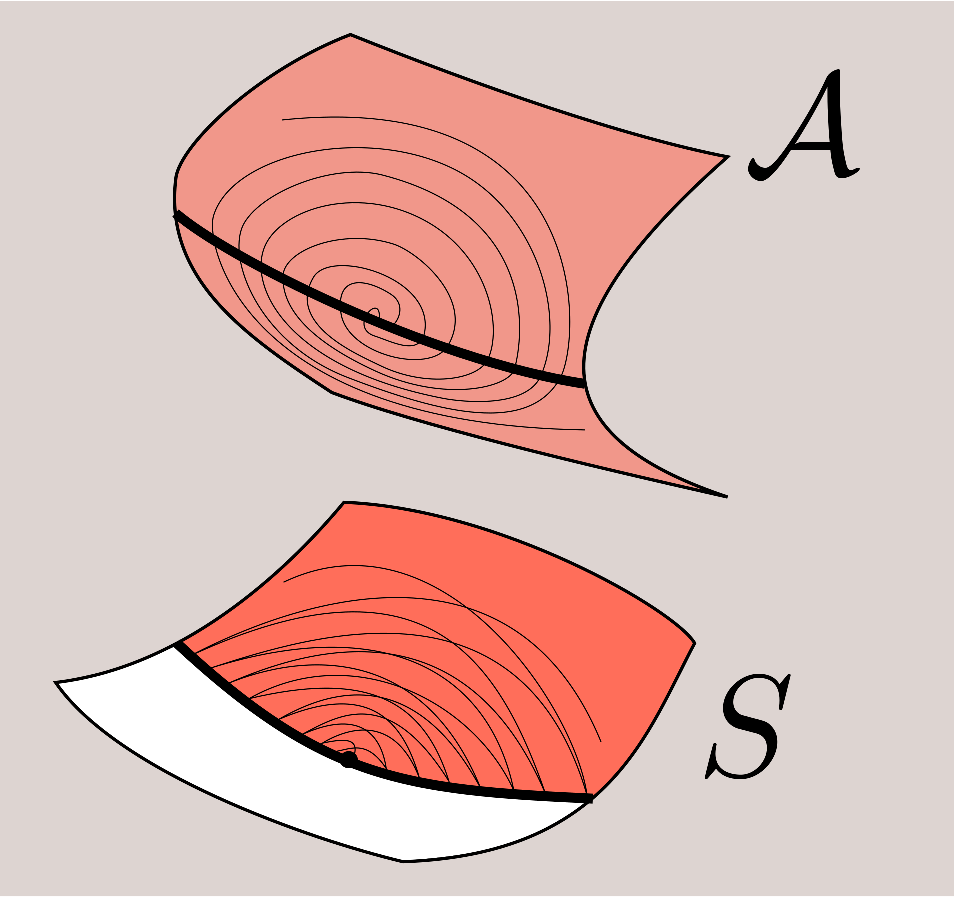,height=2.2cm}}
\caption{\small The asymptotic-double of a surface.}
\label{doble}
\end{figure}

So, under the natural projection $PT^*S\ra S$ (sending each contact element to its point
of contact) the asymptotic double $\A \subset PT^*S$ doubly covers the hyperbolic domain with 
a fold singularity over the parabolic curve. 

To study the flecnodal curve, we shall use the asymptotic double $\A$ of $S$. 

Consider a projection of $S$ to an affine plane $\R^2 \subset \RP^3$, from a 
point exterior to the projective plane that contains $\R^2$, 
$\pi : S\rightarrow \R^2\subset \RP^3$, and assume it is a local diffeomorphism 
at every point of $S$. 

On one hand, {\em any such projection sends (bijectively) the flecnodal curve of $S$ onto the 
curve of $\R^2$ formed by the inflections of the images (by $\pi$) of the 
asymptotic curves of $S$} (see Proposition~\ref{flecnodal-inflections} in \S\,\ref{subsect:flecnodal-inflections}).

On the other hand, {\em the derivative of $\pi$ sends the contact elements of $S$ 
to the contact elements of $\pi(S)\subset \R^2$ and it induces a contactomorphism 
$PT^*S \rightarrow PT^*(\pi(S))\subset PT^*\R^2$ sending $\A$ to a surface $\tilde\A\subset PT^*\R^2$ 
which we still call the asymptotic-double of $S$ and which doubly covers 
(under the natural projection $PT^*\R^2\rightarrow \R^2$) the image by $\pi$ in $\R^2$ 
of the hyperbolic domain.} We note it $\A^f$ when $S$ is given in Monge form $z=f(x,y)$.

Consider the surface $S$ as the graph of a function $z=f(x,y)$, and take the projection 
$(x,y,z)\rightarrow (x,y)$, along the $z$-axis. Then, the 
asymptotic-double $\A^f\subset PT^*\R^2$ is given 
by the contact elements of the $xy$-plane for which 
$$f_{xx}dx^2+2f_{xy}dxdy+f_{yy}dy^2=0.\eqno(*)$$

In order to make calculations, we take the space $J^1(\R,\R)$ 
(with coordinates $x,y,p$) as an `affine' chart of $PT^*\R^2$.
Namely, $J^1(\R,\R)$ parametrises all non vertical
contact elements of the $xy$-plane\,:
the contact element with slope $p_0\neq \infty$ at $(x_0,y_0)$ is represented by the 
point $(x_0,y_0,p_0)$ in $J^1(\R,\R)$. The asymptotic-double $\A^f$ 
is the surface in $J^1(\R,\R)$ given by the {\em asymptotic IDE}
$$a^f(x,y,p):=f_{xx}+2f_{xy}p+f_{yy}p^2=0,\eqno(a)$$
obtained from eq.~$(*)$ by taking $p=dy/dx$. 
{\em The asymptotic curves of $S$ are the images by $f$ of the solutions of the IDE $(a)$, whose
discriminant (in the $xy$-plane) corresponds to the parabolic curve of $S$}.




\subsection{Facts and Lemmas on IDEs $\,F(x,y,p)=0$}

Consider a smooth function $F:J^1(\R,\R)\rightarrow \R$, $(x,y,p)\mapsto F(x,y,p)$. 
\smallskip

\noindent
\textit{\textbf{Inflection Function}}. 
The {\em inflection function of $F$}, $I^F:J^1(\R,\R)\ra\R$, is defined by 
$$I^F(x,y,p):=F_x(x,y,p)+pF_y(x,y,p).$$ 
\textit{\textbf{Flec-Surface}}. The {\em flec-surface of $F$}, defined in \S\ref{sect:morse}, is the 
surface given by the equation $I^F(x,y,p)=0$ in the space $J^1(\R,\R)$.

\begin{lemma}\label{lemma 0}
The contour of the surface $\V^F$ by the projection $\pi^\vee$ {\rm (i.e. the  $\pi^\vee$-contour)} 
is the intersection of $\V^F$ with the flec-surface of the function $F$.
\end{lemma}

\begin{proof}
An easy computation shows 
(and it is well known, \cite{arnoldcs}) that the characteristic direction 
(\S\ref{characteristics}) is generated by the vector field 
$$(\dot{x},\dot{y},\dot{p})=(-F_p~,-pF_p~,F_x+pF_y).\eqno(*)$$ 
The direction field ($*$) is tangent to the $\pi^\vee$-fibre at a point 
of $\V^F$ iff this field is horizontal, i.e. iff $F_x+pF_y=0$. 
Thus the $\pi^\vee$-contour of $\V^F$ consists of the points 
of $\V^F$ at which  $I^F(x,y,p)=0$. 
\end{proof}

\begin{corollary*}
A point of the surface $\V^F$ is a characteristic point iff $d F\neq 0$, $F_p=0$ and $F_x+pF_y=0$
at this point.
\end{corollary*}

\begin{lemma}\label{Lemma A}
The origin belongs to the flec-surface of $F$ iff $F_x(\bar 0)=0.$
\end{lemma}

\begin{proof}
$I^F(\bar 0)=(F_x+pF_y)(\bar 0)=F_x(\bar 0)$.
\end{proof}

\begin{lemma}\label{Lemma B}
The flec-surface of $F$ has vertical tangent plane at $\bar{0}$ 
iff $$F_{xp}(\bar 0)+F_y(\bar 0)=0.$$ 
\end{lemma}

\begin{proof}
The tangent plane to the flec-surface 
is vertical at a point iff $I^F_p=0$ 
at that point, i.e. iff $F_{xp}+pF_{yp}+F_y=0$ at that point. 
\end{proof}

\begin{lemma}\label{Lemma C}
Suppose that the origin is a Morse critical point of $F$. 
The flec-surface of $F$ has vertical tangent plane at the origin 
iff $$F_{xp}(\bar 0)=0.$$
\end{lemma}

\begin{proof}
The statement follows from Lemma\,\ref{Lemma B}. 
\end{proof}

\begin{proposition}\label{prop:vertical-flecsurface}
Let $Q\in\V^F$ a Morse critical point of $F$. 
The flec-surface of $F$ has vertical tangent plane at $Q$ iff the tangents to the 
$\pi$ and $\pi^\vee$-fibres are conjugate diameters of the cone $K\subset T_QJ^1(\R,\R)$ determined by $F$. 
\end{proposition} 

\begin{proof}
At $Q=\bar{0}\in J^1(\R,\R)$ the contact plane $\Pi$ coincides with the 
$xp$-plane and the tangents to the $\pi$ and $\pi^\vee$-fibres coincide with the $x$- and $p$-axes, 
$\ell_x$ and $\ell_p$. So, a simple calculation shows that the cross-ratio of the $\pi$- and $\pi^\vee$-tangents, $\ell_x$, $\ell_p$, 
with the lines $\ell_1$, $\ell_2$, on which the contact plane $\Pi$ intersects the cone $K$ (being real 
or complex lines), is equal to 
\begin{equation}\label{eq:cross-ratio-formula}
 \left(\ell_x,\ell_p,\ell_1,\ell_2\right)=\dfrac{-F_{xp}+\sqrt{F_{xp}^2-F_{xx}F_{pp}}}{-F_{xp}-\sqrt{F_{xp}^2-F_{xx}F_{pp}}}\,. 
\end{equation}

The tangents to the $\pi$ and $\pi^\vee$-fibres at $\bar{0}$ are conjugate diameters of $K$ iff the cross-ratio 
$\left(\ell_x,\ell_p,\ell_1,\ell_2\right)$ equals $-1$ (\S\ref{section:cone-fibrations} (harmonicity)). 
This holds, by eq.\,\eqref{eq:cross-ratio-formula}, iff $F_{xp}(\bar{0})=0$. 
But, by Lemma\,\ref{Lemma C}, $F_{xp}(\bar{0})=0$ iff the tangent plane to the flec-surface of $F$ at $\bar{0}$ 
is vertical. 
\end{proof}

\subsection{Lemmas for the Asymptotic IDE $\,a^f(x,y,p)=0$}

\begin{lemma}\label{Lemma D}
The function $a^f$ satisfies the equalities 
$$a^f_{xp}(\bar 0)+a^f_y(\bar 0)=
\frac{3}{2}a^f_{xp}(\bar 0)=3a^f_y(\bar 0)=3f_{xxy}(0,0).$$
\end{lemma}

\begin{proof}
Direct, short and easy calculation (by hand). 
\end{proof}

\begin{lemma}\label{Lemma E}
The flec-surface of $a^f$ has vertical 
tangent plane at $\bar{0}$ iff 
$$a^f_y\left(\bar 0\right)=0.$$
\end{lemma}

\begin{proof}
Lemma\,\ref{Lemma B} (applied to $F=a^f$) and Lemma\,\ref{Lemma D} imply Lemma\,\ref{Lemma E}. 
\end{proof}

\begin{lemma}\label{Lemma F}
If the function $a^f$ has a critical point then 
its flec-surface has vertical tangent plane at that point.
\end{lemma}

\begin{proof}
Since we are supposing that the critical point is the origin, we have 
$a^f_x(\bar 0)=a^f_y(\bar 0)=a^f_p(\bar 0)=0$. 
By Lemma\,\ref{Lemma A} (for $F=a^f$), the point belongs to the flec-surface. 
Now Lemma\,\ref{Lemma F} follows from Lemma\,\ref{Lemma E}. 
\end{proof}

\begin{lemma}\label{Lemma G}
If the flec-surface of $a^f$ has vertical 
tangent plane at a point of the criminant then this point 
is critical for the function $a^f$.
\end{lemma}

\begin{proof}
Supposing our point is $\bar{0}$, we have $a^f_x(\bar 0)=0$ (by Lemma\,\ref{Lemma A})  
and $a^f_p(\bar 0)=0$ (because $\bar{0}$ is in the criminant). 
By Lemma\,\ref{Lemma E}, $a^f_y(\bar 0)=0.$ 
\end{proof}

\subsection{Flecnodal Curve and Curves of Inflections}\label{subsect:flecnodal-inflections}

\begin{proposition}\label{flecnodal-inflections}
Given a smooth surface $S$ in Monge form $z=f(x,y)$, the 
projection $(x,y,z)\mapsto (x,y)$ (along the $z$-axis) 
sends the flecnodal curve of $S$ onto the curve of inflections 
of the asymptotic IDE $\,a^f(x,y,p)=0$\,$:$ 
$$f_{xx}+2f_{xy}p+f_{yy}p^2=0. \eqno(a)$$
\end{proposition} 

Of course, the discriminant (in the $xy$-plane) of the IDE $(a)$ corresponds 
to the parabolic curve (in the $xyz$-space) of the surface $z=f(x,y)$.

\begin{proof}[Proof of Proposition \ref{flecnodal-inflections}] 
Consider a regularly parametrised curve on $S$
$$\g(t)=(\g_1(t), \g_2(t),f(\g_1(t), \g_2(t))),\ \ 
\g(0)=(0,0,0),\ \ \g'(0)\neq(0,0,0)\,.
$$
Its tangent line at $\g(0)$ is parametrised by $\ell(s)=\g(0)+s\g'(0)$. 
Consider the function $G(x,y,z)=f(x,y)-z$. 
The line $\ell$ has $4$-point contact with $S$ at the 
origin iff the function $g(s)=(G\circ\ell)(s)$ 
has a $0$ of multiplicity $4$ at $s=0$.
We already have $g(0)=g'(0)=0$ because $\ell$ is tangent to $S$ at the origin.
(To simplify, we write $\g_1'$ and $\g_2'$ instead of $\g_1'(0)$ and $\g_2'(0)$;
and since the function $f$ and its partial derivatives are evaluated at 
$(s\g_1',s\g_2')$, we omit to note this point.) 
The function $g$ is given by
$$g(s)=s(f_x\g_1'+f_y\g_2')-f.$$
We have therefore 
$$
\begin{array}{rcc}
g''(s) & = & s(f_{xxx}(\g_1')^3+3f_{xxy}(\g_1')^2\g_2'+ 3f_{xyy}\g_1'(\g_2')^2+f_{yyy}(\g_2')^3) + \\
 & & +f_{xx}(\g_1')^2+2f_{xy}\g_1'\g_2'+f_{yy}(\g_2')^2.
\end{array}$$
So $g''(0)=0$ iff the vector $(\g_1',\g_2')$ satisfies the 
equation 
$$(f_{xx}(\g_1')^2+2f_{xy}\g_1'\g_2'+f_{yy}(\g_2')^2)|_{s=0}=0.$$
$$g'''(s)=f_{xxx}(\g_1')^3+3f_{xxy}(\g_1')^2\g_2'+ 3f_{xyy}\g_1'(\g_2')^2+f_{yyy}(\g_2')^3
+s\cdot(\mathrm{something}).$$
Thus, $g'''(0)=0$ iff the vector $(\g_1',\g_2')$ satisfies the 
equation 
$$(f_{xxx}(\g_1')^3+3f_{xxy}(\g_1')^2\g_2'+ 
3f_{xyy}\g_1'(\g_2')^2+f_{yyy}(\g_2')^3)|_{s=0}=0.$$
But the curve of inflections of the IDE $a^f(x,y,p)=0$ is given by the equations 
$$f_{xx}+2f_{xy}p+f_{yy}p^2=0\, \ \ \mbox{ and } \ \ \,I^{a^f}=a_x^f+pa_y^f=0$$ 
and the last equation is written as \ 
$\,f_{xxx}+3f_{xxy}p+3f_{xyy}p^2+f_{yyy}p^3=0.$ 
\end{proof}

\begin{proposition}\label{prop:flacnodal-space-inflections}
The curve formed by the inflections of the asymptotic curves of a smooth surface 
coincides with the flecnodal curve of that surface. 
\end{proposition}

\begin{proof}
Let $\g$ be a regularly parametrised left (or right) asymptotic curve of $S$ such that $\g(0)=(0,0,0)$. 
Along our asymptotic curve $\g$ we have 
\begin{equation}\label{eq:asymptotic-diff-equation}
  f_{xx}(\g_1')^2++2f_{xy}\g_1'\g_2'+f_{yy}(\g_2')^2\equiv 0
\end{equation}

Derivating \eqref{eq:asymptotic-diff-equation} we get that the following equation holds along $\g$ 
\begin{equation}\label{eq:flecnode-inflection}
\begin{array}{rc}
 \left(f_{xxx}(\g_1')^3+3f_{xxy}(\g_1')^2\g_2'+ 3f_{xyy}\g_1'(\g_2')^2+f_{yyy}(\g_2')^3\right)&   \\
+ 2\left(f_{xx}\g_1'\g_1''+f_{xy}(\g_1'\g_2''+\g_1''\g_2')+f_{yy}\g_2'\g_2''\right)&= \, \, 0.
\end{array}  
\end{equation}
The origin, $\g(0)$, belongs to the left flecnodal curve iff the asymptotic tangent line $\ell(s)=\g(0)+s\g'(0)$ 
has at least $4$-point contact with $S$ at $\g(0)$. According to the proof of 
Proposition\,\ref{flecnodal-inflections}, this $4$-point contact is attained iff 
\begin{equation}\label{eq:asymptotic-direction-vector}
  f_{xx}(\g_1')^2++2f_{xy}\g_1'\g_2'+f_{yy}(\g_2')^2=0\quad \mbox{and} 
\end{equation}

$f_{xxx}(\g_1')^3+3f_{xxy}(\g_1')^2\g_2'+ 3f_{xyy}\g_1'(\g_2')^2+f_{yyy}(\g_2')^3=0$ at the origin. 
\medskip

Equation \eqref{eq:asymptotic-direction-vector} is satisfied because $\g$ is an asymptotic curve. 
Thus, the last equation holds at the origin iff  
the second term of the left-hand side of \eqref{eq:flecnode-inflection} vanishes at the origin. 
This last condition holds iff $\g''(0)$ is a multiple $\g'(0)$, that is, iff 
$\g(0)$ is an inflection of our asymptotic curve $\g$. 
\end{proof}

\section{Proofs of the Theorems on Evolving Surfaces}\label{sect:proofs-of-theorems-on-surfaces}

\subsubsection*{On the proofs of Theorem\,\ref{th:bigodron&flecgodron} and of No Cusp Theorem}\label{subsection:no-cusp-th} 
Items $a)$ and $b)$ of Theorem\,\ref{th:bigodron&flecgodron} coincide with items $a)$ and $b)$ of Theorem\,\ref{special-tangent} 
and of Theorem\,\ref{th:bifurcation-withPsmooth} (and were discussed in \cite{Uribegodron}).  
It remains to prove that, in the case of surfaces in $3$-space, item $c)$ of Theorem\,\ref{special-tangent} never occurs: 
\smallskip

\noindent
\textit{The flecnodal curve of a smooth surface in $\RP^3$ $($or in $\R^3$$)$ has never a cusp
at a godron at which the parabolic curve is smooth}. (No Cusp Theorem.)  









\begin{proof}[Proof of No Cusp Theorem]
Assume the surface has a godron $g$ at the point $z_0=f(x_0,y_0)$ with smooth parabolic curve.
Since its corresponding point $(x_0,y_0,p_0)\in\A^f\subset J^1(\R,\R)$ 
is a characteristic point of the surface $\A^f$, the point $(x_0,y_0)$ is a folded singular
point of the IDE $a^f=0$ with smooth discriminant.
Therefore $(x_0,y_0,p_0)$ is not a critical point of the function $a^f$. 


The curve of inflections of the IDE $a^f=0$ is the image under the map $\pi:(x,y,p)\mapsto (x,y)$ 
of the $\pi^\vee$-contour of $\A^f$. So, the curve of inflections has a cusp at $(x_0,y_0)$ 
iff the tangent line to the $\pi^\vee$-contour of $\A^f$ at  $(x_0,y_0,p_0)$ is vertical.
But since the $\pi^\vee$-contour of $\A^f$ is the intersection of $\A^f$ with the flec-surface of $a^f$,
this intersection has vertical tangent line at a point only if the flec-surface of 
$a^f$ has vertical tangent plane at that point. 
By Lemma\,\ref{Lemma G}, this is impossible for a non-critical point of the function $a^f$. 
\end{proof}

\subsubsection*{Proof of Theorem\,\ref{A_3-bifurcations} (On $A_3$-transitions)}\label{asymptote-IDE}
The conditions that guarantee the surface of equation $F=0$  has a ``\textit{Morse conic point in general position}''
(\S\,\ref{section:cone-fibrations}) are stated in Example\,\ref{quadratic-cone}. 
For $F=a^f$, these conditions are equivalent to the genericity conditions of Th.\,\ref{A_3-bifurcations}, 
stated in Note\,\ref{note:conic-points} in terms of the Monge form $z=f(x,y)$ at $q=(0,0;0)$. 

\begin{proposition}
If the function $a^f$ has a Morse critical point, then the quadratic cone $K$ determined by
the equation $a^f=0$ cannot have positions $e$ and $f$ with respect to the fibrations $\pi$ and $\pi^\vee$
{\rm (Fig.\ref{conos})}.
Moreover,  the tangent lines to the fibres of $\pi$ and $\pi^\vee$ are conjugate diameters of $K$. 
\end{proposition}

\begin{proof}
Lemma \ref{Lemma D} implies that $a^f_{xp}(\bar{0})=0$. By Example\,\ref{quadratic-cone} 
the intersection of the cone with the plane $y=0$ consists of the two lines given 
by the equation \,$a^f_{xx}(\bar{0})x^2+a^f_{pp}(\bar{0})p^2=0$.\,
Therefore the positions $e$ and $f$ of the cone are impossible (see Fig.\ref{conos}) 
and the cross-ratio of these lines with the $x$- and $p$-axes (the directions of the $\pi^\vee$- and $\pi$-fibres) 
is evidently $-1$. 
\end{proof}

So, for surfaces, only the positions a,b,c,d of the cone $K$ in Fig.\,\ref{conos} can take place. 
\medskip

\noindent
{\bf Proof of Theorem \ref{A_3-bifurcations}}. \ 
A generic family of smooth surfaces $\{z=f_\e(x,y)\}$ defines the family
of asymptotic IDEs $\{a^{f_\e}=0\}$.
By the stability of the Morse critical points, the functions $a^{f_\e}$ have a Morse critical point 
near the origin; but we can suppose that that critical point is the origin 
(by a translation and a rotation in the $xy$-plane, depending on $\e$).

%

Then, by Lemma\,\ref{Lemma F} and Lemma\,\ref{Lemma C}, the $2$-jet (in the variables $x,y,p$) of our family $\F^f$ of 
asymptotic IDE can be written in the form 
$$j^2\F^f (x,y,p;\e)=Ax^2+2Bxy+Cy^2+Dp^2+2Gyp-\E,$$
where \ $\,A=a+a_1\e+\cdots,\,$ \ $B=b+b_1\e+\cdots,\,$ \ $C=c+c_1\e+\cdots,\,$ \  
$E=e+e_1\e+\cdots,\,$ \ $G=g+g_1\e+\cdots,\,$ \ and \ $\,\E=\e+\mathrm{h.o.t.}(\e)$.

The first approximation in $\e$ is given by the family if IDEs
$$\Psi^f(x,y,p;\e)=a^f-\e=ax^2+2bxy+cy^2+dp^2+2gyp-\e=0.$$

The fact that $a^f(\bar 0)=0$ and that $\bar{0}$ is a critical point of $a^f(x,y,p)$ 
implies that at $(x,y)=(0,0)$ we have
$$f_{20}=f_{11}=f_{30}=f_{21}=0\,.$$
So the part of the Taylor expansion of $f$ determining the cone of $a^f=0$ is 
\begin{equation}\label{Taylor-expansion-A3}
f(x,y)=\frac{ax^4}{12}+\frac{bx^3y}{3}+\frac{cx^2y^2}{2}+\frac{gxy^2}{2}+\frac{dy^2}{2}+\cdots \,.
\end{equation}

With a scaling we make $d=1$. The condition $f_{40}(f_{22}-2f_{12}^2)-f_{31}^2\neq 0$ means 
$a(c-g^2)-b^2\neq 0$. 
We can make $g=0$ by a projective transformation (using the fact that $xy^2$ is a multiple of the 
quadratic part of $f$). Thus the above first approximation in $\e$ of the family of 
asymptotic IDEs becomes
\begin{equation}\label{reduced-cone-equation}
  a^f(x,y,p)-\e=ax^2+2bxy+cy^2+p^2-\e=0\,. 
\end{equation} 
Then the flec-surface is given by the equation
$$\,I^{a^f}/2=(ax+by)+(cy+bx)p=0.$$

Making the change of variables $\,Y=ax+by$, $\,X=bx+cy$ (which is possible only if $\Delta=ac-b^2\neq 0$) 
we get the equalities 
$$x=(bX-cY)/\Delta, \qquad \,y=(bY-aX)/\Delta \qquad \mbox{and} \qquad p=-Y/X\,.$$

Inserting them in \eqref{reduced-cone-equation} and puting $\a=-a/\Delta$, we get the equation 
\begin{equation}\label{simple-umbrella}
  Y^2=\e X^2+\a X^4-(c/\Delta)X^2Y^2+2(b/\Delta)X^3Y\,,
\end{equation}
of a Whitney umbrella (elliptic if $\a>0$ and hyperbolic if $\a<0$) whose isochronal section $\e=0$ 
is tangent to the umbrella at the pinch-point. 
%

For $g\neq 0$ we also get a Whitney umbrella, as expected, whose isochronal plane
section $\e=0$ is diffeomorphic to two tangent parabolas given by the equation
$(3g^2a^2+\Delta a)Y^4+4g^3a^2XY^2+\Delta^2X^2=0$, with $\Delta\neq ag^2$.
\smallskip

\noindent
\textbf{Stability}. The preceding calculations show that for each $\e$ the flecnodal curve of 
the IDE $a^f-\e = 0$ is the vertical projection of the intersection curve of the flec-surface 
of $a^f$ with the surface $\mathcal{A}_\e$ (of the IDE $a^f-\e=0$). 
The flec-surface of $a^f$ is thus foliated by curves labelled by $\e$ whose vertical projection 
is the flecnodal curve of the corresponding IDE $a^f-\e= 0$. 
This defines a map from the flec-surface of $a^f$ to the plane-time 3-space $(x,y,\e)$, whose image 
is the Whitney umbrella (without its handle), swiped by the flecnodal curves of this family. 
This map singularity is stable \cite{Whitney3} (see Appendix\,\ref{pinch-point}).
\smallskip

\noindent
{\bf Parabolic Curve}. The parabolic curve is provided by the Hessian curve, $(f_{20}f_{02}-f_{11}^2)(x,y)=0$.
At the transition moment it is given by the equation: 
\begin{equation}
(ax^2+2bxy+cy^2)(cx^2+gx+1)-(bx^2+2cxy+gy)^2=0\,.
\end{equation}
Then (for $\Delta\neq ag^2$) the parabolic curve is diffeomorphic, near the origin, to the two transverse
lines (may be imaginary) given by the equation
$$ax^2+2bxy+(c-g^2)y^2=0\,. \eqno{\square}$$

\subsubsection*{Proof of Theorem\,\ref{theorem:lips_bec-a-bec} (Lips, bec-\`a-bec, swallowtail)}\label{asymptote-IDE}
\noindent 
\textbf{Proof of Theorem\,\ref{theorem:lips_bec-a-bec}\,$(a)$}. 
For generic $1$-parameter families of surfaces in $\RP^3$, the corresponding pairs of surfaces 
$a^f=0$ and $I^f=0$ (in $J^1(\R,\R)$) may have one point of tangency for isolated parameter values. 
At such a point the differentials of $a^f$ and $I^f$ are proportional; and at $\bar{0}$ 
this proportionality is equivalent to the equality $3f_{21}^2-2f_{31}f_{11}=0$ because $f_{30}=f_{40}=0$. 

If we consider these surfaces as graphs of functions $p=p_a(x,y)$ and $p=p_I(x,y)$, 
the projection of their intersection to the $xy$-plane consists of the zeroes of the 
difference $p_a(x,y)-p_I(x,y)$, whose Taylor series is given by 
$Ax^2+2Bxy+Cy^2+\cdots$, where $A=f_{50}f_{11}$, $\,B=f_{41}f_{11}-2f_{31}f_{21}$ and 
$\,C=3f_{31}f_{12}-\frac{9}{2}f_{22}f_{21}+f_{32}f_{11}$. 

The statements of Theorem\,\ref{theorem:lips_bec-a-bec}\,$(a)$ follow from this fact.  
\medskip

\noindent 
\textbf{Proof of Theorem\,\ref{theorem:lips_bec-a-bec}\,$(b)$}. 
Since $f_{30}=f_{40}=0$, the pair of surfaces given by the equations $a^f(x,y,p)=0$ and $I^f(x,y,p)=0$ 
are smooth and transverse at $\bar{0}$ because $3f_{21}^2-2f_{31}f_{11}\neq 0$ 
(just compute the differentials). 
If we consider these surfaces as graphs of functions $p=p_a(x,y)$ and $p=p_I(x,y)$, 
then the projection of their intersection to the $xy$-plane (that is, the flecnodal curve) 
is given by the zeroes of the difference $p_a(x,y)-p_I(x,y)$. 
A direct computation shows that the Taylor series of this difference has the form
$$\left(\dfrac{f_{31}}{3f_{21}}-\frac{f_{21}}{2f_{11}}\right)y+2\b xy+\g y^2+\frac{f_{60}}{3f_{21}}x^3+\cdots,$$  
which implies the left (right) flecnodal curve has second order tangency with the left (resp. right) 
asymptotic line (for any value of $\b$ and $\g$). Then we have a double biflecnode, which can disappear or 
split into two simple biflecnodes. 

This proves Theorem\,\ref{theorem:lips_bec-a-bec}\,$(b)$.

\subsubsection*{On the proof of Theorem\,\ref{theorem:flat-umbilics} (On $D_4$-transitions)} 

We shall study the local transition at $\e=0$ of the flecnodal and parabolic curve of 
the surfaces $S_\e$ given by the following $1$-parameter family of functions\,: 
\begin{equation}\label{flat-umbilic-deformation}
f_\e(x,y)=\frac{\e}{2}(x^2\pm y^2)+\frac{1}{2}(x^2y\pm y^3).
\end{equation}

\noindent
(A generic family of functions would contain terms of degree $\geq 4$, which will only break 
the symmetries -  the terms of the form $\a x^3y$ and $\b y^4$ will ban the local presence of biflecnodes 
as we have seen at the end of \S\,\ref{section:bifurcation-theorems} 
 - see also Fig.\,\ref{ricD4+}.) 

The parabolic curves of this family of surfaces 
are the conics of equation
\[\pm (3y+2\e)^2-3x^2=\pm\e^2\,,\] 
(hyperbolas in the case $D_4^+$\, and ellipses in the case $D_4^-$) 
which undergo respectively a hyperbolic and an elliptic cone section transition.  

To describe the flecnodal curves, we use the equations $a^{f_\e}=0$, $I^{f_\e}=0$, that is 
\[(y+\e)+2xp\pm(\e+3y)p^2=0\,, \qquad 3p(1\pm p^2)=0\,.\]

\noindent
\textbf{Case $D_4^+$}. Insert the solutions of the equation $3p(1+p^2)=0$ ($p=0$ and $p=\pm i$) in 
the equation $(y+\e)+2xp+(\e+3y)p^2=0$. We respectively get the line of equation $y=-\e$, 
tangent to the parabolic curve at the godron $(0,-\e)$, and 
the complex conjugate lines $y=\pm ix$ whose real intersection point, the origin, is an ellipnode.  
\smallskip

\noindent
\textbf{Case $D_4^-$}. Insert the solutions of the equation $3p(1-p^2)=0$ 
in the equation $(y+\e)+2xp+(\e+3y)p^2=0$. We get the lines $y=-\e$ and $y=\pm x$,  
tangent to the parabolic curve at the respective godrons $(0,-\e)$ and $(\pm\e/2, -\e/2)$.  
These lines intersect at the three hyperbonodes $(0,0)$, $(\pm \e, -\e)$. 

Theorem\,\ref{theorem:flat-umbilics} follows from these facts because the flecnodal and parabolic curves, 
and the godrons, hyperbonodes and ellipnodes are stable. 
The introduction of terms of higher degree (to get a generic $1$-parameter family) would 
deform the flecnodal curve, which consists of three straight lines, into smooth curves, 
but would not change the configurations.




\begin{appendices}

\section{Examples of families of BIDEs}


\begin{example*}
Consider the family of implicit differential equations 
$$\{F_\e(x,y,p)=y-x^3-\e x-p^2=0\}.$$  
For a fixed value of $\e$, the criminant is provided by the 
system of equations $F_\e=0$ and $(\D F_\e/\D p)=0.$ So, the 
discriminant is given by the equation:
$$y=x^3+\e x.$$
By Lemma\,\ref{lemma 0} and Inflection-contour Theorem, the curve of inflections of the equation $F_\e=0$  
is the projection along the $p$-direction of the intersection of the 
surfaces given by the equations $F(x,y,p)=0$ and 
$I^F(x,y,p)=0$: 
$$y=x^3+\e x+(3x^2+\e)^2.$$
For each $x$, the coordinate 
$y$ of the curve of inflections is not less than the 
coordinate $y$ of the discriminant. For $\e>0$ these curves have 
no intersection point. For $\e<0$ they have two points of tangency, 
i.e. two folded singular points. 
For $\e=0$ these two folded singular points collapse into a folded singular 
point with multiplicity 2. 
\end{example*}

\begin{example*}[multiple folded singular point]
Consider the implicit differential equation 
$$F(x,y,p)=y-f(x)-p^2=0,$$
where $f:\R \rightarrow \R$ is a smooth function. 
Its discriminant curve and its curve of inflections are defined 
by the respective equations
$$y=f(x),\ \mbox{ and }\ y=f(x)+(f'(x))^2.$$
If $f$ is a polynomial of degree $k+1$ then 
$f(x)+(f'(x))^2$ has degree $2k$. Moreover, if $f$ has an 
$A_k$ singularity at the origin then $f(x)+(f'(x))^2$ has also an $A_k$ 
singularity at the origin. In particular, if $f(x)=x^{k+1}$, the 
multiplicity of intersection at the origin of
the criminant with the curve of inflections is $2k$, i.e. we have a 
folded singular point with multiplicity $k$. 
\end{example*}



\begin{example*}[perestroikas $\g v$ and $d$, by P.\,Pushkar]\label{petya}
Take the family of BIDEs
\[F_\e(x,y,p)=x^2+xp+p^2+\frac{1}{2}(y-1)^2-\e=0\, \qquad  \mbox{(with $\e\in\R$)}\,.\]
whose surfaces $\V^{F_\e}$ are ellipsoids for $\e>0$.

Combining the equations $F_\e(x,y,p)=0$ and $I^{F_\e}(x,y,p)=0$, 
one obtains the curve of inflections of the BIDE
$F_\e(x,y,p)=0$: 
$$4x^2+(\frac{1}{2}-\e)y^2-y^3-2x^2y+x^2y^2+\frac{1}{2}y^4=0.$$
The discriminant of the BIDE $F_\e(x,y,p)=0$ is the ellipse of equation
$$\frac{3}{4}x^2+\frac{1}{2}(y-1)^2=\e.$$
In Fig.\ref{petya1}, the curve of inflections and the discriminant curve are depicted 
for $\e=1, 1/2, 2/9\ \mathrm{and}\ 0$. At $\e=1/2$, the 
curve of inflections undergoes a $\g v$-perestroika; at $\e=0$, we have the 
transition $d$ of Fig.\ref{inflections}. 
\end{example*}

\begin{figure}[ht]
\centerline{\psfig{figure=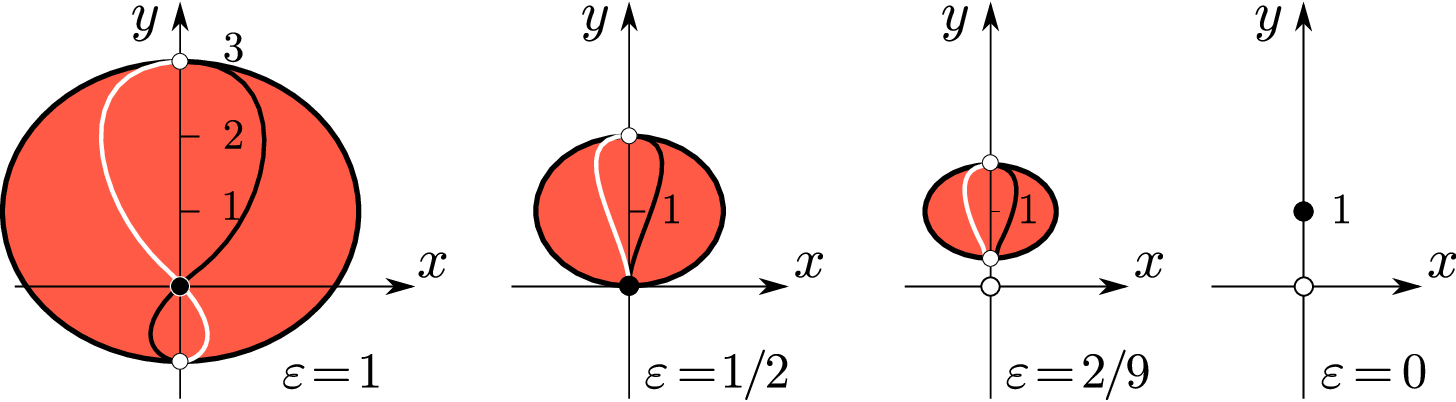,height=2.8cm}}
\caption{\small Discriminant curves and curve of inflections of P.\,Pushkar's Example.}
\label{petya1}
\end{figure}


\section{The Whitney Umbrella}\label{pinch-point}

The {\em standard Whitney umbrella} is the germ at the origin of the ``surface'',
in $3$-space, given by the equation $x^2=zy^2$ (Fig.\,\ref{Whitney}).
It intersects the planes $z=\mathrm{const}>0$ in pairs of lines $x^2=ay^2$ and the planes 
$y=\mathrm{const}$ in parabolas $z=bx^2$. This ``surface'' contains the $z$-axis and 
has the form of an eccentric umbrella, whose handle is the negative $z$-axis. 
A {\em Whitney umbrella} is a germ of surface diffeomorphic to the standard Whitney umbrella. 
For example, the equation $y^2=zx^2+\a x^4$ determines the so-called \textit{elliptic} ($\a>0$) or 
\textit{hyperbolic} ($\a<0$) \textit{Whitney umbrella} (Fig.\,\ref{Whitney}). 

\begin{figure}[ht]
\centerline{\psfig{figure=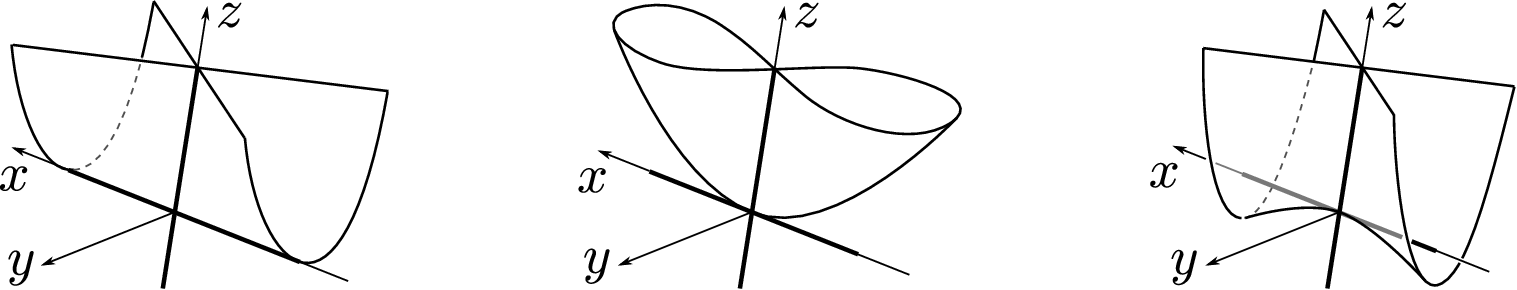,height=2.4cm}}
\caption{\small Standard/elliptic/hyperbolic Whitney umbrellas}
\label{Whitney}
\end{figure}

\noindent
\textit{\textbf{Whitney Singularity}}. 
The {\em Whitney singularity} of a map 
$\R^2 \rightarrow \R^3$ is the germ at zero of the map 
which, in suitable coordinates, is given by the formula $(u,v)\mapsto (uv,v,u^2)$. 
Its image, called \textit{cross-cap}, is the standard Whitney umbrella
without its handle. {\em This singularity is stable} (see  Whitney \cite{Whitney3}). 
The image of the map $(u,v)\mapsto(u,uv,v^2-\a u^2)$ is an elliptic ($\a>0$) or 
hyperbolic ($\a<0$) cross-cap. 

\begin{remark*}
A generic transverse section of the Whitney umbrella, through its pinch-point, 
is a curve having a semi-cubic cusp (Fig.\,\ref{umbrellasections}).
\end{remark*}

\noindent
\textit{\textbf{Single Tangent Line}}. \ 
At the pinch-point of a Whitney umbrella, there is a special tangent line:  
at each point of the self-intersection curve (the curve of double points) the 
surface has two transverse tangent planes (one for each local branch). Going along 
the self-intersection curve towards the pinch point, these planes tend to a single plane, 
and coincide with that plane at the pinch-point. In this plane there are two distinguished 
lines, one of them is the tangent line to the self-intersection curve and the other,  
called the \textit{single tangent line}, is the line tangent to the surface (the $y$-axis 
in Fig.~\ref{Whitney}). 
\medskip

\noindent
\textit{\textbf{Tangent Planes to the Umbrella}}.
If a map $\R^2\to\R^3$ has the Whitney singularity at a point, then the image of its differential 
(at that point) is the single tangent line to the umbrella image. 
So, a plane is said to be \textit{tangent to the umbrella at its pinch point} iff 
that plane contains the single tangent line. 
\end{appendices}



{\small
\begin{thebibliography}{99}

\bibitem {Arnoldcancon} {\bf Arnol'd V.I.}, {\em Critical points of 
smooth functions} in: Proceedings 
of the International Congress of Mathematicians (Vancouver, 1974) vol. 1, 
1975, 19-39. 

\bibitem {Arnoldwfeeml} {\bf Arnol'd V.I.}, {\em Wave front evolution and 
equivariant Morse lemma}, Commun. Pure Appl. Math. 29 (1976), 557-582. 


\bibitem {arnoldcs} {\bf Arnold V.I.}, {\em Geometrical Methods in the Theory of Ordinary
Differential Equations}, Spriger-Verlag, New York Heidelberg Berlin, 
1983. (Russian version: Nauka, 1978.) 

\bibitem {Arnoldspr} {\bf Arnold V.I.}, {\em Indices of singular points 
of $1$-forms on manifolds with boundary, convolution of invariants of 
groups generated by reflections, and singular projections of smooth 
hypersurfaces}, Russian Math. Surveys {\bf 34}:2 (1979) p.1-42. 

\bibitem {avg} {\bf Arnol'd V.I., Varchenko A.N., Gussein-Zade S.M.},
  {\em Singularities of Differentiable Maps}, Vol. 1, Birkh{\"a}user (1986). 
  
\bibitem{Arnold-Suf-Hyperbolic-eqs} \textbf{Arnold V.I.}, \textit{Surfaces Defined by Hyperbolic Equations}, 
Matem. Zametki \textbf{44}:1 (1988) p.3-18. English tranls:  Math. Notes Acad. of Sci. of the USSR \textbf{44}:1 (1988) p.489-497.





\bibitem {BT95} {\bf Bruce J.W., Tari F.}, {\em On binary differential equations}, 
Nonlinearity {\bf 8} (1995) p.255-271. 

\bibitem {BT97} {\bf Bruce J.W., Tari F.}, 
{\em Generic $1$-parameter families of binary differential equations of Morse type}, 
Discrete Contin. Dyn. Syst. {\bf 3} (1997) p.79-90. 

\bibitem {BT} {\bf Bruce J.W., Tari F.}, {\em Duality and 
implicit differential equations}, Nonlinearity {\bf 13} (2000) p.791-811. 

\bibitem {BFT2000} {\bf Bruce J.W., Fletcher G.J., Tari F.}, {\em Bifurcations of 
implicit differential equations}, Proc. R. Soc. Edinbourgh {\bf 130}A (2000) p.485-506. 

\bibitem {Cibrario} {\bf Cibrario M.}, {\em Sulla reduzione a forma delle equationi lineari 
alle derivate parziale di secondo ordine di tipo misto}, Accad. Sci. Lett., Inst. Lomb. Rend. {\bf 65} (1932) p.889-906. 

\bibitem {Dara} {\bf Dara L.}, {\em Singularit\'es g\'en\'eriques des \'equations diff\'erentielles multiformes}, 
Bol. Soc. Bras. Mat. {\bf 6} (1975) p.95-128. 

\bibitem {Davidov} {\bf Davydov A.A.}, {\em Normal form of a differential 
equation, not solvable with for the derivative, in a neighbourhood of a 
singular point}, Funct. Anal. Appl. {\bf 19}:2 (1985) 81-89. 

\bibitem {Davidov-Rosales} {\bf Davydov A.A.}, \textbf{Rosales-Gonzalez E.}, {\em Smooth normal forms of 
folded elementary singular points}, Jour. Dyn. Control Syst. {\bf 1} (1995) 463-482. 

\bibitem {Davidov-Ishi-Izu} {\bf Davydov A.A.}, \textbf{Ishikawa G.}, \textbf{Izumiya S.}, \textbf{Sun W-S.}, 
{\em Generic singularities of implicit systems of first order differential equations on the plane}, 
Japanese J. Math. {\bf 3} (2008) 93-119.


\bibitem {Docarmo} {\bf Do Carmo M.}, {\em Riemannian Geometry}, Birkhauser (1994) 4th Ed.

\bibitem {goryunov83} {\bf Goryunov V.V.}, {\em Singularities of projections of 
complete intersections}, J. Sov. Math. {\bf 27} (1984) 2785-2811.   
  
\bibitem {Toru2} \textbf{Kabata Y.}, \textbf{Deolindo Silva J.}, \textbf{Ohmoto T.},
 \textit{Binary differential equations at parabolic and umbilical pointsfor 2-parameter families of surfaces},
 Topology and its Applications \textbf{234} (2018) 457-473.
.  

\bibitem {Kazarian-Uribe} \textbf{Kazarian M.}, \textbf{Uribe-Vargas R.},
  \textit{Characteristic Points, Fundamental Cubic Form and Euler Characteristic 
  of Projective Surfaces}. Mosc. Math. J. \textbf{20}:3 (2020) 511-530. 

\bibitem {k-Thom} {\bf Kergosien Y.L., Thom R.}, {\em Sur les points 
paraboliques des surfaces}, C.~R.~Acad. Sci. Paris S\'er. A-B, {\bf 299} (1980) 705-710
  
\bibitem {Kortewegpp} {\bf Korteweg D.J.}, {\em Sur les points de plissement}, 
Arch. N\'eerl. {\bf 24} (1891) 57-98. 

\bibitem {Korteweggtp} {\bf Korteweg D.J.}, {\em La th\'eorie g\'en\'erale des plis}, 
Arch. N\'eerl. {\bf 24} (1891) 295-368. 

\bibitem {Landis} {\bf Landis E.E.}, {\em Tangential singularities},
Funct. Anal. Appl. {\bf 15}:2 (1981), 103--114. 

\bibitem {Levelt} {\bf Levelt Sengers J.}, {\em How fluids unmix: 
\small Discoveries by the School of Van der Waals and Kamerlingh Onnes}, KNAW, 
Amsterdam 2002. 


\bibitem {Dima} {\bf Panov D.A.}, {\em Special Points of Surfaces in the 
Three-Dimensional Projective Space}, 
Funct. Anal. Appl. {\bf 34}:4 (2000) 276--287. 

\bibitem {Ovsienko-Tabachnikov} \textbf{Ovsienko V.}, \textbf{Tabachnikov S.}, 
  \textit{Hyperbolic Carathéodory Conjecture}. Proc. of Steklov Inst. of Mathematics,
  \textbf{258} (2007) 178-193.

\bibitem {Platonova} {\bf Platonova O.A.}, {\em Singularities of the mutual position 
of a surface and a line}, Russ. Math. Surv. {\bf 36}:1 (1981) 248--249. Zbl.458.14014.


\bibitem {Salmon} {\bf Salmon G.}, {\em A treatise in analytic geometry of three dimensions},
Chelsea Pub. 1927. 

\bibitem {Toru} \textbf{Sano H.}, \textbf{Kabata Y.}, \textbf{Deolindo Silva J.}, \textbf{Ohmoto T.},
  \textit{Classification of Jets of Surfaces in Projective $3$-Space Via Central Projection},
  Bull. Math. Braz. Soc. (2017), DOI 10.1007/s00574-017-0036-x. 

\bibitem {Uribetesis} {\bf Uribe-Vargas R.}, {\em Symplectic and Contact Singularities in 
Differential Geometry of curves and surfaces}, 
PhD. Thesis, Universit{\'e} Paris 7 (2001).

\bibitem {Uribegodron} {\bf Uribe-Vargas R.}, {\em A Projective Invariant for
Swallowtails and Godrons, and Global Theorems on the Flecnodal Curve}, 
Mosc. Math. J. {\bf 6}:4 (2006) 731-772.   

\bibitem {Uribeinvariant} \textbf{Uribe-Vargas R.},
{\em On Projective Umbilics: a Geometric Invariant and an Index}.
Journal of Singularities \textbf{17}, Worldwide Center of Mathematics, LLC (2018) 81-90. DOI 10.5427/jsing.2018.17e


\bibitem {Whitney3} {\bf Whitney H.}, {\em On singularities of mappings 
of Euclidean spaces I. Mappings of the plane into the plane}, 
Ann. Math., II. Ser. 62 (1955), 374--410.




\end{thebibliography}
}
\end{document}